x

\documentclass[12pt]{article}

\usepackage{amsmath}
\usepackage{amscd}
\usepackage{amsthm}
\usepackage{array}
\usepackage{amssymb}
\usepackage{enumerate}
\usepackage{amssymb}
\usepackage{latexsym}
\usepackage{tikz-cd}
\usepackage{xcolor}
\usepackage{lastpage}
\usepackage{hyperref}
\usepackage{mathtools}
\usepackage{mathrsfs}
\usepackage{caption}
\usepackage[shortlabels]{enumitem}

\usepackage{scalerel,stackengine}
\stackMath
\newcommand\widecheck[1]{%
\savestack{\tmpbox}{\stretchto{%
  \scaleto{%
    \scalerel*[\widthof{\ensuremath{#1}}]{\kern-.6pt\bigwedge\kern-.6pt}%
    {\rule[-\textheight/2]{1ex}{\textheight}}
  }{\textheight}%
}{0.5ex}}%
\stackon[1pt]{#1}{\scalebox{-1}{\tmpbox}}%
}

\newtheorem{theorem}{Theorem}[section]

\newtheorem{proposition}[theorem]{Proposition}
\newtheorem{corollary}[theorem]{Corollary}
\newtheorem{lemma}[theorem]{Lemma}

\newtheorem*{problem*}{Problem}
\newtheorem*{theorem*}{Theorem}
\newtheorem*{proposition*}{Proposition}
\newtheorem*{corollary*}{Corollary}
\newtheorem*{lemma*}{Lemma}
\newtheorem*{claim*}{Claim}
\newtheorem*{solution*}{Solution}

\theoremstyle{definition}
\newtheorem{example}[theorem]{Example}
\newtheorem{notation}[theorem]{Notation}
\newtheorem{definition}[theorem]{Definition}
\newtheorem{remark}[theorem]{Remark}
\newtheorem*{remark*}{Remark}
\newtheorem*{summary*}{Summary}
\newtheorem*{example*}{Example}
\newtheorem{claim}[theorem]{Claim}

\newcommand{\I}{{\mathbb{I}}}
\newcommand{\N}{{\mathbb{N}}}
\newcommand{\R}{{\mathbb{R}}}
\renewcommand{\P}{{\mathbb{P}}}
\newcommand{\Z}{{\mathbb{Z}}}
\newcommand{\C}{{\mathbb{C}}}
\newcommand{\Q}{{\mathbb{Q}}}
\newcommand{\T}{{\mathbb{T}}}

\newcommand{\Oc}{{\mathbb O}}
\newcommand{\D}{{\mathbb D}}

\newcommand{\SU}{\text{SU}}
\newcommand{\SO}{\text{SO}}
\newcommand{\U}{\text{U}}

\newcommand{\Sp}{\text{Sp}}
\newcommand{\FS}{\text{FS}}
\newcommand{\CZ}{\text{CZ}}
\newcommand{\fp}{\mathfrak{p}}
\newcommand{\fP}{\mathfrak{P}}

\newcommand{\xddots}{%
  \raise 4pt \hbox {.}
  \mkern 6mu
  \raise 1pt \hbox {.}
  \mkern 6mu
  \raise -2pt \hbox {.}
}

\begin{document}
\pagenumbering{arabic} \setcounter{page}{2}

\renewcommand{\title}{\hspace{.3cm}Cylindrical contact homology of links of simple singularities}
\renewcommand{\author}{Leonardo Digiosia \footnote{This work was supported by NSF grants \href{https://www.nsf.gov/awardsearch/showAward?AWD_ID=1745670&HistoricalAwards=false}{DMS-1745670}, \href{https://www.nsf.gov/awardsearch/showAward?AWD_ID=1840723&HistoricalAwards=false}{DMS-1840723}, and \href{https://www.nsf.gov/awardsearch/showAward?AWD_ID=2104411&HistoricalAwards=false}{DMS-2104411}.}}


 \thispagestyle{empty}
  {\centering \Large
\title \par}
{\centering by \par}  {\centering \large \author \par}

 {\centering \section*{Abstract}}
 \addcontentsline{toc}{section}{Abstract}
A contact structure on a manifold is a maximally non-integrable hyperplane field in the tangent bundle. Contact manifolds are generalizations of constant energy hypersurfaces in Hamiltonian phase space, and have many applications to classical mechanics and dynamical systems. Cylindrical contact homology is a homology theory associated to a contact manifold, whose chain complex is generated by periodic trajectories of a Reeb vector field, and whose differential counts pseudoholomorphic cylinders between generating orbits. We compute the cylindrical contact homology of the links of the simple singularities. These 3-manifolds are contactomorphic to $S^3/G$ for finite subgroups $G\subset\SU(2)$. We perturb the degenerate contact form on $S^3/G$ with a Morse function, invariant under the corresponding $H\subset\SO(3)$ symmetries of $S^2$, to achieve nondegeneracy up to an action threshold. The cylindrical contact homology is recovered by taking a direct limit of the action filtered homology groups. The ranks of this homology are given in terms of the number of irreducible representations of $G$, namely $|\text{Conj}(G)|$, demonstrating a dynamical form of the McKay correspondence. Furthermore, the homology encodes information regrading the orbifold base $S^2/H$ over which $S^3/G$ fibers.

\newpage

\thispagestyle{empty} {\centering \section*{Acknowledgments} }
 \addcontentsline{toc}{section}{Acknowledgments}
An elder graduate student once suggested that a Cantor function acts as an appropriate model of progress in a PhD program; advancement feels stalled almost all of the time, save a handful of bursts of mathematical insight. Peers, mentors, and loved ones make the uneventful periods bearable, and without their support, the dissertation cannot arrive. Thus, some acknowledgments are in order:

First, to the person I owe everything -- thank you Mom. You have given me not only the ultimate gift of life, but one that is both full and stable. You trained me to be a responsible and hardworking man. Second, to my sister Natalia: thank you for the perspective that comes from sharing a childhood. That my best friend is also my sibling is not something to take for granted.

To Tatia Totorica, my first calculus teacher, thank you. The impact you made on me as a role model during an impressionable time of life eclipses the value of the powerful mathematical tools you gave me. You are my guiding example of an effective instructor when I teach calculus classes of my own.

I owe the majority of my development as a mature mathematician to my first mathematical parents, Iva Stavrov and Paul Allen. I am part of a flock of students that can call themselves empathetic and curious thinkers after working under your wings. Both of you will find influences of your teaching strewn throughout this work, including my stubborn unwillingness to denote the gradient of $f$ by $\nabla f$. 

To Frank Morgan, a mathematical parent and chosen family member, thank you. You are a father to many. I will always consider your guidance when making the big kinds of decisions. I also give gratitude for my mathematical and chosen sister, Lea Kenigsberg. If it were not for you, I would not have a strong case for the notion of fate. I do know deeply that our paths were meant to cross, and I think brightly on the future of our friendship and where it will take us.

To all of the graduate students in the basement of Herman Brown Hall, where would I be without your support? Firstly, to my algebraic geometer office mates Austen James and Zac Spaulding -- thank you for your patience in my perpetual demands that we take the ground field to be $\C$ in our discussions. Thank you Giorgio Young for looking over my Sobolev spaces, and for the laughs. Thank you Alex Nolte for your willingness to chat about math and more. To those unnamed, thank you.

The most important thanks goes to my advisor, Jo Nelson. Thank you for teaching me the most beautiful branch of mathematics that I have encountered. I hold a deep appreciation for Floer theory and how it fits into the field of mathematics as a whole. I am proud to have contributed to this branch of math and to have made it more  accessible for students to come - certainly, I could not have made this contribution without your patience and direction.

Finally, I must thank Tina Fey and include my favorite quote from the 2004 film \emph{Mean Girls}, which has guided me in a surprising number and variety of situations -  \begin{center}
``All you can do in life is try to solve the problem in front of you." 
\end{center} \begin{center}
- Cady Heron
\end{center}

\newpage
\tableofcontents

\newpage

\section{Introduction}\label{section: introduction}

The scope of this project  combines a wide palate of distinct mathematical fields. In this dissertation, our studies range from surface singularities and varieties in algebraic geometry, to Fredholm theory and Sobolev spaces in analysis; from trees and diagrams in combinatorics, to flows and critical point theory in differential geometry; from characteristic classes and covering spaces in algebraic topology, to representations  and characters in finite group theory; from classical mechanics and Hamiltonian systems in dynamics, to the groups and their associated algebras in Lie theory.

Each of these topics and objects plays a role in either contextualizing or proving Theorem \ref{theorem: main}; given any \emph{link of a simple singularity}, $L$, our result produces the \emph{cylindrical contact homology of $L$}, a graded $\Q$-vector space associated to $L$, denoted $CH_*(L)$. This infinite dimensional space is the homology of a chain complex whose generators are periodic trajectories of a Hamiltonian system, and whose differential counts cylindrical solutions to a geometric partial differential equation.

This cylindrical contact homology falls under the wide umbrella of  \emph{Floer theory}. Loosely speaking, Floer theories are (co)-homology theories whose chain complexes are generated by solutions to differential equations (critical points of functions, fixed points of diffeomorphisms, or periodic trajectories of a geometric flow, for example), and the differentials count notions of flow lines between these generating objects. Examples of Floer theories include Hamiltonian-Floer homology (\cite{F1}, \cite{F2}, \cite{F3}, \cite{S}), embedded contact homology (\cite{H}), symplectic homology (\cite{BO}, \cite{FH}, \cite{FHW}, \cite{Si}, \cite{V}), and symplectic field theory (\cite{EGH}, \cite{BEHWZ},  \cite{W3}). These theories have been used to prove existence results in dynamics - for example, by establishing lower bounds on the number of fixed points of a Hamiltonian symplectomorphism, or the number of closed embedded Reeb orbits in contact manifolds.

In Section \ref{section: introduction}, we first review some algebraic geometry necessary to explore the geometric objects of study in this paper, and to see how the result fits into a classical story via the McKay correspondence. The rest of the section is devoted to outlining the general construction of cylindrical contact homology, and the proof of Theorem \ref{theorem: main}. Section \ref{section: background: orbifold Morse homology and Conley-Zehnder indices} establishes background for the homology computed by Theorem \ref{theorem: main}, and is meant to provide the reader with a more complete picture of cylindrical contact homology and its motivations. Then, Sections \ref{section: geometric setup}, \ref{section: computation of filtered contact homology}, and \ref{section: direct limits of filtered homology} fill in the technical details of the proof of the main theorem.

\subsection{McKay correspondence and links of simple singularities}

One major way in which algebraic geometry diverges from its differential counterpart is that the study permits its geometric objects to exhibit singular points. In the study of 2-dimensional complex varieties, there is a notion of a \emph{simple singularity}. A simple singularity is locally modeled by the isolated singular point $[(0,0)]$ of the variety $\C^2/G$, for a finite nontrivial subgroup $G\subset\SU(2)$. The natural action of $G$ on $\C[u,v]$ has an invariant subring, generated by three monomials, $m_i(u,v)$ for $i=1,2,3$. These $m_i$ satisfy a minimal polynomial relation, \[f_G(m_1(u,v),m_2(u,v),m_3(u,v))=0,\] 
for some nonzero $f_G\in\C[z_1,z_2,z_3]$, whose complex differential has an isolated 0 at the origin. The polynomial $f_G$ provides an alternative perspective of the simple singularity as a hypersurface in $\C^3$.  Specifically, the map \[\C^2/G\to V:=f^{-1}_G(0),\,\,\,\,\,[(u,v)]\mapsto(m_1(u,v),m_2(u,v),m_3(u,v))\] defines an isomorphism of complex varieties, $\C^2/G\simeq V$. In this manner, we produce a hypersurface singularity given any finite nontrivial $G\subset \SU(2)$.

Conversely, we can recover the isomorphism class of $G$ from $V$ by studying the \emph{Dynkin diagram} associated to the minimal resolution of a simple singularity, using the \emph{McKay correspondence}. We detail this process now. Suppose a variety $V$ exhibits a simple singularity at the point $p\in V$. There exists a \emph{minimal resolution}, denoted $(\widetilde{V}, \pi, Z)$, of $(V, p)$. Explicitly, this means that $\widetilde{V}$ is a complex 2-dimensional variety, $\pi:\widetilde{V}\to V$ is birational,  $\pi^{-1}(p)=Z\subset\widetilde{V}$, $\widetilde{V}$ is smooth near $Z$, and $\pi$ away from $Z$ provides an isomorphism of varieties;\[\pi:\widetilde{V}\setminus Z\xrightarrow{\sim} V\setminus\{p\}.\] The resolution $( \widetilde{V},\pi, Z)$ is unique (or \emph{minimal}) in the sense that any other smoothing of $(V,p)$ factors through $\pi$.

We call $Z$ the \emph{exceptional divisor} of the resolution; $Z$ is the transverse intersection of a finite number of copies of $\C P^1\subset\widetilde{V}$, we enumerate these as $Z_1, Z_2, \dots, Z_n$. The \emph{Dynkin diagram} associated to $(V,p)$ is the finite graph whose vertex $v_i$ is labeled by the sphere $Z_i$, and $v_i$ is adjacent to $v_j$ if and only if $Z_i$ transversely intersects with $Z_j$. In this way, we we associate to any simple singularity $(V,p)$ a graph, $\Gamma(V,p)$. It is a classical fact that $\Gamma(V,p)$ is isomorphic to one of the $A_n$, $D_n$, or the $E_6$, $E_7$, or $E_8$ graphs (see \cite[\S 6]{Sl}), depicted in Figure \ref{figure: dynkin diagrams}. This provides the $A_n$, $D_n$, $E_6$, $E_7$, and $E_8$ singularities, alternatively known as the \emph{Klein}, \emph{Du Val}, or $\emph{ADE}$ singularities.

 \begin{figure}[h]
    \centering
    \captionsetup{justification=centering}
    \includegraphics[width=0.65\textwidth]{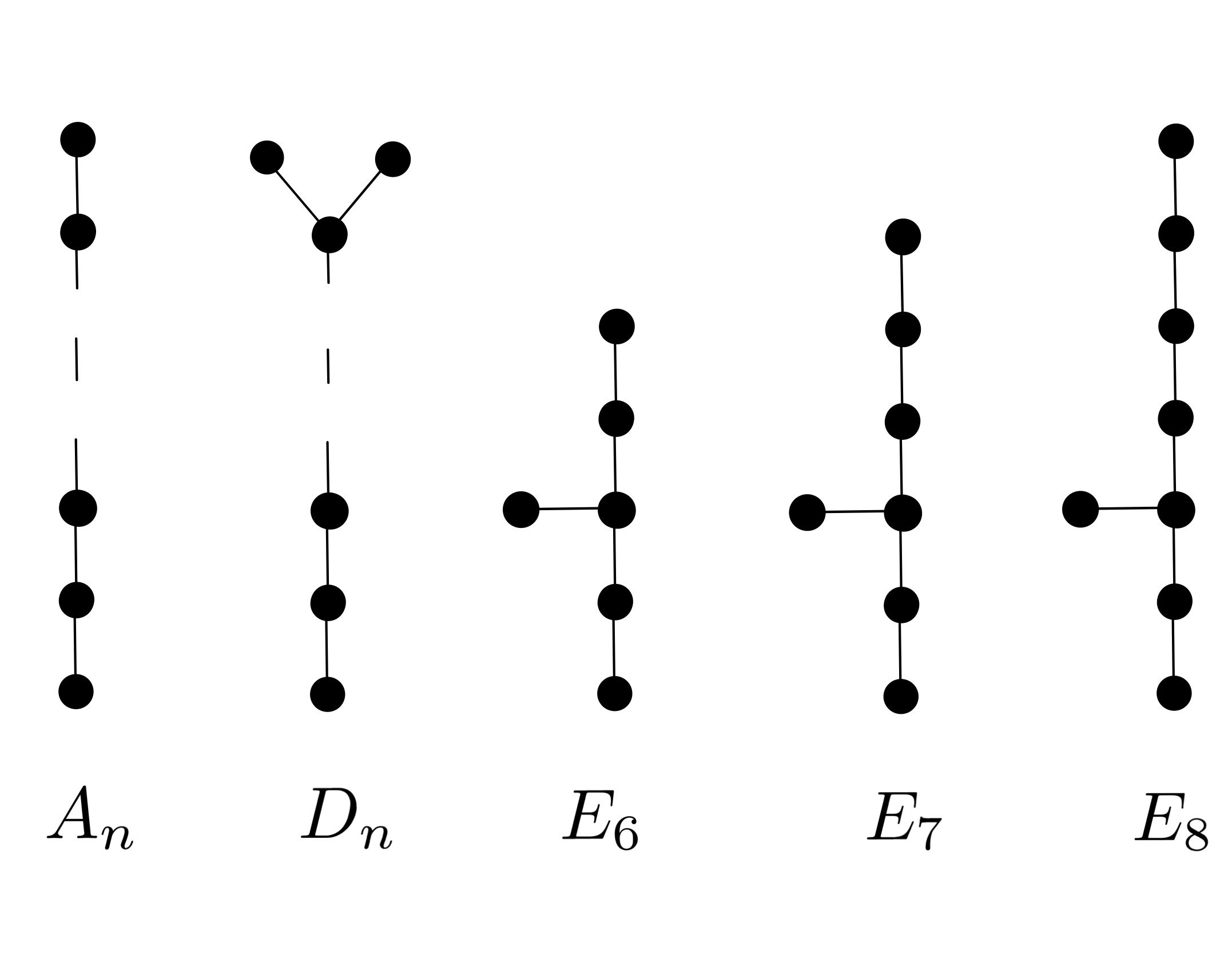}
    \caption{The Dykin diagrams. The diagrams $A_n$ and $D_n$ both feature $n$ nodes}
    \label{figure: dynkin diagrams}
\end{figure}

\textbf{Types of finite subgroups of $\SU(2)$:} These graphs simultaneously classify the types of conjugacy classes of finite subgroups $G$ of $\SU(2)$. To illuminate our discussion of these conjugacy classes, it is helpful to recall the double cover of Lie groups $P:\SU(2)\to\SO(3)$ (described in coordinates in \eqref{equation: P in coordinates}).
\begin{enumerate}
    \item A finite subgroup $G$ could be \emph{cyclic} of order $n$, written $G\cong\Z_n$. If $G$ is cyclic, then its image $H:=P(G)$ under $P$ is also cyclic. If $|G|$ is even, then $2|H|=|G|$, otherwise $|H|=|G|$. The natural action of $H$ on $S^2\subset\R^3$ is by rotations about some axis. The Dynkin diagram that will be associated to $G\cong\Z_n$ is $A_{n-1}$. 
    \item A finite subgroup $G$ could be conjugate to the \emph{binary dihedral group}, denoted $\D^*_{2n}\subset\SU(2)$, for some integer $n>1$. The binary dihedral group satisfies $|\D_{2n}^*|=4n$ and has $n+3$ conjugacy classes. We provide explicit matrix generators of $\D_{2n}^*$ in Section \ref{subsection: dihedral}, and descriptions of the $n+3$ conjugacy classes in Table \ref{table: Dihedral homotopy classes of Reeb orbits}. The group $\D_{2n}:=P(\D^*_{2n})\subset\SO(3)$ is the \emph{dihedral group}, and has $2n$ elements. Explicit matrix generators of $\D_{2n}$ are also provided in Section \ref{subsection: dihedral}. The natural $\D_{2n}$-action on $S^2\subset \R^3$ is given by symmetries of a regular $2n$-gon whose vertices are contained in a great circle of $S^2$. The Dynkin diagram that will be associated to $G\cong\D_{2n}^*$ is $D_{n+2}$.
    \item Finally a finite subgroup $G$ is \emph{binary polyhedral} if the image group $P(G)\subset\SO(3)$ can be identified with the symmetry group of a regular polyhedron, inscribed in $S^2$. This polyhedron is either a tetrahedron, an octahedron, or an icosahedron (the cube and dodecahedron can be omitted by duality), in which case the symmetry group $P(G)\subset\SO(3)$ is denoted $\T$, $\Oc$, or $\I$, and the group order of $P(G)$ is 12, 24, or 60\footnote{These facts may easily be checked by considering an geometric orbit-stabilizer argument.}. In each respective case, we write $G=\T^*$, $\Oc^*$, or $\I^*$; $|G|$  is 24, 48, or 120, and $G$ has 7, 8 , or 9 conjugacy classes. The Dynkin diagram that will be associated to $G\cong\T^*$, $\Oc^*$, or $\I^*$ is $E_6$, $E_7$, or $E_8$.
\end{enumerate}
It is a classical fact that any finite subgroup $G\subset\SU(2)$ must be either cyclic, conjugate to $\D_{2n}^*$, or is a binary polyhedral group (\cite[\S 1.6]{Z}). Associated to each  type of finite subgroup $G\subset\SU(2)$ is a finite graph, $\Gamma(G)$. The  vertices of $\Gamma(G)$ are in correspondence with the nontrivial irreducible representations of $G$ (of which there are $|\text{Conj}(G)|-1$, where $\text{Conj}(G)$ denotes the set of conjugacy classes of a group $G$), and adjacency is determined by coefficients of tensor products of representations. For more details regarding this construction of $\Gamma(G)$, consult \cite{St}. Again, $\Gamma(G)$ is isomorphic to one of the $A_n$, $D_n$, or the $E_6$, $E_7$, or $E_8$ graphs. Table \ref{table: diagrams of graphs} lists $\Gamma(G)$.

\begin{table}[h!]
\centering
 \begin{tabular}{||m{2cm} | m{2.5cm} | m{1.5cm} ||} 
 \hline
$G\subset\SU(2)$  & $|\text{Conj}(G)|-1$ & $\Gamma(G)$  \\ [0.5ex] 
 \hline\hline
 $\Z_n$ & $n-1$ & $A_{n-1}$  \\ 
 \hline
 $\D_{2n}^*$ & $n+2$ & $D_{n+2}$\\
 \hline
  $\T^*$ & $6$ & $E_6$\\
 \hline
   $\Oc^*$ & $7$ & $E_7$\\
 \hline
    $\I^*$ & $8$ & $E_8$\\
 \hline
\end{tabular}
\caption{Dykin diagrams associated to finite subgroups $G\subset\SU(2)$. See Section \ref{subsection: free homotopy classes represented by reeb orbits}} for more discussion regarding $\text{Conj}(G)$, the set of conjugacy classes of $G$.
\label{table: diagrams of graphs}
\end{table}

The McKay correspondence tells us that, given any variety $V$ exhibiting a simple singularity $p\in V$, and given any finite subgroup $G$ of $\SU(2)$, then $\Gamma(V,p)=\Gamma(G)$ if and only if a neighborhood of $p$ in $V$ is isomorphic to a neighborhood of $[(0,0)]$ in $\C^2/G$.

The \emph{link} of the singularity $V=f^{-1}_G(0)\subset\C^3$ is the 3-dimensional contact manifold $L:=S^5_{\epsilon}(0)\cap V$, with contact structure $\xi_L:=TL\cap J_{\C^3}(TL)$, where $J_{\C^3}$ is the standard integrable almost complex structure on $\C^3$, and $\epsilon>0$ is small. Analogously to the isomorphism of varieties $\C^2/G\cong V$, we have a contactomorphism $(S^3/G,\xi_G)\simeq(L,\xi_L)$, where $\xi_G$ on $S^3/G$ is the descent of the standard contact structure $\xi$ on $S^3$ to the quotient by the $G$-action. 

We adapt a method of computing the cylindrical contact homology of $(S^3/G,\xi_G)$ as a direct limit of filtered homology groups, described by Nelson in \cite{N2}. This process uses a Morse function, invariant under a symmetry group, to perturb a degenerate contact form. We prove a McKay correspondence result, that the ranks of the cylindrical contact homology of these links are given in terms of the number of conjugacy classes of the group $G$. Our computations agree with McLean and Ritter's work, which provides the positive $S^1$-equivariant symplectic cohomology of the crepant resolution of $\C^n/G$ similarly in terms of the number of conjugacy classes of the finite $G\subset\SU(n)$, \cite[Corollary 2.11]{MR}.

\subsection{Definitions and overview of cylindrical contact homology}\label{subsection: definitions}
Let $M$ be a $2n+1$ dimensional smooth manifold, and let $\xi\subset TM$ be a rank $2n$ sub-bundle of $TM$. We say that $\xi$ is a \emph{contact structure} on $M$ if, for any 1-form $\lambda$ that locally describes $\xi$\footnote{We say that $\lambda$ describes $\xi$ if, for all $p$ where $\lambda$ is defined, $\text{ker}(\lambda_p)=\xi_p$. Note that such a $\lambda$ may only be locally defined. There exists a global 1-form describing $\xi$ precisely when the quotient line bundle $TM/\xi\to M$ is trivial.}, \[\lambda\wedge(d\lambda)^n\in\Omega^{2n+1}(M)\] is nowhere vanishing. In this case, we say $(M,\xi)$ is a \emph{contact manifold}. If the contact structure $\xi$ is defined by a global 1-form $\lambda$, then we say that $\lambda$ is a \emph{contact form} on $M$ for $\xi$.

A contact distribution $\xi$ is \emph{maximally non-integrable}, meaning there does not anywhere exist a hypersurface $S$ in $M$ whose tangent space agrees with $\xi$ at every point of $S$. When $M$ is 3-dimensional, the condition that $\xi$ is a contact structure is equivalent to \[[X,Y]_p\notin\xi_p,\] for all $p\in M$ and all pointwise-linear independent sections $X$ and $Y$ of $\xi$ defined locally near $p$. If $\lambda$ is a contact form on $M$ defining contact distribution $\xi$, and if $f:M\to\R$ is a smooth, nowhere vanishing function on $M$, then $f\lambda$ is also a contact form on $M$ defining the same contact structure $\xi$. 

A \emph{contactomorphism} $f$ from contact manifold $(M_1,\xi_1)$ to contact manifold $(M_2,\xi_2)$ is a diffeomorphism $f:M_1\to M_2$ such that $f_*(\xi_1)=\xi_2$. That is, a contactomorphism takes vectors in the contact distribution of the first manifold to vectors in the contact distribution of the second manifold. If $\lambda_i$ is a contact form on $M_i$ defining the contact structure $\xi_i$, then we say that $f:M_1\to M_2$ is a \emph{strict} contactomorphism if $f^*\lambda_2=\lambda_1$. Strict contactomorphisms are always contactomorphisms.

\begin{example}
Take $M=\R^{2n+1}$ with coordinates $\{(z,x_1,y_1,\cdots, x_n,y_n)\}$. Set \[\lambda_0:=dz+\sum_{k=1}^nx_kdy_k.\] Now we have that \[d\lambda_0=\sum_{k=1}^ndx_k\wedge dy_k,\,\implies \lambda_0\wedge(d\lambda_0)^n=dz\wedge dx_1\wedge dy_1\wedge\cdots\wedge dx_n\wedge dy_n\neq0,\] and we see that $\lambda_0$ is a contact form on $M$. For any $z_0\in\R$, the translation \[f_{z_0}:M\to M,\,\,\,f_{z_0}(z,x_1,\dots,y_n)=(z+z_0,x_1,\dots, y_n)\] defines a strict contactomorphism of $M$.
\end{example}
Contact manifolds are the odd-dimensional siblings of \emph{symplectic manifolds}. A symplectic form $\omega$ on a  manifold $W$ is a closed, nondegenerate 2-form. If $W$ admits a symplectic form then we say that $(W,\omega)$ is a \emph{symplectic manifold}. A symplectic manifold $W$ is necessarily even dimensional, and we write $\text{dim}(W)=2n$. Furthermore, $\omega^n$ is a volume form on $W$. A diffeomorphism $f$ from symplectic manifold $(W_1,\omega_1)$ to symplectic manifold $(W_2,\omega_2)$ is a \emph{symplectomorphism} if $f^*\omega_2=\omega_1$.

\begin{example}
Take $W=\R^{2n}$ with coordinates $(x_1,y_1,\cdots,x_n,y_n)$. Set \[\omega_0:=\sum_{k=1}^ndx_k\wedge dy_k.\] Now we have that \[\omega_0^n=dx_1\wedge dy_1\wedge\cdots\wedge dx_n\wedge dy_n\neq0.\] Thus, $\omega_0$ is a symplectic form on $W$.
\end{example}

There are substantial interplays between symplectic and contact geometry;
\begin{enumerate}
    \item Given a $2n+1$-dimensional contact manifold $(M,\xi=\text{ker}(\lambda))$, the \emph{symplectization} of $M$ is the $2n+2$-dimensional manifold $\R\times M$ with symplectic form $\omega:=d(e^r\pi_M^*\lambda)$, where $\pi_M:W\to M$ is the canonical projection to $M$, and $r$ is the $\R$-coordinate.
    \item \emph{Contact-type hypersurfaces} in symplectic manifolds have induced contact structures, inherited from the ambient symplectic form. These hypersurfaces generalize the notion of constant energy level sets in phase space. For example, $S^3$ is a contact-type hypersurface in $(\R^4,\omega_0)$, which we define and work with in Section \ref{subsection: spherical geometry and associated Reeb dynamics}.
    \item A \emph{prequantization bundle} is a contact manifold $(Y,\lambda)$ presenting as a principal $S^1$-bundle over a symplectic manifold, \[\pi:(Y,\lambda)\to(V,\omega),\] satisfying $\pi^*\omega = -d\lambda$. For more details, see (\cite[Example 3.5.11]{MS}). The Hopf fibration,  $\fP:S^3\to S^2$, detailed in Section \ref{subsection: spherical geometry and associated Reeb dynamics}, is an example of a prequantization bundle. 
    \item On the linear level, if $\lambda$ is a contact form on $M$, then for any $p\in M$, $d\lambda|_{\xi_p}$ defines a \emph{linear symplectic form} on the even dimensional $\xi_p\subset T_pM$. Consequently, we say that $(\xi,d\lambda|_{\xi})$ is an example of a \emph{symplectic vector bundle} over $M$.
\end{enumerate}

In the interest of presenting an outline of cylindrical contact homology, let $(Y,\xi)$ be a closed contact 3-manifold with defining contact form $\lambda$. Most of the following objects and constructions described also exist in arbitrary odd dimensions, but the specialization to $2n+1=3$ simplifies some key parts. 

The contact form $\lambda$ on $Y$ determines a smooth vector field, $R_{\lambda}$, called the \emph{Reeb vector field}, which uniquely satisfies $\lambda(R_{\lambda})=1$ and $d\lambda(R_{\lambda},\cdot)=0$. For $t\in\R$, let $\phi^t:Y\to Y$ denote the time $t$ Reeb flow. Explicitly, this is a 1-parameter family of diffeomorphisms of $Y$ satisfying \[\tfrac{d}{dt}\phi^t|_p=R_{\lambda}(\phi^t(p))\in T_{\phi^t(p)}Y, \,\,\,\text{for all}\,\, p\in\ Y.\] A $\emph{Reeb orbit}$ is a map  $\gamma:\R/T\Z\to Y$ with $\gamma'(t)=R_{\lambda}(\gamma(t))$, for some $T>0$. This positive quantity $T$ is called the \emph{action} (or sometimes \emph{period}) of $\gamma$, denoted $\mathcal{A}(\gamma)$, and can always be computed via an integral; \[T=\mathcal{A}(\gamma)=\int_{\gamma}\lambda.\] We consider Reeb orbits up to reparametrization, meaning that, given some fixed $s\in\R$, we do not distinguish between the Reeb orbit $t\mapsto \gamma(t)$ and the Reeb orbit $t\mapsto\gamma(t+s)$, for $t\in\R/T\Z$. Let $\mathcal{P}(\lambda)$ denote the set of Reeb orbits of $\lambda$ up to reparametrization. If $\gamma\in\mathcal{P}(\lambda)$ and $m\in\N$, then the $m$-fold iterate of $\gamma$, a map $\gamma^m:\R/mT\Z\to Y$, is the precomposition of $\gamma$ with the $m$-fold cover, $\R/mT\Z\to\R/T\Z$, and is also an element of $\mathcal{P}(\lambda)$. The orbit $\gamma$ is \emph{embedded} when $\gamma:\R/T\Z\to Y$ is injective. If $\gamma$ is the $m$-fold iterate of an embedded Reeb orbit, then $m(\gamma):=m$ is the \emph{multiplicity} of $\gamma$.

The linearized return map of $\gamma\in\mathcal{P}(\lambda)$, denoted $P_{\gamma}$, is the restriction of the time $T$ linearized Reeb flow to a linear symplectomorphism of $(\xi_{\gamma(0)},d\lambda|_{\xi_{\gamma(0)}})$,  defined only after selection of a basepoint $\gamma(0)\in\text{im}(\gamma)$. That is, \[P_{\gamma}:=d\phi^T|_{\xi_{\gamma(0)}}:\xi_{\gamma(0)}\to \xi_{\gamma(0)}.\] Two linearized return maps corresponding to different choices of basepoint differ by conjugation, so the spectrum of $P_{\gamma}$ is independent of this choice. We say $\gamma$ is \emph{nondegenerate} if $1\notin\text{Spec}(P_{\gamma})$, and that  $\lambda$ is \emph{nondegenerate} if all $\gamma\in\mathcal{P}(\lambda)$ are nondegenerate. The product of eigenvalues of $P_{\gamma}$ equals 1: if both eigenvalues are unit complex numbers, then $\gamma$ is \emph{elliptic}, and if both eigenvalues are positive (negative) real numbers, then $\gamma$ is \emph{positive (negative) hyperbolic}. An orbit $\gamma\in\mathcal{P}(\lambda)$ is \emph{bad} if it is an even iterate of a negative hyperbolic orbit, otherwise, $\gamma$ is \emph{good}. In Section \ref{subsection: cylindrical contact homology as an analogue of orbifold Morse homology} we provide a bit of discussion as to which characteristics of bad Reeb orbits give them their name. Let $\mathcal{P}_{\text{good}}(\lambda)\subset\mathcal{P}(\lambda)$ denote the set of good Reeb orbits.

 For a nondegenerate contact form $\lambda$, define the  $\Q$-vector space generated by $\mathcal{P}_{\text{good}}(\lambda)$: \[CC_*(Y,\lambda):=\Q\langle\mathcal{P}_{\text{good}}(\lambda)\rangle.\] We equip this vector space with a grading using the \emph{Conley-Zehnder index}, which we describe now. Given $\gamma\in\mathcal{P}(\lambda)$, a \emph{symplectic trivialization} $\tau$ of $(\gamma^*\xi,d\lambda|_{\xi})$ is a diffeomorphism \[\tau:\gamma^*\xi\to \R/T\Z\times\R^2,\]
 that restricts to linear symplectomorphisms $(\xi_{\gamma(t)},d\lambda|_{\xi_{\gamma(0)}})\to(\{t\}\times\R^2,\omega_0),$ where $\omega_0$ is the standard linear symplectic form on $\R^2$. Symplectic trivializations of $(\gamma^*\xi,d\lambda|_{\xi})$ always exist. Now, given $\gamma\in\mathcal{P}(\lambda)$ and symplectic trivialization $\tau$, the \emph{Conley-Zehnder index}, $\mu_{\CZ}^{\tau}(\gamma)\in\Z$, is an integer that describes the rotation of the Reeb flow along $\gamma$ according to the trivialization, and depends only on the homotopy class of $\tau$. We develop multiple constructions and perspectives of the Conley-Zehnder index in Section \ref{subsection: the cz index}.
 
A bundle automorphism $J$ of a rank $2m$ symplectic vector bundle $(E,\omega)$ over a manifold $X$ is called a \emph{complex structure} on $E$ if $J^2=-\text{Id}$. The complex structure $J$ is $\omega$-\emph{compatible} if $\omega(\cdot, J\cdot)$ defines a $J$-invariant inner product on the fibers of $E$. The collection of $\omega$-compatible complex structures on $E$ is denoted $\mathcal{J}(\omega)$, and is non-empty and contractible (\cite[Proposition 2.6.4]{MS}). Any choice of $J\in\mathcal{J}(\omega)$ endows $E$ with the structure of a rank $m$ complex vector bundle, inducing Chern classes $c_i(E)\in H^{2i}(X;\Z)$, for $i\in\N$. Contractibility of $\mathcal{J}(\omega)$ ensures that these cohomology classes are independent in choice of $J$. Because $(\xi,d\lambda|_{\xi})$ defines a rank 2 symplectic vector bundle over $Y$, this discussion tells us that the first Chern class $c_1(\xi)\in H^2(Y;\Z)$ is well-defined. We say a complex structure $J$ on $\xi$ is $\lambda$-compatible if it is $d\lambda|_{\xi}$-compatible, and we set $\mathcal{J}(\lambda):=\mathcal{J}(d\lambda|_{\xi})$. We call $\xi$ a complex line bundle even without a $J$ specified.
 
 If $\langle c_1(\xi), \pi_2(Y)\rangle=0$ and if $\mu_{\CZ}^{\tau}(\gamma)\geq3$ for all contractible  $\gamma\in\mathcal{P}(\lambda)$, with any $\tau$ extendable over a disc, we say the nondegenerate contact form $\lambda$ is \emph{dynamically convex}. The complex line bundle $\xi$ admits a global trivialization if $c_1(\xi)=0$, which is unique up to homotopy\footnote{This  is due to the fact that two global trivializations of $\xi$ differ by a smooth map $Y\to\C^*$, which must nullhomotopic if the abelianization of $\pi_1(Y)$ is completely torsion. Furthermore, for a topological space $X$ with finitely generated integral homology groups, the integer $\beta_i(X)$ denotes the rank of the free part of $H_i(X;\Z).$} if $\beta_1(Y)=0$. In this case, the integral grading $|\gamma|$ of the generator $\gamma\in\mathcal{P}(\lambda)_{\text{good}}$ is defined to be $\mu_{\CZ}^{\tau}(\gamma)-1$ for any $\tau$ induced by a global trivialization of $\xi$. The links $(S^3/G,\xi_G)$ that we study in this paper all feature $c_1(\xi_G)=0$ and $\beta_1(S^3/G)=0$, and thus $\xi_G$ admits a global trivialization (explicitly described in Section \ref{section: geometric setup}), unique up to homotopy, inducing a canonical $\Z$-grading on the chain complex. 
 
Any  $J\in\mathcal{J}(\lambda)$ can be uniquely extended to an $\R$-invariant, $\omega=d(e^r\pi_Y^*\lambda)$-compatible almost complex structure on the symplectization $\R\times Y$, mapping $\partial_r\mapsto R_{\lambda}$, which we continue to call $J$. A smooth map \[u:\R\times S^1\to \R\times Y,\,\,(s,t)\mapsto u(s,t)\] is $J$-\emph{holomorphic} if $du\circ j=J\circ du$, where $j$ is the standard complex structure on the cylinder $\R\times S^1$, satisfying $j\partial_s=\partial_t$, and $S^1=\R/\Z$. The map $u$ is said to be positively (negatively) asymptotic to the Reeb orbit $\gamma$ if $\pi_Y\circ u(s,\cdot)$ converges uniformly to a parametrization of $\gamma$ as $s\to\infty$ ($s\to-\infty$) and if $\pi_{\R}\circ u(s,t)\to\pm\infty$ as $s\to\pm\infty$.

Given $\gamma_+$ and $\gamma_-$ in $\mathcal{P}(\lambda)$, $\mathcal{M}^J(\gamma_+,\gamma_-)$ denotes the set of $J$-holomorphic cylinders positively (negatively) asymptotic to $\gamma_+$ ($\gamma_-$), modulo holomorphic reparametrization of $(\R\times S^1, j)$. Note that $\R$ acts on $\mathcal{M}^J(\gamma_+,\gamma_-)$ by translation on the target. If $m\in\N$ and $u\in\mathcal{M}^J(\gamma_+,\gamma_-)$, we can precompose $u$ with the unbranched degree $m$-holomorphic covering map of cylinders, and we denote the composition by $u^m;$ \[u^m:\R\times S^1\xrightarrow{m\,:\,1}\R\times S^1\xrightarrow{u}\R\times Y.\]
This $m$-\emph{fold cover} $u^m$ is $J$-holomorphic and is an element of $\mathcal{M}^J(\gamma_+^m,\gamma_-^m)$. 

A $J$-holomorphic cylinder $u$ is \emph{somewhere injective} if there exists a point $z$ in the domain for which $u^{-1}(u(z))=\{z\}$ and $du_z\neq0$. Every $J$-holomorphic cylinder is the $m$-fold cover of a somewhere injective $J$-holomorphic cylinder $v$, written $u=v^m$ (\cite[Theorem 3.7]{N1}). Define the \emph{multiplicity} of $u\in\mathcal{M}^J(\gamma_+,\gamma_-)$, $m(u)$, to be this integer $m$. It is true that $m(u)$ divides both $m(\gamma_{\pm})$ (a fact that can be seen by considering the number of preimages of a single point in the image of the Reeb orbits), but it need not be the case that $\text{GCD}(m(\gamma_+),m(\gamma_-))=m(u)$, as discussed in Remark \ref{remark: multiplicity of u is not always gcd}.

Take a cylinder $u\in \mathcal{M}^J(\gamma_+,\gamma_-)$, and symplectically trivialize $\gamma_{\pm}^*\xi$ via $\tau_{\pm}$. Let $\tau$ denote the trivialization over the disjoint union of Reeb orbits. The \emph{Fredholm index of} $u$ is an integer associated to the cylinder which is related to the dimension of the moduli space in which $u$ appears. Computationally, this index may be expressed as \[\text{ind}(u)=\mu_{\CZ}^{\tau_+}(\gamma_+)-\mu_{\CZ}^{\tau_-}(\gamma_-)+2c_1(u^*\xi,\tau),\] where $c_1(u^*\xi,\tau)$ is a relative first Chern number which vanishes when $\tau$ extends to a trivialization of $u^*\xi$ (see \cite[\S 3.2]{H} for a full definition). This  term should be thought of as a correction term, determining how compatible the trivializations $\tau_{\pm}$ of $\gamma_{\pm}^*\xi$ are. That is, one could make the difference $\mu_{\CZ}^{\tau_+}(\gamma_+)-\mu_{\CZ}^{\tau_-}(\gamma_-)$ as arbitrarily large or small as desired, by choosing $\tau_{\pm}$ appropriately.\footnote{This is possible by the fact that Conley-Zehnder indices arising from two different choices of trivialization of $\xi$ along a single Reeb orbit differ by twice the degree of the associated overlap map (see \cite[Exercise 3.37]{W3}).} The value of  $c_1(u^*\xi,\tau)$ accommodates for this freedom so that $\text{ind}(u)$ is independent in the choice of trivializations. Of course, when $c_1(\xi)=0$, a global trivialization $\tau$ exists. For this $\tau$, and for any cylinder $u$, the integer $c_1(u^*\xi,\tau)$ is always 0, and the Fredholm index of $u$ is simply \[\text{ind}(u)=\mu_{\CZ}(\gamma_+)-\mu_{\CZ}(\gamma_-),\] where $\mu_{\CZ}(\gamma)$ is implicitly computed with respect to the global trivialization.

For  $k\in\Z$, $\mathcal{M}_k^J(\gamma_+,\gamma_-)\subset \mathcal{M}^J(\gamma_+,\gamma_-)$ denotes those cylinders with $\text{ind}(u)=k$. If $u\in \mathcal{M}_k^J(\gamma_+,\gamma_-)$ is somewhere injective and $J$ is generic, then  $\mathcal{M}^J_{k}(\gamma_+,\gamma_-)$ is a manifold of dimension $k$ near $u$ (see \cite[\S 8]{W2}). 

\begin{remark}
If $\text{GCD}(m(\gamma_+),m(\gamma_-))=1$, then every cylinder $u\in\mathcal{M}^J(\gamma_+,\gamma_-)$ must have $m(u)=1$, and particular, is somewhere injective, implying $\mathcal{M}^J_k(\gamma_+,\gamma_-)$ is a $k$-dimensional manifold.
\end{remark}

\begin{definition}\label{definition: trivial cylinders} (Trivial cylinders) Given a Reeb orbit $\gamma:\R/T\Z\to Y$ with period $T>0$, we define  the \emph{trivial cylinder} over $\gamma$ by \[u_{\gamma}:\R\times \R/\Z\to\R\times Y,\,\,(s,t)\mapsto (T\cdot s,\gamma(T\cdot t)).\]
Note that \[\frac{\partial u_{\gamma}}{\partial s}(s,t)=T\partial_r\in T_p(\R\times Y),\,\,\,\text{and}\,\,\,\frac{\partial u_{\gamma}}{\partial t}(s,t)=TR_{\lambda}\in T_p(\R\times Y),\] where $p:=u_{\gamma}(s,t)$, and $r$ is the $\R$-coordinate on the symplectization. Importantly, this provides that (for any choice of complex structure $J\in\mathcal{J}(\lambda)$) \[J\frac{\partial u_{\gamma}}{\partial s}=JT\partial_r=TJ\partial_r=TR_{\lambda}=\frac{\partial u_{\gamma}}{\partial t},\] implying that $J\circ du_{\gamma}=du_{\gamma}\circ j$. The cylinder $u_{\gamma}$ is called the \emph{trivial cylinder} associated to $\gamma$, and is asymptotic to $\gamma$ at $\pm\infty$. By appealing to Stokes' theorem and action of Reeb orbits, one concludes that any other parametrized $J$-holomorphic cylinder asymptotic to $\gamma$ at both ends is a holomorphic reparametrization of $u_{\gamma}$ so that $\mathcal{M}^J(\gamma,\gamma)$ is a singleton set. Note that the $\R$-action on $\mathcal{M}^J(\gamma,\gamma)$ is trivial and that the Fredholm index of $u_{\gamma}$ is 0.

\end{definition}

Suppose that for every pair of good Reeb orbits $\gamma_{\pm}$ that $\mathcal{M}_1^J(\gamma_+,\gamma_-)/\R$ is a compact, zero dimensional manifold. In this desired case, the differential \[\partial:CC_*(Y,\lambda)\to CC_{*-1}(Y,\lambda)\] between two generators $\gamma_{\pm}\in\mathcal{P}_{\text{good}}(\lambda)$ may be defined as the signed and weighted count of $\mathcal{M}_1^J(\gamma_+,\gamma_-)/\R$: \begin{equation}\label{equation: CCH differential} \langle\partial \gamma_+,\gamma_-\rangle=\sum_{u\in\mathcal{M}_1^J(\gamma_+,\gamma_-)/\R}\epsilon(u)\frac{m(\gamma_+)}{m(u)}\in\Z.\end{equation} Here, the values $\epsilon(u)\in\{\pm1\}$ are induced by a choice of coherent  orientations (see \cite{BM}). Recall that  $m(u)$ divides $m(\gamma_{\pm})$, and so this count is always an integer. In favorable circumstances, $\partial^2=0$ (for example, if $\lambda$ is dynamically convex, \cite[Theorem 1.3]{HN}),  and the resulting homology is denoted $CH_*(Y,\lambda,J)$. Under additional hypotheses, this homology is independent of contact form $\lambda$ defining $\xi$ and generic $J$ (for example, if $\lambda$ admits no contractible Reeb orbits, \cite[Corollary 1.10]{HN2}), and is denoted $CH_*(Y,\xi)$. This is the \emph{cylindrical contact homology} of $(Y,\xi)$. Upcoming work of Hutchings and Nelson will show that $CH_*(Y,\xi)=CH_*(Y,\lambda,J)$ is independent of dynamically convex $\lambda$ and generic $J$.

\subsection{Main result and connections to other work}

The link of the $A_n$ singularity is shown to be contactomorphic to the lens space $L(n+1,n)$ in \cite[Theorem 1.8]{AHNS}. More generally, 
  the links of simple singularities $(L,\xi_L)$ are shown to be contactomorphic to quotients $(S^3/G,\xi_G)$ in \cite[Theorem 5.3]{N3}, where $G$ is related to $L$ via the McKay correspondence.  Theorem \ref{theorem: main} computes the cylindrical contact homology of $(S^3/G,\xi_G)$ as a direct limit of action filtered homology groups.

\begin{theorem} \label{theorem: main}
    Let $G\subset\mbox{\em SU}(2)$ be a finite nontrivial group, and let $m=|\mbox{\em Conj}(G)|\in\N$, the number of conjugacy classes of $G$. The cylindrical contact homology of $(S^3/G, \xi_G)$ is \[CH_*(S^3/G,\xi_G):=\varinjlim_N CH_*^{L_N}(S^3/G,\lambda_N, J_N)\cong\bigoplus_{i\geq0}\Q^{m-2}[2i]\oplus\bigoplus_{i\geq0} H_*(S^2;\Q)[2i].\]
    \end{theorem}
     The directed system of action filtered cylindrical contact homology groups \[CH_*^{L_N}(S^3/G,\lambda_N,J_N),\,\,\,\text{for}\,\,N\in\N,\] is described in Section \ref{subsection: structure of proof of main theorem}. Upcoming work of Hutchings and Nelson will show that the direct limit of this system is an invariant of $(S^3/G,\xi_G)$, in the sense that it is isomorphic to $CH_*(S^3/G,\lambda,J)$  where $\lambda$ is any dynamically convex contact form on $S^3/G$ with kernel $\xi_G$, and $J\in\mathcal{J}(\lambda)$ is generic. 
     
     The brackets in Theorem \ref{theorem: main} describe the degree of the grading. For example, $\Q^8[5]\oplus H_*(S^2;\Q)[3]$ is a ten dimensional space with nine dimensions in degree 5, and one dimension in degree 3. By the classification of finite subgroups $G$ of $\SU(2)$ (see Table \ref{table: diagrams of graphs}), the following enumerates the possible values of $m=|\text{Cong}(G)|$: 
\begin{enumerate}
        \itemsep-.35em
       \item If $G$ is cyclic of order $n$, then $m=n$.
       \item If $G$ is binary dihedral, written  $G\cong\D^*_{2n}$ for some $n$, then $m=n+3$.
       \item If $G$ is binary tetrahedral, octahedral, or icosahedral, written $G\cong\T^*, \Oc^*$, or $\I^*$, then  $m=7, 8$, or $9$, respectively.
\end{enumerate}

\begin{remark}
Theorem \ref{theorem: main} can alternatively be expressed as
\[CH_*(S^3/G,\xi_{G})\cong\begin{cases} \Q^{m-1} & *=0,\\ \Q^{m} & *\geq2\,\, \mbox{and even} \\ 0 &\mbox{else.}\end{cases}\]
In this form, we can compare the cylindrical contact homology to the positive $S^1$-equivariant symplectic cohomology of the crepant resolutions $Y$ of the singularities $\C^n/G$ found by McLean and Ritter. Their work shows that these groups with $\Q$-coefficients are free $\Q[u]$-modules of rank equal to $m=|\text{Conj}(G)|$, where $G\subset\SU(n)$ and $u$ has degree 2 \cite[Corollary 2.11]{MR}. By \cite{BO}, it is expected that these Floer theories are isomorphic.
\end{remark}

The cylindrical contact homology in Theorem \ref{theorem: main} contains information of the manifold $S^3/G$ as a Seifert fiber space, whose $S^1$-action agrees with the Reeb flow of a contact form defining $\xi_G$. Viewing the manifold $S^3/G$ as a Seifert fiber space, which is an $S^1$-bundle over an orbifold surface homeomorphic to $S^2$ (see Diagram \ref{diagram: commuting square} and Figure \ref{figure: fundamental domain}), the copies of $H_*(S^2;\Q)$ appearing in Theorem \ref{theorem: main} may be understood  as the \emph{orbifold Morse homology}  of this orbifold base $S^2/H$ (see Section \ref{subsection: morse and orbifold morse theory} for the construction of this chain complex). Each orbifold point $p$ with isotropy order $n_p\in\N$ corresponds to an exceptional fiber, $\gamma_p$, in $S^3/G$, which may be realized as an embedded Reeb orbit. The generators of the $\Q^{m-2}[0]$ term are the iterates \[\gamma_p^k,\,\,\, \text{for $p$ an orbifold point, and }\,\, k=1,2,\dots, n_p-1,\] so that the dimension $m-2=\sum_p(n_p-1)$ of this summand can be regarded as a kind of total isotropy of the base. Note that all but finitely many terms $n_p-1$ are nonzero, for $p$ ranging in $S^2/H$. The subsequent terms $\Q^{m-2}[2i]$ are generated by a similar set of $m-2$ iterates of orbits $\gamma_p$.

\begin{remark}
Recent work of Haney and Mark  computes the cylindrical contact homology in \cite{HM} of a family of Brieskorn manifolds $\Sigma(p,q,r)$, for $p$, $q$, $r$ relatively prime positive integers satisfying $\frac{1}{p}+\frac{1}{q}+\frac{1}{r}<1$, using methods from \cite{N2}. Their work uses a family of \emph{hypertight} contact forms, whose Reeb orbits are non-contractible. These manifolds are also Seifert fiber spaces, whose cylindrical contact homology features summands arising from copies of the homology of the orbit space, as well as summands from the total isotropy of the orbifold. 
\end{remark}
\subsection{Structure of proof of main theorem}\label{subsection: structure of proof of main theorem}

We now outline the proof of Theorem \ref{theorem: main}.  Section \ref{section: geometric setup}  explains the process of perturbing a degenerate contact form $\lambda_G$ on $S^3/G$ using an orbifold Morse function. Given a finite, nontrivial subgroup $G\subset\SU(2)$, $H$ denotes the image of $G$ under the double cover of Lie groups $P:\SU(2)\to\SO(3)$. By Lemma \ref{lemma: commutes}, the quotient by the $S^1$-action on the Seifert fiber space $S^3/G$ may be identified with a map $\fp:S^3/G\to S^2/H$. This $\fp$ fits into a commuting square of topological spaces (Diagram \eqref{diagram: commuting square}) involving the Hopf fibration $\fP:S^3\to S^2$.

An $H$-invariant Morse-Smale function on $(S^2, \omega_{\FS}(\cdot,j\cdot))$ (constructed in Section \ref{subsection: constructing morse functions}) descends to an \emph{orbifold Morse function}, $f_H$, on $S^2/H$, in the language of \cite{CH}. Here, $\omega_{\FS}$ is the Fubini-Study form on $S^2\cong\C P^1$, and $j$ is the standard integrable complex structure. By Lemma \ref{lemma: reeb1}, the Reeb vector field of the $\varepsilon$-perturbed contact form $\lambda_{G,\varepsilon}:=(1+\varepsilon\fp^*f_H)\lambda_G$ on $S^3/G$ is the descent of the vector field
\[\frac{R_{\lambda}}{1+\varepsilon\fP^*f}-\varepsilon\frac{\widetilde{X_f}}{(1+\varepsilon\fP^*f)^2}\]
 to $S^3/G$. Here, $\widetilde{X_f}\in \xi\subset TS^3$ is a $\fP$-horizontal lift  of the Hamiltonian vector field $X_f$ of $f$ on $S^2$, computed with respect to $\omega_{\FS}$, and with the convention that $\iota_{X_f}\omega_{\FS}=-df$. Thus, the $\widetilde{X_f}$ term vanishes along exceptional fibers $\gamma_p$ of $S^3/G$ projecting to orbifold critical points $p\in S^2/H$ of $f_H$, implying that these parametrized circles and their iterates $\gamma_p^k$ are Reeb orbits of $\lambda_{G,\varepsilon}$. Lemma \ref{lemma: ActionThresholdLink} computes $\mu_{\CZ}(\gamma_p^k)$ in terms of $k$ and the Morse index of $f_H$ at $p$, with respect to a global trivialization $\tau_G$ of $\xi_G$ (see \eqref{equation: global trivialization}).

Given a contact 3-manifold $(Y,\lambda)$ and $L>0$, we let $\mathcal{P}^L(\lambda)\subset\mathcal{P}(\lambda)$ denote the set of orbits $\gamma$ with $\mathcal{A}(\gamma)=\int_{\gamma}\lambda<L$. A contact form $\lambda$ is $L$-\emph{nondegenerate} when all $\gamma\in\mathcal{P}^L(\lambda)$ are nondegenerate. If $\langle c_1(\xi),\pi_2(Y)\rangle=0$ and $\mu_{\CZ}^{\tau}(\gamma)\geq 3$, for all contractible $\gamma\in\mathcal{P}^L(\lambda)$, with any symplectic trivialization $\tau$ extendable over a disc, we say that the $L$-nondegenerate contact form $\lambda$ is $L$-\emph{dynamically convex}. By Lemma \ref{lemma: ActionThresholdLink}, given $L>0$, all  $\gamma\in\mathcal{P}^L(\lambda_{G,\varepsilon})$ are nondegenerate and project to critical points of $f_H$ under $\fp$, for $\varepsilon>0$ small. This lemma allows for the computation in Section \ref{section: computation of filtered contact homology} of the  \emph{action filtered} cylindrical contact homology, which we describe now. Recall from Section \ref{subsection: definitions}, for $\gamma_{\pm}\in\mathcal{P}_{\text{good}}(\lambda)$ with $\mu_{\CZ}(\gamma_+)-\mu_{\CZ}(\gamma_-)=1$, the differential $\partial$ of cylindrical contact homology is given by \begin{equation}\label{equation: boundary of CCH} \langle\partial \gamma_+,\gamma_-\rangle=\sum_{u\in\mathcal{M}_1^J(\gamma_+,\gamma_-)/\R}\epsilon(u)\frac{m(\gamma_+)}{m(u)}.\end{equation} If $\mathcal{A}(\gamma_+)<\mathcal{A}(\gamma_-)$, $\mathcal{M}^J(\gamma_+,\gamma_-)$ is empty (because action decreases along holomorphic cylinders in a symplectization, by Stokes' Theorem), ensuring that this sum is zero. Thus, after fixing $L>0$, $\partial$ restricts to a differential, $\partial^L$, on the subcomplex generated by  $\gamma\in\mathcal{P}_{\text{good}}^L(\lambda)$, denoted $CC_*^L(Y,\lambda)$,  whose homology is denoted $CH_*^L(Y,\lambda,J)$.

In Section \ref{section: computation of filtered contact homology}, we use Lemma \ref{lemma: ActionThresholdLink} to produce a  sequence $(L_N, \lambda_N, J_N)_{N=1}^{\infty}$, where $L_N\to\infty$ in $\R$, $\lambda_N$ is an $L_N$-dynamically convex contact form for $\xi_G$, and $J_N\in\mathcal{J}(\lambda_N)$ is generic. By Lemmas \ref{lemma: CZdihedral} and \ref{lemma: CZpolyhedral}, every orbit $\gamma\in\mathcal{P}^{L_N}_{\text{good}}(\lambda_N)$ is of even degree, and so $\partial^{L_N}=0$, providing
\begin{equation}\label{equation: Identify}
    CH_*^{L_N}(S^3/G,\lambda_N,J_N)\cong\Q\langle\,\mathcal{P}_{\text{good}}^{L_N}(\lambda_N)\,\rangle\cong\bigoplus_{i=0}^{2N-1}\Q^{m-2}[2i]\oplus\bigoplus_{i=0}^{2N-2} H_*(S^2;\Q)[2i].
\end{equation}

Finally, we prove Theorem \ref{theorem: cobordisms induce inclusions} in Section \ref{section: direct limits of filtered homology}, which states  that a completed symplectic  cobordism $(X,\lambda,J)$  from $(\lambda_N,J_N)$ to $(\lambda_M,J_M)$, for $N\leq M$, induces a  homomorphism,  \[\Psi:CH_*^{L_N}(S^3/G,\lambda_N,J_N)\to  CH_*^{L_M}(S^3/G,\lambda_M,J_M)\] which takes the form of the standard inclusion when making the identification \eqref{equation: Identify}. The proof of Theorem \ref{theorem: cobordisms induce inclusions} comes in two steps. First, the moduli spaces of index 0 cylinders in cobordisms $\mathcal{M}_0^J(\gamma_+,\gamma_-)$ are finite by Proposition \ref{proposition: buildings in cobordisms} and Corollary \ref{corollary: moduli spaces are finitie}, implying that the map $\Psi$ is well defined. Second, the identification of $\Psi$ with a standard inclusion is made precise in the following manner. Given $\gamma_+\in\mathcal{P}^{L_N}_{\text{good}}(\lambda_N)$, there is a unique $\gamma_-\in\mathcal{P}^{L_M}_{\text{good}}(\lambda_M)$ that (i) projects to the same critical point of $f_H$ as $\gamma_+$ under $\fp$, and (ii) satisfies $m(\gamma_+)=m(\gamma_-)$. When (i) and (ii) hold, we write $\gamma_+\sim\gamma_-$. We argue in Section \ref{section: direct limits of filtered homology} that $\Psi$ takes the form $\Psi([\gamma_+])=[\gamma_-]$, when $\gamma_+\sim\gamma_-$.

Theorem \ref{theorem: cobordisms induce inclusions} now implies that the system of filtered contact homology groups is identified with a sequence of inclusions of vector spaces, providing isomorphic direct limits: \[\varinjlim_N CH_*^{L_N}(S^3/G,\lambda_N, J_N)\cong\bigoplus_{i\geq0}\Q^{m-2}[2i]\oplus\bigoplus_{i\geq0}H_*(S^2;\Q).\]

\section{Background: orbifold Morse homology and Conley-Zehnder indices}\label{section: background: orbifold Morse homology and Conley-Zehnder indices}

In Section \ref{section: introduction}, we alluded to orbifold Morse theory and the Conley-Zehnder index of Reeb orbits. In this section, we provide background  for these notions and explain how they are related.

\subsection{Morse and orbifold Morse theory}\label{subsection: morse and orbifold morse theory}
We first review the ingredients and construction of the Morse chain complex for smooth manifolds, and then compare with its orbifold cousin. Let $(M,g)$ be a closed Riemannian manifold of dimension $n$, and take a smooth function $f:M\to \R$. The triple $(M,g,f)$ determines a vector field, the \emph{gradient} of $f$, \[\text{grad}(f)\in\Gamma(TM),\,\,\,g(\text{grad}(f),v)=df(v), \,\,\text{for}\,\,v\in TM.\]
Let $\text{Crit}(f):=\{p\in M\,|\, df_p=0\}$ denote the \emph{critical point set} of $f$. To each $p\in\text{Crit}(f)$, we associate a linear endomorphism, $H_p(f):T_pM\to T_pM$, the \emph{Hessian} of $f$ at $p$, whose eigenvalues provide information about the behavior of $f$ near $p$. We define the Hessian now.

The Riemannian metric provides the data of a \emph{connection} on the tangent bundle $TM$, written $\nabla$. Given any point $p\in M$ and a smooth vector field $V$, defined locally near $p$, the connection defines a linear map $T_pM\to T_p M$, written $v\mapsto \nabla_v V$, which satisfies a Leibnitz rule. The Hessian at the critical point $p\in\text{Crit}(f)$ is then simply defined in terms of this connection: \[H_p(f):T_pM\to T_pM \,:=\, v\mapsto \nabla_v \,\text{grad}(f).\]
The values of $H_p(f)$ do not depend on the Riemannian metric, because $p$ is a critical point of $f$. Many quantities in Morse theory do not depend on the (generic) choice of Riemannian metric $g$; the values of $H_p(f)$ are a first example of this independence.

We say that the critical point $p$ is \emph{nondegenerate} if $0$ is not an eigenvalue of $H_p(f)$. We say that $f$ is a \emph{Morse function} on $M$ if all of its critical points are nondegenerate. Note that the Morse condition is independent of $g$, that generic smooth functions on $M$ are Morse, and that a Morse function has a finite critical point set. All of these facts can be seen by the helpful interpretation of $\text{Crit}(f)$ as the intersection of the graph of $df$ with the zero section in $T^*M$; this intersection is transverse if and only if $f$ is Morse.

As an endomorphism of an inner product space, $H_p(f)$ is self-adjoint, and thus, its spectrum consists of real eigenvalues. Let \[\text{ind}_p(f)\in \{0,1,\dots, n=\text{dim}(M)\}\] denote the number of negative eigenvalues of $H_p(f)$ (counted with multiplicity). This is the \emph{Morse index} of $p$. If $f$ is Morse and $\text{ind}_f(p)=n$, then $p$ is a local maximum of $f$; if $f$ is Morse and $\text{ind}_f(p)=0$ then $p$ is a local minimmum of $f$. 

For the rest of this section, unless otherwise stated, we assume that $f$ is Morse. We now associate to each $p\in\text{Crit}(f)$ the stable and unstable manifolds of $p$, whose dimension depends on this integer $\text{ind}_f(p)$. For every $t\in\R$, we have the diffeomorphism of $M$, $\phi^t:M\to M$, given by the time $t$ flow of the smooth vector field $-\text{grad}(f)$. Note that this $\phi^t$ is defined for all $t\in\R$, because $M$ is closed. Given $p\in\text{Crit}(f)$, define the subsets \[W^{\pm}(p):=\{x\in M\,|\,\text{lim}_{t\to\pm\infty}\phi^t(x)=p\}.\]  We refer to $W^{-}(p)$ as the \emph{unstable submanifold} of $p$, and to $W^{+}(p)$ as the \emph{stable submanifold} of $p$. Indeed, these sets are manifolds; it is a classical fact of Morse theory that $W^{-}(p)$ is a smoothly embedded  $\text{ind}_f(p)$ disc in $M$, and $W^{+}(p)$ is a smoothly embedded $n-\text{ind}_f(p)$ disc.

Fix distinct critical points $p$, $q\in\text{Crit}(f)$. If $W^{-}(p)$ and $W^{+}(q)$ transversely intersect in $M$, then their intersection is an $\text{ind}_f(p)-\text{ind}_f(q)$ dimensional submanifold, admitting a free $\R$-action by flowing along $-\text{grad}(f)$. The set \[\mathcal{M}(p,q):=\big(W^{-}(p) \cap W^{+}(q)\big)/\R\]
inherits the topology of an $\text{ind}_f(p)-\text{ind}_f(q)-1$ dimensional manifold, whose points are identified with unparametrized negative gradient flow lines of $f$. This space $\mathcal{M}(p,q)$ is referred to as the \emph{moduli space of flow lines} from $p$ to $q$. 

We say that the pair $(f,g)$ is \emph{Smale} if $W^{-}(p)$ and $W^{+}(q)$ are transverse in $M$, for all choices $p$, $q\in\text{Crit}(f)$. Given a Morse function $f$ on $M$, a generic choice of Riemannian metric $g$ provides a Smale pair $(f,g)$. Assume that $(f,g)$ is a Morse-Smale pair.

An orientation of each disc $W^{-}(p)\subset M$, for $p$ ranging in  $\text{Crit}(f)$, induces an orientation on all $\mathcal{M}(p,q)$ as smooth manifolds (even if $M$ is not orientable). For more details, see \cite[Exercise 1.11]{S}. In particular, in the 0-dimensional case (that is, for $\text{ind}_f(p)-\text{ind}_f(q)=1$), we can associate to each $x\in\mathcal{M}(p,q)$ a sign, $\epsilon(x)\in\{\pm1\}$.

Additionally, the manifold $\mathcal{M}(p,r)$ admits a compactification by broken flow lines, providing $\overline{\mathcal{M}(p,r)}$, a manifold with corners. In the 1-dimensional case (that is, for $\text{ind}_f(p)-\text{ind}_f(r)=2$), $\overline{\mathcal{M}(p,r)}$ is an oriented 1-dimensional manifold with boundary, whose boundary consists of broken flow lines of the form $(x_+,x_-)$, for $x_+\in\mathcal{M}(p,q)$ and $x_-\in\mathcal{M}(q,r)$, where $q\in\text{Crit}(f)$. We can write this as 
\begin{equation}\label{equation: oritented boundary}
    \partial\mathcal{M}(p,r):=\overline{\mathcal{M}(p,r)}\setminus\mathcal{M}(p,r)=\bigcup_{q\in\text{Crit}(f)}\mathcal{M}(p,q)\times \mathcal{M}(q,r).
\end{equation}
This is an equality of \emph{oriented} 0-dimensional manifolds, in the sense that $(x_+,x_-)\in\partial\mathcal{M}(p,r)$ is a \emph{positive} end of $\overline{\mathcal{M}(p,r)}$ if and only if $\epsilon(x_+)\epsilon(x_-)=+1$.

We associate to the triple $(M,f,g)$ a chain complex, under the assumptions that $f$ is Morse and that $(f,g)$ is Smale. Select an orientation on each $W^{-}(p)$, for $p\in\text{Crit}(f)$. For $k\in\Z$,  define $CM_k$ to be the free abelian group generated by the set $\text{Crit}_k(f):=\{p\in\text{Crit}(f)\, |\,\text{ind}_f(p)=k\}$. For each $k\in\Z$, we define $\partial:CM_k\to CM_{k-1}$ to be the \emph{signed} count of negative gradient trajectories between critical points of index difference 1. Specifically, for $p\in\text{Crit}_k(f)$ and $q\in\text{Crit}_{k-1}(f)$, \[\langle\partial p,q\rangle:=\sum_{x\in\mathcal{M}(p,q)}\epsilon(x).\]
We explain why $\partial\circ\partial=0$. Select generators $p\in\text{Crit}_k(f)$ and $r\in\text{Crit}_{k-2}(f)$, and consider the value: \[\langle \partial^2 p, r\rangle=\sum_{q\in\text{Crit}(f)}\sum_{x_+\in\mathcal{M}(p,q)}\sum_{x_-\in\mathcal{M}(q,r)}\epsilon(x_+)\epsilon(x_-).\]
We argue that the above integer is always 0, for any choice of $p$ and $r$. Note that by equation \eqref{equation: oritented boundary}, this sum is precisely equal to the signed count of $\partial\mathcal{M}(p,r)$. It is known that the signed count of the oriented boundary of any compact, oriented 1-manifold with boundary is zero, implying that the above sum is zero. This demonstrates that $\big(CM_*,\partial\big)$ defines a chain complex over $\Z$. Denote the corresponding homology groups by $H^{\text{Morse}}_*(f,g)$.

The Morse homology is isomorphic to the classical singular homology of $M$ with integer coefficients; \[H^{\text{Morse}}_*(f,g)\cong H_*(M;\Z).\] This is because the pair $(f,g)$ determines a cellular decomposition of $M$; the set $\text{Crit}_k(f)$ is in correspondence with the $k$-cells. To each $p\in\text{Crit}_k(f)$, the associated $k$-cell is the embedded disc $W^{-}(p)$. In fact, at the level of chain complexes, $\big(CM_*,\partial\big)$ is isomorphic to this cellular chain complex associated to the pair $(f,g)$. For this reason, $H^{\text{Morse}}_*(f,g)$ does not depend on the pair $(f,g)$, and is a topological invariant of $M$. This invariance is used to prove the \emph{Morse inequalities}, that the number of critical points of a Morse function $f$ on a closed manifold $M$ is bounded below by the sum of the Betti numbers of $M$: \[|\text{Crit}(f)|\geq\sum_{i=0}^{\text{dim}(M)}\beta_i(M).\]

\textbf{Orbifold Morse homology.} We now extend the discussion of Morse data and the associated chain complex to \emph{orbifolds}. The reader should consult \cite{CH} for complete details regarding the orbifold-Morse-Smale-Witten complex. For simplicity, we take our orbifolds to be of the form $M/H$, where $(M,g)$ is a closed Riemannian manifold and $H$ is a finite group acting on $M$ by isometries. Let $f:M\to\R$ be a Morse function on $M$ which is invariant under the $H$-action. The quotient space $M/H$ inherits the structure of an \emph{orbifold}, and $f$ descends to an \emph{orbifold Morse function}, $f_H:M/H\to\R$. For motivating examples, the reader should consider $M=S^2$ with the round metric, and $H\subset\SO(3)$ is a finite group of dihedral or polyhedral symmetries of $S^2$. See Figures \ref{figure: dihedral}, \ref{figure: tetrahedral}, \ref{figure: octahedral}, and \ref{figure: icosahedral} for depictions of possible $H$-invariant Morse functions on $S^2$ for these cases.

As before, for each $p\in\text{Crit}(f)$, select an orientation of the embedded disc $W^{-}(p)$. Consider that the action of the stabilizer (or, \emph{isotropy}) subgroup of $p$, \[H_p:=\{h\in H\,|\, h\cdot p=p\}\subset H,\] on $M$ restricts to an action on $W^{-}(p)$ by diffeomorphisms. We say that the critical point $p$ is \emph{orientable} if this action is by \emph{orientation preserving} diffeomorphisms. Let $\text{Crit}^+(f)\subset\text{Crit}(f)$ denote the set of orientable critical points.

Note that the $H$-action on $M$ permutes $\text{Crit}(f)$, and the action restricts to a permutation of $\text{Crit}^+(f)$. Furthermore, the index of a critical point is preserved by the action. Let $\text{Crit}(f_H)$, $\text{Crit}^+(f_H)$, and $\text{Crit}^+_k(f_H)$ denote the quotients $\text{Crit}(f)/H$, $\text{Crit}^+(f)/H$, and $\text{Crit}^+_k(f)/H$, respectively. As in the smooth case, we define the $k^{\text{th}}$-orbifold Morse chain group, denoted $CM_k^{\text{orb}}$, to be the free abelian group generated by $\text{Crit}_k^+(f_H)$. The differential will be defined by a signed \emph{and weighted} count of negative gradient trajectories in $M/H$. The homology of this chain complex is, as in the smooth case, isomorphic to the singular homology of $M/H$ (\cite[Theorem 2.9]{CH}).

\begin{remark} \label{remark: nonorientable1}
(Discarding non-orientable critical points to recover singular homology) Consider that every index 1 critical point of $f:S^2\to\R$ depicted in Figures \ref{figure: dihedral},  \ref{figure: tetrahedral}, \ref{figure: octahedral}, and \ref{figure: icosahedral} is non-orientable. This is because the unstable submanifolds associated to each of these critical points is an open interval, and the action of the stabilizer of each such critical point is a 180-degree rotation of $S^2$ about an axis through the critical point. Thus, this action reverses the orientation of the embedded open intervals. If we were to include these index 1 critical points in the chain complex, then $CM^{\text{orb}}_*$ would have rank three, with \[CM^{\text{orb}}_0\cong CM^{\text{orb}}_1\cong CM^{\text{orb}}_2\cong\Z\]
(in Section \ref{subsection: visualizing holomorphic cylinders: an example} we fill in these details for $H=\T$, where $\text{Crit}(f_{\T})=\{\mathfrak{v},\mathfrak{e},\mathfrak{f}\}$). There is no differential on this (purported) chain complex with homology isomorphic to $H_*(S^2/H;\Z)\cong H_*(S^2;\Z)$. This last isomorphism holds because $S^2/H$ is a topological $S^2$ in these examples.\footnote{See Section \ref{subsection: visualizing holomorphic cylinders: an example} and Figure \ref{figure: fundamental domain} for an explanation of why $S^2/H$ is a topological $S^2$ with three orbifold points.} Indeed, the correct chain complex, obtained by discarding the non-orientable index 1 critical points, has rank two: \[CM_0^{\text{orb}}\cong CM_2^{\text{orb}}\cong\Z,\,\,\, CM_1^{\text{orb}}=0\]  and has completely vanishing differential, producing homology groups isomorphic to $H_*(S^2/H;\Z)\cong H_*(S^2;\Z)$.
\end{remark}

\begin{remark}\label{remark: nonorientable2}
(Discarding non-orientable critical points to orient the gradient trajectories) Let $p$ and $q$ be orientable critical points in $M$ with Morse index difference equal to 1. Let $x:\R\to M/H$ be a negative gradient trajectory of $f_H$ from $[p]$ to $[q]$. Because $p$ and $q$ are orientable, the value of $\epsilon(\widetilde{x})\in\{\pm1\}$ is independent of any choice of lift of $x$ to a negative gradient trajectory $\widetilde{x}:\R\to M$ of $f$ in $M$. We define $\epsilon(x)$ to be this value. Conversely, if one of the points $p$ or $q$ is non-orientable, then there exist two lifts of $x$ with opposite signs, making the choice $\epsilon(x)$ dependent on choice of lift.
\end{remark}

We now define $\partial^{\text{orb}}:CM_k^{\text{orb}}\to CM_{k-1}^{\text{orb}}$. Select $p\in\text{Crit}_k^+(f_H)$ and $q\in\text{Crit}_{k-1}^+(f_H)$ and let $\mathcal{M}(p,q)$ denote the set of negative gradient trajectories (modulo the $\R$ action by translations) in $M/H$ of $f_H$ from $p$ to $q$. In favorable circumstances, $\mathcal{M}(p,q)$ is a finite set of signed points. For a flow line $x\in\mathcal{M}(p,q)$, let $|H_x|$ denote the order of the isotropy subgroup $H_a\subset H$ for any choice of $a\in\text{im}(x)\subset M/H$. This quantity is independent in choice of $a$ and divides both group orders $|H_p|$ and $|H_q|$. We define the $\langle \partial^{\text{orb}} p, q\rangle$ coefficient to be the following \emph{signed and weighted} count of $\mathcal{M}(p,q)$: \[\langle \partial^{\text{orb}} p, q\rangle:=\sum_{x\in\mathcal{M}(p,q)}\epsilon(x)\frac{|H_p|}{|H_x|}\in\Z.\]  This differential should be directly compared to the differential of cylindrical contact homology, equation \eqref{equation: CCH differential}. The similarities between these differentials will be discussed more in Sections \ref{subsection: cylindrical contact homology as an analogue of orbifold Morse homology} and \ref{subsection: visualizing holomorphic cylinders: an example}.

\subsection{Asymptotic operators and spectral flows} \label{subsection: asymptotic operators and spectral flows}

In Morse theory and its orbifold version, the number of negative eigenvalues of the Hessian determines the integral grading of the chain complex. The analogue of the Hessian in cylindrical contact homology is the \emph{asymptotic operator} associated to a Reeb orbit. This asymptotic operator is a bounded  isomorphism between Hilbert spaces that is similar to the Hessian in that it is self adjoint, and its spectrum is a discrete subset of $\R$.

Let $(M,\lambda)$ be a closed contact manifold of dimension $2n+1$, and let $\gamma$ be a Reeb orbit. In this section for simplicity, we write $S^1=\R/\Z$ and understand $\gamma$ as a map $S^1\to\R$, so that $\tfrac{d}{dt}\gamma(t)=T
\cdot R(\gamma(t))$, where $T$ is the action of $\gamma$, and we use $t$ to denote values in $S^1$. The Reeb orbit provides the following symplectic pullback bundle over $S^1$: \[(\gamma^*\xi,d\lambda|_{\xi})\to S^1,\] where $\xi=\text{ker}(\lambda)$. Consider that the linearization of the Reeb flow along $\gamma$ induces an $\R$-action on $\gamma^*\xi$ that provides a parallel transport of the fibers by linear symplectomorphisms. This family of parallel transports defines a symplectic connection $\nabla$ on $(\gamma^*\xi,d\lambda|_{\xi})$. 

Suppose we have a complex structure $J$ on $\xi$ that is compatible with the symplectic form $d\lambda|_{\xi}$. Letting $\Gamma(\gamma^*\xi)$ denote the space of smooth sections of the bundle $\gamma^*\xi\to S^1$, we can define the \emph{asymptotic operator} associated to $
\gamma$, \[A_{\gamma}:\Gamma(\gamma^*\xi)\to\Gamma(\gamma^*\xi),\,\,\,\eta\mapsto-J\nabla_t\eta.\]
The motivation for the definition of the asymptotic operator is that it serves as the linearization of the positive gradient of a contact action functional on the loop space (i.e. the operator is a Hessian), a framework that we outline now.

Consider the loop space $\mathcal{L}(M)=C^{\infty}(S^1,M)$. We have a \emph{contact action functional}: \[\mathcal{A}_{\lambda}:\mathcal{L}(M)\to\R,\,\,\,\mathcal{A}_{\lambda}(\gamma)=\int_{\gamma}\lambda.\]  We want to take a sort of gradient of this functional. For $\gamma\in\mathcal{L}(M)$, $T_{\gamma}\mathcal{L}(M)=\Gamma(\gamma^*TM)$. Let $E\subset T\mathcal{L}(M)$ denote the sub-bundle whose fiber $E_{\gamma}$ over $\gamma\in\mathcal{L}(M)$ is $\Gamma(\gamma^*\xi)\subset \Gamma(\gamma^*TM)$. Define a section $\sigma$ of $E\to\mathcal{L}(M)$ by \[\sigma:\mathcal{L}(M)\to E,\,\,\,\sigma(\gamma)=-J\pi_{\xi}\dot{\gamma},\]
where $\pi_{\xi}:TM\to\xi$ is the projection along the Reeb direction. Importantly, if $\gamma$ is a Reeb orbit of $\lambda$, then for all $t\in S^1$, $\pi_{\xi}(\dot{\gamma}(t))=0$, and we conclude that $\sigma(\gamma)=0$, for all  $\gamma\in\mathcal{P}(\lambda)$. This section $\sigma$ is a positive gradient of $\mathcal{A}_{\lambda}$ with respect to an $L^2$-inner product (defined below). Finally, $A_{\gamma}$ is obtained by linearizing $\sigma$ at the Reeb orbit $\gamma$. For a more thorough treatment of this narrative, see \cite[\S 3.3]{W3}

Our next goal is to extend the domain and range of the asymptotic operator $A_{\gamma}$ to Banach spaces and apply Fredholm theory. Recall that an inner product $\langle \cdot\, ,\cdot \rangle_p$ on $\xi_p$ is induced by $d\lambda|_{\xi_p}$ and $J_p$, for $p\in M$. We introduce an $L^2$-inner product on $\Gamma(\gamma^*\xi)$: \[\eta_1, \eta_2\in\Gamma(\gamma^*\xi),\,\,\,\,\,\,\langle\eta_1,\eta_2\rangle:=\int_{t\in S^1}\langle\eta_1(t),\eta_2(t)\rangle_{\gamma(t)}\,dt.\]
The operator $A_{\gamma}$ is symmetric with respect to this inner product. That is, $\langle A_{\gamma}\eta_1,\eta_2\rangle=\langle \eta_1, A_{\gamma}\eta_2\rangle$. This can be shown by integrating by parts and observing that $J_p$ is a $\langle\cdot,\cdot\rangle_p$-isometry and that $\nabla$ is a symplectic connection. 

Let $L^2(\gamma^*\xi)$ denote the completion of $\Gamma(\gamma^*\xi)$ with respect to this $L^2$-inner product, whose corresponding norm is denoted $|\cdot|_{L^2}$. Let $W^{1,2}(\gamma^*\xi)$ denote the completion of $\Gamma(\gamma^*\xi)$ with respect to the norm \[|\eta|_{W^{1,2}}:=|\nabla_t\eta|_{L^2}+|\eta|_{L^2}.\] A \emph{unitary trivialization} of the Hermitian bundle $(\gamma^*\xi,d\lambda|_{\xi},J)$ (defined below) provides an identification $\Gamma(\gamma^*\xi)\cong C^{\infty}(S^1,\R^{2n})$, and also isometrically identifies $L^2(\gamma^*\xi)$ and $W^{1,2}(\gamma^*\xi)$ with the classical Banach spaces $L^2(S^1,\R^{2n})$ and $W^{1,2}(S^1,\R^{2n})$, respectively. In this way we regard $W^{1,2}(\gamma^*\xi)$ as a dense subspace of $L^2(\gamma^*\xi)$. 
    
The asymptotic operator extends to $A_{\gamma}:W^{1,2}(\gamma^*\xi)\to L^2(\gamma^*\xi)$. We  understand $A_{\gamma}$ to be an \emph{unbounded operator} on $L^2(\gamma^*\xi)$, whose dense domain of definition is $W^{1,2}(\gamma^*\xi)\subset L^2(\gamma^*\xi)$. The operator $A_{\gamma}$ is  self-adjoint. Thus, its spectrum is discrete and consists entirely of real eigenvalues, with each eigenvalue appearing with finite  multiplicity.

Alternatively, we can understand the asymptotic operator as a \emph{bounded operator} between Banach spaces, \[A_{\gamma}:\big(W^{1,2}(\gamma^*\xi),|\cdot|_{W^{1,2}}\big)\to\big(L^{2}(\gamma^*\xi),|\cdot|_{L^2}\big).\]
We will see that the bounded operator $A_{\gamma}$ is an example of a \emph{Fredholm operator}, meaning its kernel and cokernel are finite dimensional. The \emph{index} of a Fredholm operator $T$ is defined to be \[\text{ind}(T):=\text{dim}(\text{ker}\,T)-\text{dim}(\text{coker}\,T).\] 

\begin{proposition}
The bounded operator  $A_{\gamma}$ is Fredholm of index 0.
\end{proposition}

\begin{proof}
Take any unitary trivialization $\tau$ of $\gamma^*\xi$: \[\tau:\gamma^*\xi\to S^1\times\R^{2n}.\] Explicitly, this is a diffeomorphism, preserving the fibration over $S^1$ that restricts to linear isomorphisms $\xi_{\gamma(t)}\xrightarrow{\sim}\{t\}\times\R^{2n}$ which take $d\lambda|_{\xi}$ and $J$ to the standard symplectic form and complex structure on $\{t\}\times\R^{2n}$, for $t\in S^1$. For some  smooth, symmetric matrix-valued function on $S^1$, depending on the choice of $\tau$, denoted $t\in S^1\mapsto S(t)$, the asymptotic operator takes the form \[A_{\tau}:W^{1,2}(S^1,\R^{2n})\to L^2(S^1,\R^{2n}),\,\,\,\eta\mapsto -J_0\dot{\eta}-S\eta.\]
Here, $J_0$ is the $2n\times2n$ matrix defining the standard complex structure on $\R^{2n}$. 

We first argue that the map \[W^{1,2}(S^1,\R^{2n})\to L^2(S^1,\R^{2n}),\,\,\,\eta\mapsto\dot{\eta}\] is Fredholm of index 0 by showing that both the kernel and cokernel are $2n$-dimensional. To show this, it is enough to demonstrate that the operator \[T:W^{1,2}(S^1,\R)\to L^2(S^1,\R),\,\,\,T(f)=f'\] has 1-dimensional kernel and cokernel. The kernel contains the constant functions, and conversely,  if the weak derivative of $f$ is zero almost everywhere, then $f$ has a representative that is constant on $S^1$ in $L^2(S^1,\R)$. This shows $\text{dim}(\text{ker}(T))=1$. To study the cokernel, note that $f\in L^2(S^1,\R)$ is in the image of $T$ if and only if \[\int_{S^1} f(t)\,dt=0.\] This means that $\text{im}(T)$ precisely equals those elements $f\in L^2(S^1,\R)$ whose first Fourier coefficient vanishes, i.e. $a_0=\tfrac{1}{2}\int_{S^1}f(t),\,dt=0$. The orthogonal complement of this space is precisely the constant functions, and so the cokernel of $T$ is also 1-dimensional.

Next, we use this to argue that the asymptotic operator is Fredholm of index 0 as well. The map \[F:W^{1,2}(S^1,\R^{2n})\to L^2(S^1,\R^{2n}),\,\,\,\eta\mapsto -J_0\dot{\eta}
\] is also Fredholm of index 0, because fiber-wise post-composition by the constant $-J_0$ is a Banach space isometry of $L^2(S^1,\R^{2n})$. Finally, note that the inclusion of $W^{1,2}(S^1,\R^{2n})$ into $L^2(S^1,\R^{2n})$ is a \emph{compact operator}, thus so is the map \[K:W^{1,2}(S^1,\R^{2n})\to L^2(S^1,\R^{2n}),\,\,\,\eta\mapsto -S\eta.\] In total, the asymptotic operator takes the form $A_{\tau}=F+K$, where $F$ is Fredholm of index 0, and $K$ is a compact operator. Because the Fredholm property (and its index) are stable under compact perturbations, the asymptotic operator is a Fredholm operator of index 0.
\end{proof}

\begin{remark}
The asymptotic operator $A_{\gamma}$ has trivial kernel (and thus, is an isomorphism) if and only if $\gamma$ is a nondegenerate  Reeb orbit. To see why, recall that a Reeb orbit is degenerate whenever the linearized return map has 1 as an eigenvalue. If there is a nonzero 1-eigenvector $v$ in $\xi_{\gamma(0)}$ of the linearized return map $d\phi^T$, then one can produce a smooth section $\eta$ of the pullback bundle $\gamma^*\xi$ by pushing forward $v$  using the linearized Reeb flow. Importantly, this $\eta$ is indeed periodic because we have that $\eta(1)=d\phi^T(v)=v=\eta(0)$. The resulting section $\eta$ is parallel with respect to $\nabla$ (by definition), and so $A_{\gamma}(\eta)=-J\nabla_t\eta=0$. Conversely, given a nonzero section $\eta$ in the kernel of $A_{\gamma}$, $\eta(0)$ is a 1-eigenvector of $d\phi^T:\xi_{\gamma(0)}\to\xi_{\gamma(0)}$.
\end{remark}

Following \cite[\S 3.2]{W3}, we define the \emph{spectral flow} between two asymptotic operators, which is the first of two ways that we will describe the Conley-Zehnder indices of Reeb orbits in Section \ref{subsection: the cz index}. Let $\mathcal{F}$ denote the collection of Fredholm operators \[A:W^{1,2}(S^1,\R^{2n})\to L^{2}(S^1,\R^{2n}).\] This is an open subset of the space of bounded operators from $W^{1,2}(S^1,\R^{2n})$ to $L^2(S^1,
\R^{2n})$, and so is itself a Banach manifold. Furthermore, the function \[\text{ind}:\mathcal{F}\to\Z, \,\, A\mapsto\text{ind}(A)\] is locally constant, and so the collection of Fredholm index 0 operators, $\mathcal{F}^0\subset\mathcal{F}$, is an open subset and is itself a Banach manifold. Let $\mathcal{F}^{0,1}\subset\mathcal{F}^0$ denote those Fredholm index 0 operators whose kernel is 1-dimensional. By \cite[Proposition 3.5]{W3}, $\mathcal{F}^{0,1}$ is a smooth Banach submanifold of $\mathcal{F}^0$ of codimension 1. Select two invertible elements, $A_{+}$ and $A_{-}$, of $\mathcal{F}^0$ in the same connected component. Then their \emph{relative spectral flow}, \[\mu^{\text{spec}}_2(A_+,A_-)\in\Z/2\Z\] is defined to be the mod-2 count of the number of intersections between a generic path in $\mathcal{F}^0$ from $A_+$ to $A_-$ and $\mathcal{F}^{0,1}$ (a quantity independent in generic choice of path).

We can upgrade the relative spectral flow to an integer by restricting our studies to symmetric operators. An operator $A:W^{1,2}(S^1,\R^{2n})\to L^2(S^1,\R^{2n})$ is \emph{symmetric} if $\langle A x,y\rangle=\langle x,Ay\rangle$ for all $x$ and $y$ in the domain. Let $\mathcal{F}_{\text{sym}}$, $\mathcal{F}_{\text{sym}}^0$, and $\mathcal{F}_{\text{sym}}^{0,1}$ denote the intersections of $\mathcal{F}$, $\mathcal{F}^0$, and $\mathcal{F}^{0,1}$ with the closed subspace of symmetric operators.  Now $\mathcal{F}_{\text{sym}}^0$ is a Banach manifold and $\mathcal{F}_{\text{sym}}^{0,1}$ is a \emph{co-oriented}\footnote{The \emph{co-oriented} condition means that the rank 1 real line bundle $T\mathcal{F}_{\text{sym}}^0/T\mathcal{F}_{\text{sym}}^{0,1}$ over $\mathcal{F}_{\text{sym}}^{0,1}$ is an \emph{oriented} line bundle. Given a differentiable path $(-\epsilon,\epsilon)\to \mathcal{F}_{\text{sym}}^0$ that transversely intersects $\mathcal{F}_{\text{sym}}^{0,1}$ at 0, we can use the given orientation to determine if the intersection is a \emph{positive} or \emph{negative} crossing. In Remark \ref{remark: co-oritentation of hypersurface of operators}, we provide a method to determine the sign of a transverse intersection.} codimension-1 submanifold. 

Now, given two symmetric, Fredholm index 0 operators, $A_+$ and $A_-$, with trivial kernel in the same connected component of $\mathcal{F}^0_{\text{sym}}$, define \[\mu^{\text{spec}}(A_+,A_-)\in\Z\] to be the signed count of a generic path in $\mathcal{F}^0_{\text{sym}}$ from $A_+$ to $A_-$ with the co-oriented $\mathcal{F}^{0,1}_{\text{sym}}$. See Remark \ref{remark: co-oritentation of hypersurface of operators} at the end of this section for algebraic descriptions of the co-orientation and formulaic means of computing $\mu^{\text{spec}}(A_+,A_-)$.

Recall that unitarily trivialized asymptotic operators associated to Reeb orbits are examples of symmetric Fredholm index 0 operators. A path of these operators joining $A_+$ to $A_-$ is described by a 1-parameter family of loops of $2n\times 2n$ real symmetric matrices, \[\{S_s\}_{s\in[-1,1]},\,\,\,S_s:S^1\to\text{Sym}(2n),\,\,\,t\mapsto S_s(t),\]
where the corresponding family of asymptotic operators is \[\{A_s\}_{s\in[-1,1]}\subset\mathcal{F}^0_{\text{sym}},\,\,\,A_s(\eta)=-J_0\dot{\eta}-S_s\eta,\] and $A_{\pm1}=A_{\mp}$. The spectrum of each $A_s$ is real, discrete, and consists entirely of eigenvalues, with each eigenvalue appearing with finite  multiplicity. There exists a $\Z$-family of continuous functions enumerating the spectra of $\{A_s\}$: \[\{e_j:[-1,1]\to\R\}_{j\in\Z},\,\,\,\text{spec}(A_s)=\{e_j(s)\}_{j\in\Z}, \,\,\,\text{for}\,\,s\in[-1,1]\] with multiplicity, in the sense that for any $e\in\R$ \[\text{dim}(\text{ker}(A_s-e\cdot\iota))=\big|\{j\in\Z:e_j(s)=e\}\big|,\] where $\iota:W^{1,2}(S^1,\R^{2n})\to L^2(S^1,\R^{2n})$ is the standard inclusion. Assuming that $A_{\pm}$ have trivial kernels, then the spectral flow from $A_+$ to $A_-$ is then the \emph{net} number of eigenvalues (with multiplicity) that have changed from negative to positive: \[\mu^{\text{spec}}(A_+,A_-)=\big|\{j\in\Z:e_j(-1)<0<e_j(1)\}\big|-\big|\{j\in\Z:e_j(-1)>0>e_j(1)\}\big|.\] See \cite[Theorem 3.3]{W3}. 

\begin{remark}\label{remark: co-oritentation of hypersurface of operators} (Co-orientation of the hypersurface $\mathcal{F}_{\text{sym}}^{0,1}$) On one hand, the spectral flow is given as a path's signed count of intersections with a co-oriented hypersurface. On the other, the spectral flow equals the net change in eigenvalues that have changed from negative to positive. This equality implies that a differentiable path $A:(-\epsilon,\epsilon)\to\mathcal{F}_{\text{sym}}^0$ \emph{positively} crosses  $\mathcal{F}_{\text{sym}}^{0,1}$ at $0$ if a corresponding 1-parameter family of eigenvalues changes sign from negative to positive at 0. That is, when there exists differentiable families $e:(-\epsilon,\epsilon)\to\R$, and $\eta:(-\epsilon,\epsilon)\to W^{1,2}(S^1,\R^{2n})$ nonzero, with $A_s\eta_s=e(s)\eta_s$, the transverse intersection of $A$ with $\mathcal{F}_{\text{sym}}^{1,0}$ is positive if and only if $\dot{e}(0)>0$.
\end{remark}

We use Remark \ref{remark: co-oritentation of hypersurface of operators} to algebraically refashion our geometric description of  $\mu^{\text{spec}}(A_+,A_-)$. Fix invertible asymptotic operators $A_{\pm}$. Let $A:=\{A_s\}_{s\in[-1,1]}$ be a generic differentiable path of asymptotic operators from $A_+$ to $A_-$ (meaning $A_{\pm1}=A_{\mp}$) in $\mathcal{F}_{\text{sym}}^0$, so that  $A_s=-J_0\partial_t-S_s$, where $S_s:S^1\to\text{Sym}(2n)$ is a loop  of symmetric matrices, for $s\in[-1,1]$. The path transversely intersects the co-oriented hypersurface $\mathcal{F}^{0,1}_{\text{sym}}$. A value of $s\in(-1,1)$ for which $A_s\in\mathcal{F}_{\text{sym}}^{0,1}$ is called a \emph{crossing} for A. Let $s_0\in(-1,1)$ be a crossing for $A$, and define a quadratic form \[\Gamma(A,s_0):\text{ker}(A_{s_0})\to\R,\,\,\,\eta\mapsto\langle\dot{A}_{s_0}\eta,\eta\rangle_{L^2},\] where $\langle\cdot\,,\cdot\rangle_{L^2}$ is the standard $L^2$-inner product. Because $A$ is generic, the eigenvalue of $\Gamma(A,s_0)$ is nonzero. The sign of the single eigenvalue of $\Gamma(A,s_0)$, denoted $\text{Sign}(\Gamma(A,s_0))\in\{\pm1\}$, is the signature of the quadratic form.
\begin{claim}
The path $A$ intersects $\mathcal{F}_{\text{sym}}^{0,1}$ positively at $s_0$ if and only if $\text{Sign}(\Gamma(A,s_0))=+1$.
\end{claim}
\begin{proof}\
 By arguments used in \cite[Theorem 3.3]{W3}, we  have a differentiable path of nonzero eigenvectors of $A_s$, \[\eta:(s_0-\epsilon,s_0+\epsilon)\to W^{1,2}(S^1,\R^{2n}),\,\,\,s\mapsto\eta_s,\] and a path of corresponding eigenvalues \[e:(s_0-\epsilon,s_0+\epsilon)\to\R,\,\,\,s\mapsto e(s),\] such that for all $s$, \[A_s\eta_s=e(s)\eta_s.\]
 We have that $e(s_0)=0$. For simplicity, we can assume that the eigenvectors are of unit length, i.e. $|\eta_s|_{L^2}=1$ for all $s$. Now, take a derivative in the $s$ variable of the above equation, evaluate at $s_0$, and take an inner product with $\eta_{s_0}$ to obtain \[\Gamma(A,s_0)(\eta_{s_0})=\langle\dot{A}_{s_0}\eta_{s_0},\eta_{s_0}\rangle_{L^2}=\dot{e}(s_0),\]
 where we have used that $\langle A_{s_0}\dot{\eta}_{s_0},\eta_{s_0}\rangle=\langle \dot{\eta}_{s_0},A_{s_0}\eta_{s_0}\rangle=0$ by symmetry and that $\eta_{s_0}\in\text{ker}(A_{s_0})$. Ultimately, this verifies that the sign of the quadratic form $\Gamma(A,s_0)$ equals that of $\dot{e}(s_0)$, which is the sign of the intersection of the path $A$ with $\mathcal{F}_{\text{sym}}^{0,1}$ at $s_0$.
\end{proof}
\begin{corollary}\label{corollary: spectral flow in terms of crossings}
The spectral flow satisfies \[\mu^{spec}(A_+,A_-)=\sum_{s\in(-1,1)}\mbox{\em Sign}(\Gamma(A,s)),\]
where $A_{\pm}$ are invertible elements of $\mathcal{F}_{\text{sym}}^{0}$, $A=\{A_s\}_{s\in[-1,1]}$ is a generic\footnote{See \cite[Appendix C]{W3} for a thorough description of the topology of the space of asymptotic operators and of genericity in this context.}, differentiable family of asymptotic operators from $A_+$ to $A_-$ (meaning $A_{\pm1}=A_{\mp}$) of the form $A_s=-J_0\partial_t-S_s$.
\end{corollary}

\subsection{The Conley-Zehnder index} \label{subsection: the cz index}

In this section, we outline various ways to define and work with Conley-Zehnder indices. The Conley-Zehnder index is an integer that we assign to a nondegenerate Reeb orbit $\gamma$ with respect to a choice of trivialization $\tau$ of $\gamma^*\xi$, denoted $\mu_{\CZ}^{\tau}(\gamma)\in\Z$. This integer mirrors the role of the Morse index associated to a critical point. However, as a standalone number, the Conley-Zehnder index has less geometric meaning than the Morse index; rather, the \emph{difference} of Conley-Zehnder indices is generally the more valuable quantity to consider.

Let $\gamma$ be a nondegenerate Reeb orbit in contact manifold $(M,\lambda)$ of dimension $2n+1$. Let $J$ be a complex structure on $\xi=\text{ker}(\lambda)$ that is compatible with $d\lambda|_{\xi}$. Let \[\tau: \gamma^*\xi\to S^1\times \R^{2n}\] be a unitary trivialization of $(\gamma^*\xi,d\lambda|_{\xi},J)$.  Recall from Section \ref{subsection: asymptotic operators and spectral flows} that $\gamma$ and $\tau$ define a nondegenerate asymptotic operator, denoted $A_{\tau}$ \[A_{\tau}:W^{1,2}(S^1,\R^{2n})\to L^2(S^1,\R^{2n}),\,\,\,A_{\tau}(\eta)=-J_0\dot{\eta}-S\eta,\] where $S$ is some smooth, $S^1$-family of symmetric matrices. Define the \emph{Conley-Zehnder index} associated to $\gamma$ with respect to $\tau$: \begin{equation}\label{equation: cz reeb orbit definition using spectral flow}
    \mu_{\CZ}^{\tau}(\gamma):=\mu^{\text{spec}}(A_{\tau},A_0).
\end{equation} Here, $A_0:W^{1,2}(S^1,\R^{2n})\to L^2(S^1,\R^{2n})$ is the reference operator \[A_0(\eta)=-J_0\dot{\eta}-S_0\eta, \]
where $S_0$ is the constant, symmetric, diagonal matrix  $\text{Id}_n\oplus-\text{Id}_n$, and $J_0$ is the $2n\times2n$ block matrix \[J_0=\begin{pmatrix} 0 & -\text{Id}_n \\ \text{Id}_n & 0\end{pmatrix}.\] Because the spectral flow is defined in terms of an intersection of paths with a co-oriented hypersurface, we immediately have the following difference formula by concatenation \[\mu_{\CZ}^{\tau_+}(\gamma_+)-\mu_{\CZ}^{\tau_-}(\gamma_-)=\mu^{\text{spec}}(A_{+},A_{-}),\]
where $\gamma_{\pm}$ are nondegenerate Reeb orbits, $\tau_{\pm}$ are unitary trivializations of $\gamma_{\pm}^*\xi$, and $A_{\pm}:W^{1,2}(S^1,\R^{2n})\to L^2(S^1,\R^{2n})$ are the asymptotic operators induced by $\tau_{\pm}$.

There is a second definition of the Conley-Zehnder index, as an integer assigned to a path of symplectic matrices. The following remark is necessary in bridging the two perspectives.
\begin{remark}\label{remark: symplectic symmetric}
Let $\Phi:[0,1]\to\Sp(2n)$ be a differentiable, 1-parameter family of symplectic matrices. Then  $S:=-J_0\dot{\Phi}\Phi^{-1}$ is a 1-parameter family of \emph{symmetric} matrices. Conversely, given a differentiable 1-parameter family of symmetric matrices $S:[0,1]\to\text{Sym}(2n)$, then the solution $\Phi$ of the initial value problem \[\dot{\Phi}=J_0S\Phi,\,\,\, \Phi(0)=\text{Id}\] is a 1-parameter family of \emph{symplectic} matrices.
\end{remark}

We now define the \emph{Conley-Zehnder index} associated to an arc of symplectic matrices $\Phi$, denoted $\mu_{\CZ}(\Phi)\in\Z$. We'll then use $\mu_{\CZ}(\Phi)$ to provide an alternative definition of $\mu_{\CZ}^{\tau}(\gamma)$, equivalent to  \eqref{equation: cz reeb orbit definition using spectral flow}. Let $\Phi:[0,1]\to\Sp(2n)$ be a differentiable path of symplectic matrices such that $\Phi(0)=\text{Id}$ and $\text{ker}(\Phi(1)-\text{Id})=0$. Let $S$ denote the corresponding path of symmetric matrices, from Remark \ref{remark: symplectic symmetric}. We say that $t\in(0,1)$ is a \emph{crossing} for $\Phi$ if $\text{ker}(\Phi(t)-\text{Id})$ is nontrivial. For each crossing $t$, define a quadratic form $\Gamma(\Phi,t):\text{ker}(\Phi(t)-\text{Id})\to\R$ by \[\Gamma(\Phi,t)(v)=\omega_0(v,\dot{\Phi}(t)v)=\langle v, S(t)v\rangle,\] where $\omega_0$ and $\langle\cdot,\cdot\rangle$ are the standard symplectic form and inner product on $\R^{2n}$. The crossing $t$ is \emph{regular} if this form is nondegenerate, i.e., 0 is not an eigenvalue. Note that regular crossings are isolated in $(0,1)$ and that a generic path $\Phi$ has only regular crossings.

Now, we define the Conley-Zehnder index of a path $\Phi$ with only regular crossings to be \begin{equation}\label{equation: def cz crossing}
    \mu_{\CZ}(\Phi):=\frac{1}{2}\text{Sign}(S(0))+\sum_{t\in(0,1)}\text{Sign}\big(\Gamma(\Phi,t)\big),
\end{equation}
where $\text{Sign}(\cdot)$ is the number of positive eigenvalues minus the number of negative eigenvalues, with multiplicity, of either a quadratic form or a symmetric matrix. The Conley-Zehnder index is invariant under perturbations relative to the endpoints. Thus, after a  small perturbation of any $\Phi$, we can assume that all crossings are regular.

Geometrically, the Conley-Zehnder index of a path of matrices is  a signed count of its intersection with the set $X=\{A\in\Sp(2n):\text{det}(A-\text{Id})=0\}$. There is a stratification of $X$ by $\text{dim}(\text{ker}(A-\text{Id}))$. Generically, a path in $\Sp(2n)$ will transversely intersect $X$ at points in the largest stratum, the hypersurface $X^{1}=\{A\in X:\text{dim}(\text{ker}(A-\text{Id}))=1\}$. 

Let $\gamma:S^1\to M$ be a nondegenerate Reeb orbit with action $T$, and let $\tau$ be a unitary trivialization of $(\gamma^*\xi,d\lambda|_{\xi}, J)$. For $t\in[0,1]$, let $\Phi(t)\in\Sp(2n)$ represent the time $T\cdot t$-linearized Reeb flow from $\xi_{\gamma(0)}$ to $\xi_{\gamma(t)}$ with respect to the trivialization $\tau$ (we must re-normalize by $T$ because of our convention in this section of renormalization of the period of $\gamma$ by $T$). Note that $\text{ker}(\Phi(1)-\text{Id})=0$, because $\gamma$ is nondegenerate. Now, define the Conley-Zehnder index of $\gamma$ with respect to $\tau$: \begin{equation}\label{equation: def of cz of reeb wrt sp matrices}
    \mu_{\CZ}^{\tau}(\gamma):=\mu_{\CZ}(\Phi).
\end{equation}

\begin{proposition}\label{proposition: equivalence of cz definitions} The definitions \eqref{equation: cz reeb orbit definition using spectral flow} and \eqref{equation: def of cz of reeb wrt sp matrices} of the Conley-Zehnder index of a Reeb orbit with a unitary trivialization are equivalent.
\end{proposition}

\begin{proof}
 We expand on the arguments used in \cite[\S 2.5]{S} regarding spectral flows in Hamiltonian Floer theory, which concerns negative gradient trajectories of the action functional. Note the differences in sign conventions come from the fact that symplectic field theory concerns \emph{positive} gradient trajectories. See \cite[\S 3.3]{W3} for more details on regarding these sign conventions. 
 
 Take nondegenerate Reeb orbits $\gamma_{\pm}$ and unitary trivializations $\tau_{\pm}$ of $\gamma_{\pm}^*\xi$. Let $A_{\pm}=-J_0\partial_t-S_{\pm}$ be the corresponding asymptotic operators as in definition \eqref{equation: cz reeb orbit definition using spectral flow}, let $\Phi_{\pm}$ be the corresponding paths of symplectic matrices as appearing in definition \eqref{equation: def of cz of reeb wrt sp matrices}. The paths $\Phi_{\pm}$ are related to the loops $S_{\pm}$ by Remark \ref{remark: symplectic symmetric}. We first argue that \begin{equation}\label{equation: off by shift}
     \mu_{\CZ}(\Phi_+)-\mu_{\CZ}(\Phi_-)=\mu^{\text{spec}}(A_+,A_-).
 \end{equation} This will demonstrate that the definitions differ by a constant integer. Then we will show that the definitions agree on a choice of a specific value, completing the proof.
 
 Let $S:[-1,1]\times S^1\to \text{Sym}(2n)$ be a smooth, generic path of loops of symmetric matrices, $(s,t)\mapsto S_s(t)$, where $S_s:S^1\to\text{Sym}(2n)$, from $S_+$ to $S_-$. This means $S_{\pm1}=S_{\mp}$. The path of loops $\{S_s\}_{s\in[-1,1]}$ defines a path of asymptotic operators $A=\{A_s\}_{s\in[-1,1]}$ from $A_+$ to $A_-$, given by $A_s=-J_0\partial_t-S_s$, so that $A_{\pm1}=A_{\mp}$. The path of loops $\{S_s\}_{s\in[-1,1]}$ additionally defines a 1-parameter family of paths of symplectic matrices, $\{\Phi_s\}_{s\in[-1,1]}$, where $\Phi_s:[0,1]\to\Sp(2n)$, for $s\in[-1,1]$, with $\Phi_s(0)=\text{Id}$, as described in Remark \ref{remark: symplectic symmetric}. We similarly have that $\Phi_{\pm1}=\Phi_{\mp}$. Define $\Phi:[-1,1]\times[0,1]\to\Sp(2n)$, $\Phi(s,t)=\Phi_s(t)$. Finally, let $\Psi:[-1,1]\to\Sp(2n)$ be given by $\Psi(s):=\Phi(s,1)=\Phi_s(1)$. Figure \ref{figure: crossings} illustrates this data.
 
 \begin{figure}[h]
    \centering
    \captionsetup{justification=centering}
    \includegraphics[width=1.0\textwidth]{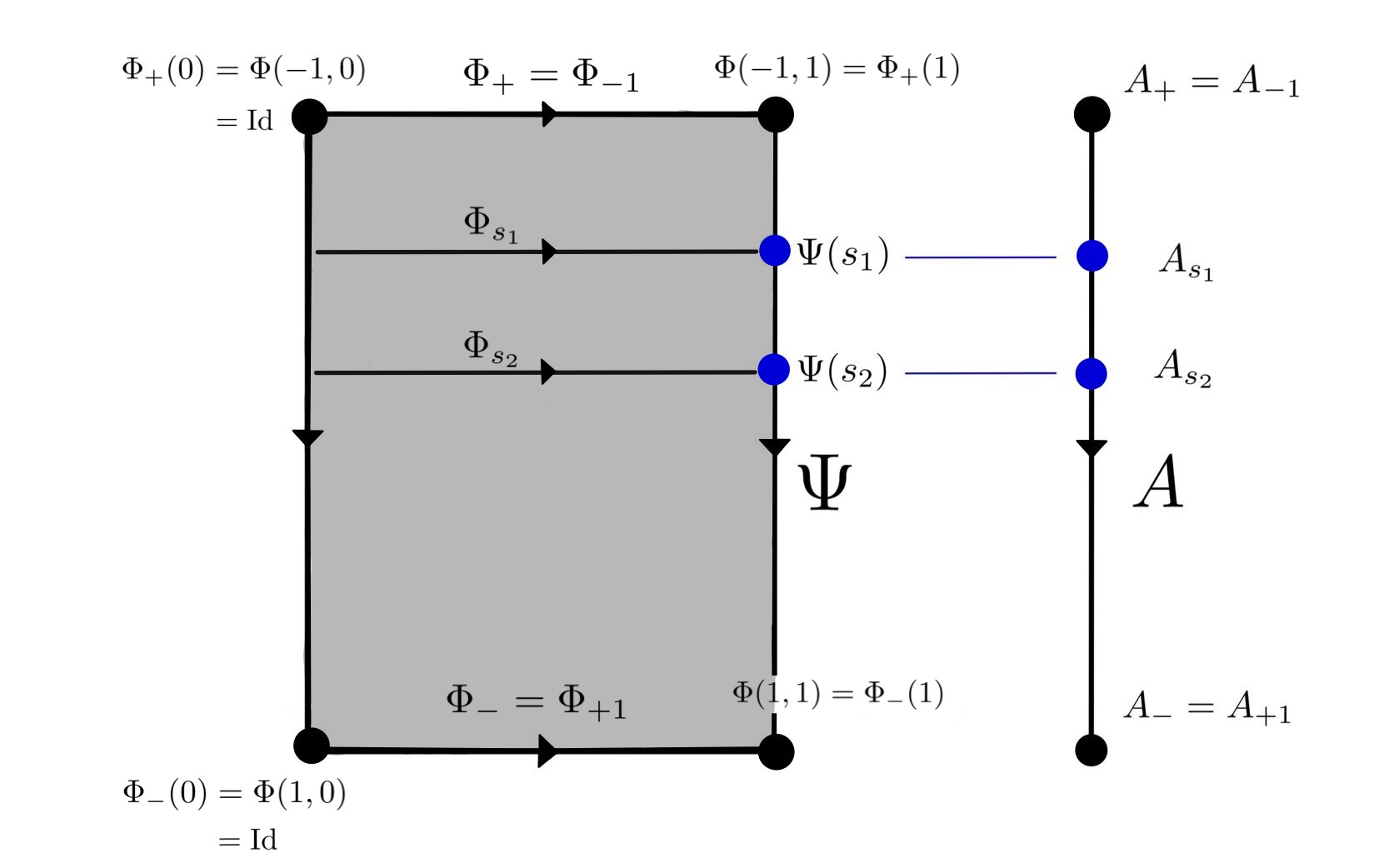}
    \caption{A rectangle of symplectic matrices on the left, and a corresponding path of asymptotic operators on the right. Both figures depict two theoretical  $s$ values in $(-1,1)$  that are crossings for both $\Psi$ and $A$ in blue.}
    \label{figure: crossings}
\end{figure}
 
 See that the concatenation $\Phi_+\cdot\Psi$ is a path of symplectic matrices from  $\text{Id}$ to $\Phi_-(1)$ that is homotopic to $\Phi_-$, rel endpoints. Thus, we have that $\mu_{\CZ}(\Phi_-)=\mu_{\CZ}(\Phi_+\cdot\Psi)$. By definition \eqref{equation: def cz crossing}, \begin{align}
     \mu_{\CZ}(\Phi_+\cdot\Psi)&=\frac{1}{2}\text{Sign}(S_+(0))+\sum_{t\in(0,1)}\text{Sign}(\Gamma(\Phi_+,t))+\sum_{s\in(-1,1)}\text{Sign}(\Gamma(\Psi,s))\nonumber\\
     &=\mu_{\CZ}(\Phi_+)+\sum_{s\in(-1,1)}\text{Sign}(\Gamma(\Psi,s)). \label{equation: relating cz of matrix}
 \end{align} Let $Z$ denote the sum of signatures  appearing on the right in equation \eqref{equation: relating cz of matrix}, so that \[\mu_{\CZ}(\Phi_+)-\mu_{\CZ}(\Phi_-)=\mu_{\CZ}(\Phi_+)-\mu_{\CZ}(\Phi_+\cdot\Psi)=-Z.\]
 We show that $Z=-\mu^{\text{spec}}(A_+,A_-)$, implying equation \eqref{equation: off by shift}. 
 
 For each $s\in(-1,1)$, we construct an isomorphism $F_s:\text{ker}(A_s)\to\text{ker}(\Psi(s)-\text{Id})$. It's worth noting that these vector spaces are generically trivial. First note that an eigenvector of any asymptotic operator has a smooth representative in $W^{1,2}(S^1,\R^{2n})$, thus $\text{ker}(A_s)$ consists entirely of smooth functions. An element $\eta\in\text{ker}(A_s)$ is necessarily of the form $t\mapsto \Phi_s(t)\eta(0)$. This is because both $\eta(t)$ and $\Phi_s(t)\eta(0)$ are solutions to the initial value problem \[\dot{Y}(t)=J_0S_s(t)Y(t),\,\,\,Y(0)=\eta(0),\] thus $\eta(t)=\Phi_s(t)\eta(0)$. 
 
 Importantly, because $\eta$ is 1-periodic, this implies that \[\eta(0)=\eta(1)=\Phi_s(1)\eta(0)=\Psi(s)\eta(0),\] demonstrating that $\eta(0)\in\text{ker}(\Psi(s)-\text{Id})$. Define \[F_s:\text{ker}(A_s)\to\text{ker}(\Psi(s)-\text{Id}),\,\,\,\eta\mapsto\eta(0).\] This linear map $F_s$ is invertible, with inverse given by $v\mapsto\eta_v$, where $\eta_v(t)=\Phi_s(t)v$. In total, the spaces $\text{ker}(A_s)$ and $\text{ker}(\Psi(s)-\text{Id})$ are isomorphic for all $s$. Because $\{S_s\}_{s\in[-1,1]}$ was chosen generically, $\text{ker}(A_s)$ is nonzero only for finitely many $s$, whose dimension is 1 at such values $s$. Furthermore, we have that $F_s$ negates the quadratic forms on $\text{ker}(A_s)$ and $\text{ker}(\Psi(s)-\text{Id})$, in the sense that for any $e\in\R$, \[\Gamma(\Psi,s)(F_s(\eta))=-\Gamma(A,s)(\eta),\]
For $\eta\in\text{ker}(A_s)$. Lemma \ref{lemma: crossing forms are related by a negative sign} proves this fact (for $\text{dim}(\text{ker}(A_s))$ arbitrary) in Appendix \ref{appendix: signature of crossing forms}. This equality allows us to conclude \[Z=\sum_{s\in(-1,1)}\text{Sign}(\Gamma(\Psi,s))=-\sum_{s\in(-1,1)}\text{Sign}(\Gamma(A,s))=-\mu^{\text{spec}}(A_+,A_-),\]
 where we have used the formula derived in Corollary \ref{corollary: spectral flow in terms of crossings} in the last equality. This proves equation \ref{equation: off by shift}. 
 
 It remains to show that Definitions \eqref{equation: cz reeb orbit definition using spectral flow} and \eqref{equation: def of cz of reeb wrt sp matrices} agree at a single value. Consider the asymptotic operator $A_0=-J_0\partial_t-S_0$, where $S_0$ is the constant loop of symmetric matrices $\text{Id}_n\oplus-\text{Id}_n$. The Conley-Zehnder index with respect to definition \ref{equation: cz reeb orbit definition using spectral flow} is immediately given by $\mu^{\text{spec}}(A_0,A_0)=0$. Let $\Phi_0:[0,1]\to\Sp(2n)$ be the corresponding path of symplectic matrices, according to Remark \ref{remark: symplectic symmetric}. We argue that $\mu_{\CZ}(\Phi_0)=0$, proving that the the definitions agree at a specified value, and thus, at all values.
 
 An explicit expression for the solution $\Phi_0$ of the initial value problem $\dot{\Phi}_0=J_0S_0\Phi_0$, $\Phi_0(0)=\text{Id}_{2n}$ is given in block form by \[\Phi_0(t)=\begin{pmatrix}\cosh{(t)}\text{Id}_n & \sinh{(t)}\text{Id}_n &\\ \sinh{(t)}\text{Id}_n  & \cosh{(t)}\text{Id}_n\end{pmatrix},\] where we use the block form of $J_0$; \[J_0=\begin{pmatrix} 0 & -\text{Id}_n \\ \text{Id}_n & 0\end{pmatrix}.\]
 Using formulas for the determinant of a matrix in block form, we see that \[\text{det}(\Phi_0(t)-\text{Id}_{2n})=2^n(1-\cosh{(t)})^n.\] This is 0 only at $t=0$, and so the path $\Phi_0$ has no crossings in $(0,1]$. Finally, by definition \eqref{equation: def cz crossing}, we have \[\mu_{\CZ}(\Phi_0)=\frac{1}{2}\text{Sign}(S_0(0))+0=\frac{1}{2}(n-n)=0,\] completing the proof.
\end{proof}
There are advantages to understanding both perspectives of  $\mu_{\CZ}^{\tau}(\gamma)$. The spectral flow  definition is useful in seeing that the difference of Conley-Zehnder indices is a net change in eigenvalues of the corresponding asymptotic operators. This highlights  the idea that the asymptotic operator and Conley-Zehnder index are analogous to the Hessian and Morse index in the finite dimensional setting. The crossing form definition is computationally useful, and we make heavy use of it in Section \ref{section: computation of filtered contact homology}.

\begin{remark}\label{remark: properties of cz}
(Properties of the Conley-Zehnder index of a path of symplectic matrices) Following \cite[\S 2.4]{S}, let $\text{SP}(n)$ denote the collection of paths $\Phi:[0,1]\to\Sp(2n)$ such that $\Phi(0)=\text{Id}$,  and $\text{det}(\Phi(1)-\text{Id})\neq0$. Then $\mu_{\CZ}:\text{SP}(n)\to\Z$ satisfies the following properties.
\begin{enumerate}
    \item \textbf{Naturality:} For any path $N:[0,1]\to\Sp(2n)$, $\mu_{\CZ}(N\Phi N^{-1})=\mu_{\CZ}(\Phi)$.
    \item \textbf{Homotopy:} The values of $\mu_{\CZ}$ are constant on the components of $\text{SP}(n)$.
    \item \textbf{Zero:} If $\Phi(t)$ has no eigenvalue on the unit circle, for $t>0$, then $\mu_{\CZ}(\Phi)=0$.
    \item \textbf{Product:} For $n_1, n_2\in\N$ with $n=n_1+n_2$, using the standard identification $\Sp(n)\cong\Sp(n_1)\oplus\Sp(n_2)$, we have $\mu_{\CZ}(\Phi_1\oplus \Phi_2)=\mu_{\CZ}(\Phi_1)+\mu_{\CZ}(\Phi_2)$.
    \item \textbf{Loop:} If $L:[0,1]\to\Sp(2n)$ is a loop with $L(0)=L(1)$ then $\mu_{\CZ}(L\Phi)=2\mu(L)+\mu_{\CZ}(\Phi)$, where $\mu(L)$ is the Maslov index of $L$ (see below).
    \item \textbf{Signature:} For $S\in\text{Sym}(2n)$ with $\|S\|\footnote{Here, $\|S\|$ denotes the norm of $S$ as an operator on $\R^{2n}$, or equivalently, $\text{max}\{|\lambda|:\lambda\in\text{Spec}(S)\}$.}<2\pi$,  $\mu_{\CZ}(\text{exp}(J_0St))=\tfrac{1}{2}\text{Sign}(S)$.
    \item \textbf{Determinant:} $(-1)^{n-\mu_{\CZ}(\Phi)}=\text{Sign}(\text{det}(\Phi(1)-\text{Id}))$.
    \item \textbf{Inverse:} $\mu_{\CZ}(\Phi^{T})=\mu_{\CZ}(\Phi^{-1})=-\mu_{\CZ}(\Phi)$.
\end{enumerate}
For more details, the reader should consult \cite[\S 3]{SZ}. The Conley-Zehnder index as a family of functions $\{\mu_{\CZ}:\text{SP}(n)\to\Z\}_{n\in\N}$ is uniquely characterized by these properties. The loop property uses the definition of the \emph{Maslov index} of a loop of symplectic matrices, $\mu(L)$, for $L:S^1\to\Sp(2n)$, defined as follows. There is a deformation retraction $r:\Sp(2n)\to \U(n):=\Sp(2n)\cap \text{O}(2n)$ using polar decompositions (see \cite[\S 2.2]{MS}). There exists an identification of the group $\U(n)$ as we have defined it with the collection of unitary matrices in $\text{GL}(n,\C)$ and in this way, we have the determinant map $\text{det}:\U(n)\to S^1\subset\C^*$. The Maslov index of $L$, denoted $\mu(L)\in\Z$, is the degree of the map $\text{det}\circ r\circ L:S^1\to S^1$. For details and applications to Chern classes, see \cite[\S 2]{MS}.
\end{remark}

\begin{remark}\label{remark: rotation numbers}
(Conley-Zehnder indices and rotation numbers) In the 3-dimensional case, we may compute Conley-Zehnder indices of trivialized Reeb orbits by using \emph{rotation numbers}, which we define now. Let $\gamma$ be a nondegenerate Reeb orbit of action $T>0$, and let $\tau$ be a symplectic trivialization of $\gamma^*\xi$. Let $\{\Phi_t\}_{t\in[0,T]}\subset\Sp(2)$ denote the associated 1-parameter family of symplectic matrices defining the Reeb flow along $\gamma$ with respect to $\tau$. We define a \emph{rotation number}, $\theta\in\R$, depending on $\gamma$ and $\tau$, satisfying $\mu_{\CZ}^{\tau}(\gamma)=\lfloor\theta\rfloor+\lceil\theta\rceil\in\Z$. \begin{itemize}
    \item When $\gamma$ is elliptic, then (after perhaps modifying $\tau$ by a homotopy) $\Phi_t$ takes the form $\Phi_t=\text{exp}(2\pi\theta_tJ_0)\in\Sp(2)$, where $t\mapsto\theta_t$ is a continuous function $[0,T]\to\R$ with $\theta_0=0$, $\theta_T\notin\Z$. Define $\theta:=\theta_T$.
    \item When $\gamma$ is hyperbolic, there exists some nonzero eigenvector $v\in\R^2$ of $\Phi_T$. Now, the path $\{\Phi_t(v)\}_{t\in[0,T]}\subset\R^2\setminus\{(0,0)\}$ rotates through an angle of $2\pi k$ from $t=0$ to $t=T$. Define $\theta:=k$. When $\gamma$ is negative hyperbolic, $v$ is an eigenvector whose eigenvalue is negative, so $\Phi_T(v)$ and $v$ point in opposite directions, thus this $2\pi k$ must be an odd multiple of $\pi$ (i.e. $k$ is a half-integer). When $\gamma$ is positive hyperbolic, the eigenvalue is positive, so $\Phi_T(v)$ and $v$ point in the same direction, so $2\pi k$ is an even multiple of $\pi$  (i.e. $k$ is an integer).
\end{itemize}
\end{remark}
In Sections \ref{subsection: cyclic}, \ref{subsection: dihedral}, and \ref{subsection: polyhedral}, we compute the rotation numbers for all of the elliptic orbits (with respect to a global trivialization of $\xi_G$) in our action filtered chain complexes, and use the fact that $\mu_{\CZ}^{\tau}(\gamma)=\lfloor\theta\rfloor+\lceil\theta\rceil$ to compute their Conley-Zehnder indices. 
\begin{remark}\label{remark: parity of Conley-Zehnder indices}
(Parity of Conley-Zehnder indices) In the 3-dimensional case, the rotation numbers provide information regarding the parity of $\mu_{\CZ}^{\tau}(\gamma)$:
\begin{itemize}
    \item If $\gamma$ is elliptic, then (for any trivialization $\tau$), $\theta\in\R\setminus\Z$, and so the quantity \[\mu_{\CZ}^{\tau}(\gamma)=\lfloor\theta\rfloor+\lceil\theta\rceil=2\lfloor\theta\rfloor+1=2\lceil\theta\rceil-1\] is always odd.
    \item If $\gamma$ is negative hyperbolic, then (for any trivialization $\tau$), $\theta$ is a half-integer of the form $\theta=k=k'+\tfrac{1}{2}$, for $k'\in\Z$, and so the quantity \[\mu_{\CZ}^{\tau}(\gamma)=\lfloor\theta\rfloor+\lceil\theta\rceil=2\theta=2k'+1\] is always odd.
    \item If $\gamma$ is positive hyperbolic, then (for any trivialization $\tau$), $\theta$ is an integer $k\in\Z$, and so the quantity \[\mu_{\CZ}^{\tau}(\gamma)=\lfloor\theta\rfloor+\lceil\theta\rceil=2\theta=2k\] is always even.
\end{itemize}
\end{remark}
\begin{remark}\label{remark: CZ indices of iterates}
(Conley-Zehnder indices of iterates) In the 3-dimensional case, we can use rotation numbers to compare the Conley-Zehnder indices of iterates of Reeb orbits. For a nondegenerate Reeb orbit $\gamma$, a symplectic trivialization $\tau$ of $\gamma^*\xi$, and a positive integer $k$ such that $\gamma^k$ is nondegenerate, there exists a unique pullback trivialization of $(\gamma^k)^*\xi$, denoted $k\tau$. Let $\theta$ and $\theta^k$ denote the rotation numbers of $\gamma$ and $\gamma^k$ with respect to $\tau$ and $k\tau$, respectively. Then it is true that $\theta^k=k\theta\in\R$. Notice that, if $\gamma$ is hyperbolic, then we have \[\mu_{\CZ}^{k\tau}(\gamma^k)=2\theta^k=2k\theta=k\mu_{\CZ}^{\tau}(\gamma).\] That is, the Conley-Zehnder index of hyperbolic orbits grows linearly with respect  to iteration (when using appropriate trivializations). In particular, when there exists a global trivialization $\tau$ of $\xi$, \[\mu_{\CZ}(\gamma^k)=k\cdot \mu_{\CZ}(\gamma),\] for all hyperbolic Reeb orbits, and with $\mu_{\CZ}$ computed implicitly with respect to the global trivialization $\tau$. Due to the non-linearity of the floor and ceiling functions, the Conley-Zehnder indices of elliptic orbits do not necessarily grow linearly with the iterate.
\end{remark}
In Sections \ref{subsection: dihedral} and \ref{subsection: polyhedral}, we use this observation to compute the Conley-Zehnder indices of the hyperbolic Reeb orbits appearing in the action filtered chain complexes.

\subsection{Cylindrical contact homology as an analogue of orbifold Morse homology} \label{subsection: cylindrical contact homology as an analogue of orbifold Morse homology}
 Let $X$ be a manifold (or perhaps an orbifold) used to study Morse (or orbifold Morse) theory, and let $(M,\lambda)$ denote a contact manifold. In a loose sense, cylindrical contact homology is a kind of Morse homology performed on $\mathcal{L}(M)$ with respect to the contact action functional, $\mathcal{A}_{\lambda}(\gamma)=\int_{\gamma}\lambda$. This analogy is highlighted in Table \ref{table: Morse SFT comparison}.

\begin{table}[h!]
\centering
 \begin{tabular}{||m{6.3cm} | m{6.9cm} ||} 
 \hline
\textbf{Morse Theory}  & \textbf{Cylindrical Contact Homology}  \\ [0.5ex] 
 \hline\hline
 \footnotesize{$p\mapsto -\text{grad}(f)_p\in T_pX$, the negative gradient of $f$ at $p$, is a section of $TX\to X$ which vanishes precisely at \text{Crit}(f)} & \footnotesize{$\gamma\mapsto-J\pi_{\xi}\dot{\gamma}\in\Gamma(\gamma^*\xi)$ is a section of  $E\to\mathcal{L}(M)$. This section \emph{resembles} a positive gradient for $\mathcal{A}_{\lambda}$, and vanishes at Reeb orbits}\\
 \hline
 \footnotesize{At each $p\in\text{Crit}(f)$, we have the self adjoint Hessian $H_p:T_pX\to T_pX$, obtained by linearizing $-\text{grad}(f)$ at $p$} & \footnotesize{At each Reeb orbit $\gamma$ we have the self adjoint asymptotic operator $A_{\gamma}:\Gamma(\gamma^*\xi)\to\Gamma(\gamma^*\xi)$, obtained by linearizing $-J\pi_{\xi}\dot{\gamma}$ at $\gamma$}\\
 \hline
 \footnotesize{$x\in\mathcal{M}(p_+,p_-)$ is a Morse flow line} & \footnotesize{$u\in\mathcal{M}(\gamma_+,\gamma_-)$ is a $J$-holomorphic cylinder}  \\ 
 \hline
 \footnotesize{$p$ nondegenerate $\iff$ $H_p$  nondegenerate;} \footnotesize{$\text{ind}(x)$ is related to the spectra of $H_{p_{\pm}}$ }& \footnotesize{ $\gamma$ nondegenerate $\iff$ $A_{\gamma}$ nondegenerate; \footnotesize{$\text{ind}(u)$ is related to the spectra of $A_{\gamma_{\pm}}$}} \\
 \hline
\end{tabular}
\caption{Analogies between the two theories}
\label{table: Morse SFT comparison}
\end{table}
Regarding the chain complexes and  differentials, there are two immediate similarities demonstrating that \emph{orbifold} Morse homology is the correct analogue of cylindrical contact homology:

\textbf{(1) The differentials are structurally identical.} Indeed, let us compare, for pairs $(p,q)$ and $(\gamma_+,\gamma_-)$ of orientable orbifold critical points and good Reeb orbits of grading difference equal to 1, 
\[\langle \partial^\text{orb} p, q\rangle=\sum_{x\in\mathcal{M}(p, q)}\epsilon(x)\dfrac{|H_p|}{|H_x|},\,\,\,\,\,\,\,\, \langle \partial \gamma_+, \gamma_-\rangle=\sum_{u\in\mathcal{M}_1^J(\gamma_+,\gamma_-)/\R}\epsilon(u)\dfrac{m(\gamma_+)}{m(u)}.\]
Recall that $|H_q|$ is the order of isotropy associated to the orbifold point $q$, $|H_x|$ is the order of isotropy associated to any point in the image of the flow line $x$, and $m$ denotes the multiplicity of either a Reeb orbit or a cylinder. Both differentials are counts of integers because $|H_x|$ divides $|H_p|$ and $m(u)$ divides $m(\gamma_+)$. The differentials use the same \emph{weighted} counts of 0-dimensional moduli spaces because of the similarities in the compactifications of the 1-dimensional  moduli spaces.  A single broken gradient path or holomorphic building may serve as the limit of \emph{multiple} ends of a 1-dimensional moduli space in either setting. For a thorough treatment of why these signed counts generally produce a differential that squares to zero, see \cite[Theroem 5.1]{CH} (in the orbifold case) and \cite[\S 4.3]{HN} (in the contact case).

For example, a 1-dimensional moduli space of flow lines/cylinders, diffeomorphic to an open interval $(0,1)$, could compactify in two geometrically distinct ways. Such an open interval's compactification by broken objects may be either a closed interval $[0,1]$, or  be a topological $S^1$. In the latter case, the single added point in the compactification serves as an endpoint, or limit, of \emph{two} ends of the open moduli space. The differential as a weighted count is designed to accommodate  this multiplicity. Cylinders and orbifold Morse flow lines appearing in our contact manifolds $S^3/G$ and orbifolds $S^2/H$ will exhibit this phenomenon of an open interval compactifying to a topological $S^1$, this is spelled out explicitly in Section \ref{subsection: visualizing holomorphic cylinders: an example} and illustrated in Figure \ref{figure: calzone}.

\begin{remark}
Compare these differentials to those in classical Morse theory and Hamiltonian Floer homology, which take the simpler forms $\sum_x\epsilon(x)$ and $\sum_u\epsilon(u)$, where $x$ ranges over Morse flow lines, and $u$ ranges over Hamiltonian trajectories in appropriate 0-dimensional moduli spaces. In both of these theories, the ends of the 1-dimensional moduli spaces are in \emph{bijective correspondence} with the broken trajectories, and so the differentials are only \emph{signed}, not weighted, counts.
\end{remark}

\textbf{(2) Bad Reeb orbits are analogous to non-orientable critical points.} Recall from Remark \ref{remark: nonorientable2}, that  in orbifold Morse theory the values $\epsilon(x)$, for $x\in\mathcal{M}(p,q)$, are only well defined when the flow line interpolates between \emph{orientable} critical points. The same phenomenon occurs in cylindrical contact homology. Loosely speaking, in  order to assign a coherent orientation to a holomorphic cylinder $u$, one needs to orient the asymptotic operators associated to the Reeb orbits $\gamma_{\pm}$ at both ends by means of a choice of base-point in the images of $\gamma_{\pm}$. The cyclic deck group of automorphisms of each Reeb orbit (isomorphic to $\Z_{m(\gamma_{\pm})}$) acts on the 2-set of orientations of the asymptotic operator. If the action is non-trivial (i.e., if the deck group \emph{reverses} the orientations), then the Reeb orbit is \emph{bad}, and there is not a well defined $\epsilon(u)$, independent of intermediate choices made.

In Section \ref{subsection: visualizing holomorphic cylinders: an example}, we explain that the Seifert projections $\fp:S^3/G\to S^2/H$ geometrically realize this analogy between non-orientable objects. Specifically, the bad Reeb orbits in $S^3/G$ are precisely the ones that project to non-orientable orbifold Morse critical points in $S^2/H$.

Finally, not only do non-orientable objects complicate the assignment of orientation to the 0-dimensional moduli spaces, their inclusion in the chain complex of both theories jeopardizes $\partial^2=0$, providing significant reasons to exclude them as generators. In Remark \ref{remark: when d squared is not zero} of Section \ref{subsection: visualizing holomorphic cylinders: an example}, we will investigate the $H=\T$ case, and will see (as we did in Remark \ref{remark: nonorientable2}) what happens if we were to include the non-orientable objects as generators of the respective chain complexes. In particular, we will produce examples of (orientable) generators $x$ and $y$ with $|x|-|y|=2$ with $\langle\partial x,y\rangle\neq 0$.

\section{Geometric setup and dynamics}\label{section: geometric setup}
We now begin our more detailed studies of the contact manifolds $S^3/G$. In this section we review a process of perturbing degenerate contact forms on $S^3$ and $S^3/G$ using a Morse function to achieve nondegeneracy up to an action threshold, following \cite[\S 1.5]{N2}. This process will aid our main goal of the section, to precisely identify the Reeb orbits of  $S^3/G$, and to compute their Conley-Zehnder indices (Lemma \ref{lemma: ActionThresholdLink}).

\subsection{Spherical geometry and associated Reeb dynamics}\label{subsection: spherical geometry and associated Reeb dynamics}
The round contact form on $S^3$, denoted $\lambda$, is defined as the restriction of the 1-form $\iota_{U}\omega_0\in\Omega^1(\C^2)$ to $S^3$, where $\omega_0$ is the standard symplectic form on $\C^2$ and $U$ is the radial vector field in $\C^2$. In complex coordinates, we have:
\[\omega_0=\frac{i}{2}\sum_{k=1}^2dz_k\wedge d\overline{z_k},\,\,\,U=\frac{1}{2}\sum_{k=1}^2z_k\partial_{z_k}+\overline{z_k}\partial_{\overline{z_k}},\,\,\,\lambda=\frac{i}{4}\sum_{k=1}^2z_k\wedge d\overline{z_k}-\overline{z_k}dz_k.\]  After interpreting $\C^2\cong\R^4$, we can rewrite $
\omega_0$, $U$, and $\lambda$ using real notation: \[\omega_0=\sum_{k=1}^2dx_k\wedge dy_k,\,\,\,\,\,\,\,\, U=\frac{1}{2}\sum_{k=1}^2x_k\partial_{x_k}+y_k\partial_{y_k},\,\,\,\,\,\,\,\,\lambda=\frac{1}{2}\sum_{k=1}^2x_kdy_k-y_kdx_k.\] The contact manifold $(S^3,\lambda)$ is a classical example of a contact-type hypersurface in the symplectic manifold $(\R^4,\omega_0)$; this means that $S^3$ is transverse to the \emph{Liouville vector field}, $U$, and inherits contact structure from the contraction of $\omega_0$ with $U$.

The diffeomorphism $S^3\subset\C^2\to \SU(2)$ provides $S^3$ with Lie group structure,
\begin{equation} \label{equation: S^3 Lie}
    (\alpha,\beta)\in S^3\mapsto\begin{pmatrix}\alpha &-\overline{\beta}\,\,\\ \beta & \overline{\alpha}\end{pmatrix}\in\SU(2),
\end{equation}
and we see that $e=(1,0)\in S^3$ is the identity element. The $\SU(2)$-action on $\C^2$ preserves  $\omega_0$ and $U$, and so the $\SU(2)$-action restricted to $S^3$ preserves $\lambda$. That is, $\SU(2)$ acts on $S^3$ by \emph{strict} contactomorphisms.

There is a natural Lie algebra isomorphism between the tangent space of the identity element of a Lie group and its collection of left-invariant vector fields.  The contact plane $\xi_e=\text{ker}(\lambda_e)$ at the identity element $e=(1,0)\in S^3$ is spanned by the tangent vectors $\partial_{x_2}|_e=\langle 0,0,1,0\rangle$ and $\partial_{y_2}|_e=\langle 0,0,0,1\rangle$, where we are viewing \[\xi_e\subset T_eS^3\subset T_e\C^2\cong T_e\R^4=\text{Span}_{\R}\big(\partial_{x_1}|_e,\,\partial_{y_1}|_e,\,\partial_{x_2}|_e,\,\partial_{y_2}|_e\big).\] Let $V_1$ and $V_2$ be the  left-invariant vector fields corresponding to $\partial_{x_2}|_e$ and  $\partial_{y_2}|_e$, respectively. Because $S^3$ acts on itself by \emph{contactomorphisms},  $V_1$ and $V_2$ are sections of $\xi$ and provide a global unitary trivialization of $(\xi, d\lambda|_{\xi},J_{\C^2})$, denoted $\tau$:
\begin{equation}\label{equation: global trivialization}
    \tau:S^3\times\R^2\to\xi, \,\,\, (p,\eta_1,\eta_2)\mapsto \eta_1V_1(p)+\eta_2V_2(p)\in\xi_p.
\end{equation}
Here, $J_{\C^2}$ is the standard integrable complex structure on $\C^2$. Note that $J_{\C^2}
(V_1)=V_2$ everywhere. Given a Reeb orbit $\gamma$ of any contact form on $S^3$, let  $\mu_{\CZ}(\gamma)\in \Z$ denote the Conley-Zehnder index of $\gamma$ with respect to this $\tau$, i.e. $\mu_{\CZ}(\gamma):=\mu_{\CZ}^{\tau}(\gamma)$. If $(\alpha,\beta)\in S^3$, write $\alpha=a+ib$ and $\beta=c+id$. Then, with respect to the ordered basis $(\partial_{x_1}, \partial_{y_1}, \partial_{x_2}, \partial_{y_2})$ of $T_{(\alpha,\beta)}\R^4$, we have the following expressions
\begin{equation}\label{equation: vector field coordinates}
    V_1(\alpha,\beta)=\langle -c,d, a, -b\rangle,\,\,\,\,
    V_2(\alpha, \beta)=\langle -d,-c,b,a\rangle.
\end{equation}

Consider the double cover of Lie groups, $P:\SU(2)\to\SO(3)$: \vspace{.2cm}\begin{equation}\label{equation: P in coordinates}
    \begin{pmatrix}\alpha & -\overline{\beta}\,\, \\ \beta & \overline{\alpha}\end{pmatrix}\in\SU(2)\xmapsto{P}\begin{pmatrix}|\alpha|^2-|\beta|^2 & 2\text{Im}(\alpha\beta) & 2\text{Re}(\alpha\beta) \\ -2\text{Im}(\overline{\alpha}\beta) & \text{Re}(\alpha^2+\beta^2) & -\text{Im}(\alpha^2+\beta^2)\\ -2\text{Re}(\overline{\alpha}\beta) & \text{Im}(\alpha^2-\beta^2) & \text{Re}(\alpha^2-\beta^2)\end{pmatrix}\in\SO(3).
\end{equation}
The kernel of $P$ has order 2 and is generated by $-\text{Id}\in\SU(2)$, the only element of $\SU(2)$ of order 2. A diffeomorphism $\C P^1\to S^2\subset\R^3$ is given in homogeneous coordinates ($|\alpha|^2+|\beta|^2=1$) by \begin{equation}\label{equation: sphere cp1 diffeomorphism}
    (\alpha:\beta)\in \C P^1\mapsto(|\alpha|^2-|\beta|^2,-2\text{Im}(\bar{\alpha}\beta),-2\text{Re}(\bar{\alpha}\beta))\in S^2.
\end{equation}
We have an $\SO(3)$-action on $\C P^1$, pulled back from the $\SO(3)$-action on $S^2$ by \eqref{equation: sphere cp1 diffeomorphism}. Lemma \ref{lemma: commutes} illustrates how the action of $\SU(2)$ on $S^3$ is related to the action of $\SO(3)$ on $\C P^1\cong S^2$ via $P:\SU(2)\to\SO(3).$

\begin{lemma} \label{lemma: commutes}
For a point $z$ in $S^3$, let $[z]\in\C P^1$ denote the corresponding point under the quotient of the $S^1$-action on $S^3$. Then for all $z\in S^3$, and all matrices $A\in\mbox{\em SU}(2)$, we have \[[A\cdot z]=P(A)\cdot[z]\in\C P^1\cong S^2.\]
\end{lemma}

\begin{proof}
First, note that the result holds for the case $z=e=(1,0)\in S^3$. This is because $[e]\in\C P^1$ corresponds to $(1,0,0)\in S^2$ under \eqref{equation: sphere cp1 diffeomorphism}, and so for any $A$, $P(A)\cdot[e]$ is the first column of the $3\times 3$ matrix $P(A)$ appearing in \eqref{equation: P in coordinates}. That is, \begin{equation}\label{equation: sphere points}
    P(A)\cdot[e]=(|\alpha|^2-|\beta|^2,-2\text{Im}(\bar{\alpha}\beta),-2\text{Re}(\bar{\alpha}\beta)),
\end{equation} (where $(\alpha,\beta)\in S^3$ is the unique element corresponding to $A\in \SU(2)$, using \eqref{equation: S^3 Lie}). By \eqref{equation: sphere cp1 diffeomorphism}, the point \eqref{equation: sphere points} equals $[(\alpha,\beta)]=[(\alpha,\beta)\cdot (1,0)]=[A\cdot e]$, and so the result holds when $z=e$.

For the general case, note that any $z\in S^3$ equals $B\cdot e$ for some $B\in\SU(2)$, and use the fact that $P$ is a group homomorphism.
\end{proof}

The Reeb flow of $\lambda$ is given by the $S^1\subset\C^*$ (Hopf) action, $z\mapsto e^{it}\cdot z$. Thus, all $\gamma\in\mathcal{P}(\lambda)$ have period $2k\pi$, with linearized return maps equal to $\text{Id}:\xi_{\gamma(0)}\to\xi_{\gamma(0)}$, and are degenerate.

\begin{notation}\label{notation: morse data}
Following the general recipe of perturbing the degenerate contact form on a prequantization bundle, outlined in \cite[\S 1.5]{N2}, we establish the following notation:
\begin{itemize}
     \itemsep-.35em
    \item $\fP:S^3\to S^2$  denotes the Hopf fibration,
    \item $f$ is a  Morse-Smale function on $(S^2,\omega_{\FS}(\cdot, j\cdot))$, $\text{Crit}(f)$ is its set of critical points,
    \item for $\varepsilon>0$;  $f_{\varepsilon}:=1+\varepsilon f:S^2\to \R$, $F_{\varepsilon}:=f_{\varepsilon}\circ\fP:S^3\to \R$, and $\lambda_{\varepsilon}:=F_{\varepsilon}\lambda\in\Omega^1(S^3)$,
    \item $\widetilde{X_f}\in\xi$ is the horizontal lift of $X_f\in TS^2$ using the fiberwise linear symplectomorphism $d\fP|_{\xi}:(\xi,d\lambda|_{\xi})\to (TS^2, \omega_{\text{FS}})$, where $X_f$ denotes the Hamiltonian vector field for $f$ on $S^2$ with respect to $\omega_{\text{FS}}$.
\end{itemize}
\end{notation}
\begin{remark}\label{remark: ham}
Our convention is that for a smooth, real valued function $f$ on symplectic manifold $(M,\omega)$, the Hamiltonian vector field $X_f$ uniquely satisfies $\iota_{X_f}\omega=-df$. 
\end{remark}

Because $F_{\varepsilon}$ is positive for small $\varepsilon$ and smooth on $S^3$, $\lambda_{\varepsilon}$ is another contact form on $S^3$ defining the same contact distribution $\xi$ as $\lambda$. We refer to $\lambda_{\varepsilon}$ as the \emph{perturbed contact form} on $S^3$. Although $\lambda$ and $\lambda_{\varepsilon}$ share the same kernel, their Reeb dynamics differ.

\begin{lemma} \label{lemma: reeb1}
The following relationship between vector fields on $S^3$ holds:
\[R_{\lambda_{\varepsilon}}=\frac{R_{\lambda}}{F_{\varepsilon}}-\varepsilon\frac{\widetilde{X_f}}{F_{\varepsilon}^2}.\]
\end{lemma}

\begin{proof}
This is \cite[Prop. 4.10]{N2}. Note that the sign discrepancy is a result of our convention regarding  Hamiltonian vector fields, see Remark \ref{remark: ham}.
\end{proof}

We now explore how the relationship between vector fields from Lemma \ref{lemma: reeb1} provides a relationship between Reeb and Hamiltonian flows. 
\begin{notation}
(Reeb and Hamiltonian flows). For any $t\in\R$,
\begin{itemize}
    \itemsep-.35em
    \item $\phi_0^t:S^3\to S^3$ denotes the time $t$ flow of the unperturbed Reeb vector field $R_{\lambda}$,
    \item $\phi^t:S^3\to S^3$ denotes the time $t$ flow of the perturbed Reeb vector field $R_{\lambda_{\varepsilon}}$,
    \item $\varphi^t:S^2\to S^2$ denotes the time $t$ flow of the vector field $V:=-\frac{\varepsilon X_{f}}{f_{\varepsilon}^2}.$
\end{itemize}
\end{notation}
Note that $\phi^t_0:S^3\to S^3$ is nothing more than the map $z\mapsto e^{it}\cdot z$.
\begin{lemma}\label{lemma: flow}
For all $t$ values, we have $\fP\circ\phi^t=\varphi^t\circ\fP$ as smooth maps $S^3\to S^2$.
\end{lemma}

\begin{proof}
 Pick $z\in S^3$ and let $\widetilde{\gamma}:\R\to S^3$ denote the unique integral curve for $R_{\lambda_{\varepsilon}}$ which passes through $z$ at time $t=0$, i.e., $\widetilde{\gamma}(t)=\phi^t(z)$.  By Lemma \ref{lemma: reeb1}, $d\fP$ carries the derivative of $\widetilde{\gamma}$ precisely to the vector $V\in TS^2$. Thus, $\fP\circ\widetilde{\gamma}:\R\to S^2$ is the unique integral curve, $\gamma$, of $V$ passing through $p:=\fP(z)$ at time $t=0$, i.e., $\gamma(t)=\varphi^t(p)$. Combining these facts provides \[\fP(\widetilde{\gamma}(t))=\gamma(t)\,\,\implies\,\,\fP(\phi^t(z))=\varphi^t(p)\,\,\implies\fP(\phi^t(z))=\varphi^t(\fP(z)).\]
\end{proof}

Lemma \ref{lemma: orbits} describes the orbits $\gamma\in\mathcal{P}(\lambda_{\varepsilon})$ projecting to critical points of $f$ under $\fP$.
\begin{lemma}\label{lemma: orbits}
Let $p\in\mbox{\em Crit}(f)$ and take $z\in\fP^{-1}(p)$. Then the map \[\gamma_p:[0,2\pi f_{\varepsilon}(p)]\to S^3,\,\,\,\,t\mapsto e^{\frac{it}{f_{\varepsilon}(p)}}\cdot z\]
descends to a closed, embedded Reeb orbit $\gamma_p:\R/2\pi f_{\varepsilon}(p)\Z\to S^3$ of $\lambda_{\varepsilon}$, passing through point $z$ in $S^3$, whose image under $\fP$ is $\{p\}\subset S^2$, where $\cdot$ denotes the $S^1\subset\C^*$ action on $S^3$.
\end{lemma}

\begin{proof}
The map $\R/2\pi\Z\to S^3$, $t\mapsto e^{it}\cdot z$ is a closed, embedded integral curve for the degenerate Reeb field $R_{\lambda}$, and so by the chain rule we have that $\dot{\gamma_p}(t)=R_{\lambda}(\gamma_p(t))/f_{\varepsilon}(p)$. Note that $\fP(\gamma(t))=\fP(e^{\frac{it}{f_{\varepsilon}(p)}}\cdot z)=\fP(z)=p$ and, because $\widetilde{X_f}(\gamma(t))$ is a lift of $X_f(p)=0$, we have $\widetilde{X_f}(\gamma(t))=0$. By the description of $R_{\lambda_{\varepsilon}}$ in Lemma \ref{lemma: reeb1}, we have $\dot{\gamma_p}(t)=R_{\lambda_{\varepsilon}}(\gamma_p(t))$.
\end{proof}

 Next we set notation to be used in describing the local models for the linearized Reeb flow along the orbits $\gamma_p$ from Lemma \ref{lemma: orbits}. For $s\in\R$, $\mathcal{R}(s)$ denotes the $2\times2$ rotation matrix: \[\mathcal{R}(s):=\begin{pmatrix}\cos{(s)} & -\sin{(s)} \\ \sin{(s)} & \cos{(s)}\end{pmatrix}\in\SO(2).\]
Note that $J_0=\mathcal{R}(\pi/2)\in\SO(2)$. For $p\in\text{Crit}(f)\subset S^2$, pick coordinates $\psi:\R^2\to S^2$, so that $\psi(0,0)=p$. Then we let $H(f,\psi)$ denote the Hessian of $f$ in these coordinates at $p$ as a $2\times2$ matrix of second partial derivatives.

\begin{notation}\label{notation: stereographic}
The term \emph{stereographic coordinates at $p\in S^2$} describes a smooth $\psi:\R^2\to S^2$ with $\psi(0,0)=p$, that has a factorization $\psi=\psi_1\circ\psi_0$, where $\psi_0:\R^2\to S^2$ is the map \[(x,y)\mapsto\frac{1}{1+x^2+y^2}(2x,2y,1-x^2-y^2),\] taking $(0,0)$ to $(0,0,1)$, and $\psi_1:S^2\to S^2$ is given by the action of some element of $\SO(3)$ taking $(0,0,1)$ to $p$. If $\psi$ and $\psi'$ are both stereographic coordinates at $p$, then they differ by a precomposition with some $\mathcal{R}(s)$ in $\SO(2)$. Note that $\psi^*\omega_{\text{FS}}=\frac{dx\wedge dy}{(1+x^2+y^2)^2}$ (\cite[Ex. 4.3.4]{MS}). 
\end{notation}

Lemma \ref{lemma: rotate} describes the linearized Reeb flow of the unperturbed $\lambda$ with respect to $\tau$.

\begin{lemma} \label{lemma: rotate}
For any $z\in S^3$, the linearization $d\phi_0^t|_{\xi_z}:\xi_z\to\xi_{e^{it}\cdot z}$ is represented by $\mathcal{R}(2t)$, with respect to ordered bases $(V_1(z),V_2(z))$ of $\xi_z$ and $(V_1(e^{it}\cdot z), V_2(e^{it}\cdot z))$ of $\xi_{e^{it}\cdot z}$.
\end{lemma}

\begin{proof}
Because (i) the  $V_i$ are $\SU(2)$-invariant, and (ii) the $\SU(2)$-action commutes with $\phi_0^t$, we may reduce to the case $z=e=(1,0)\in S^3$. That is, we must show
for all $t$ values that
\begin{align}
    (d\phi_0^t)_{(1,0)}V_1(1,0)&=\cos{(2t)}V_1(e^{it},0)+\sin{(2t)}V_2(e^{it},0) \label{equation: eq1}\\
    (d\phi_0^t)_{(1,0)}V_2(1,0)&=-\sin{(2t)}V_1(e^{it},0)+\cos{(2t)}V_2(e^{it},0). \label{equation: eq2}
\end{align}
Note that \eqref{equation: eq2}  follows from  \eqref{equation: eq1} by applying the endomorphism $J_{\C^2}$ to both sides of \eqref{equation: eq1}, and noting both that $d\phi_0^t$ commutes with $J_{\C^2}$, and that $J_{\C^2}(V_1)=V_2$. We now prove \eqref{equation: eq1}.

The coordinate descriptions \eqref{equation: vector field coordinates} tell us that $V_1(e^{it},0)$ and $V_2(e^{it},0)$ can be respectively written as $\langle 0,0,\cos{t},-\sin{t}\rangle$ and $\langle 0,0,\sin{t},\cos{t}\rangle$. Angle sum formulas now imply \[\cos{(2t)}V_1(e^{it},0)+\sin{(2t)}V_2(e^{it},0)=\langle 0, 0, \cos{t},\sin{t}\rangle.\]
The vector on the right is precisely $(d\phi_0^t)_{(1,0)}V_1(1,0)$, and so we have proven \eqref{equation: eq1}.
\end{proof}

Proposition \ref{proposition: flow} and Corollary \ref{corollary: CZ sphere} conclude our discussion of dynamics on $S^3$.

\begin{proposition}\label{proposition: flow}
Fix a critical point $p$ of $f$ in $S^2$ and stereographic coordinates $\psi:\R^2\to S^2$ at $p$, and suppose $\gamma\in\mathcal{P}(\lambda_{\varepsilon})$ projects to $p$ under $\fP$. Let $\Phi_t\in\mbox{\em Sp}(2)$ denote $d\phi^t|_{\xi_{\gamma(0)}}:\xi_{\gamma(0)}\to\xi_{\gamma(t)}$ with respect to the trivialization $\tau$. Then $\Phi_t$ is a a conjugate of the matrix  \[\mathcal{R}\bigg(\frac{2t}{f_{\varepsilon}(p)}\bigg)\cdot\exp\bigg(\frac{-t\varepsilon}{f_{\varepsilon}(p)^2} J_0\cdot H(f,\psi)\bigg)\] by some  element of $\mbox{\em SO}(2)$, which is independent of $t$.
\end{proposition}

\begin{proof}
Let $z:=\gamma(0)$. We linearize the identity $\fP\circ\phi^t=\varphi^t\circ\fP$ (Lemma \ref{lemma: flow}), restrict to $\xi_z$, and rearrange to recover the equality $d\phi^t|_{\xi_z}=a\circ b\circ c:\xi_z\to\xi_{\phi^t(z)}$, where \[a=(d\fP|_{\xi_{\phi^t(z)}})^{-1}:T_pS^2\to\xi_{\phi^t(z)},\,\, b=d\varphi^t_p:T_pS^2\to T_pS^2,\,\,\text{and}\,\, c=d\fP|_{\xi_z}:\xi_z\to T_pS^2.\]

Let $v_i:=d\fP (V_i(z))\in T_pS^2$, and see that the vector pairs $(v_1,v_2)$ and $(V_1,V_2)$ provide an oriented basis of each of the three vector spaces appearing in the above composition of linear maps. Let $A$, $B$, and $C$ denote the matrix representations of $a$, $b$, and $c$ with respect to these ordered bases. We have $\Phi_t=A\cdot B\cdot C$. Note that $C=\text{Id}$. We compute $A$ and $B$:

To compute $A$, recall that $\phi^t_0:S^3\to S^3$  denotes the time $t$ flow of the unperturbed Reeb field (alternatively, the Hopf action). Linearize the equality $\fP\circ\phi^t_0=\fP$, then use $\phi_0^t(z)=\phi^{t/f_{\varepsilon}(p)}(z)$ (Lemma \ref{lemma: orbits}) and  Lemma \ref{lemma: rotate} to find that
\begin{equation}\label{equation A_t}A=\mathcal{R}(2t/f_{\varepsilon}(p)).\end{equation}

To compute $B$, note that $\varphi^t$ is the Hamiltonian flow of the function $1/f_{\varepsilon}$ with respect to $\omega_{\FS}$. Is is advantageous to study this flow using our stereographic coordinates: recall that $\psi$ defines a symplectomorphism $(\R^2,\frac{dx\wedge dy}{(1+x^2+y^2)^2})\to(S^2\setminus\{p'\},\omega_{\FS})$, where $p'\in S^2$ is antipodal to $p$ in $S^2$ (see Notation \ref{notation: stereographic}). Because symplectomorphisms preserve Hamiltonian data, we see that $\psi^{-1}\circ\varphi^t\circ\psi:\R^2\to\R^2$ is  given near the origin as the time $t$ flow of the Hamiltonian vector field for $\psi^*(1/f_{\varepsilon})$ with respect to $\frac{dx\wedge dy}{(1+x^2+y^2)^2}$. That is, \[\psi^{-1}\circ\varphi^t\circ\psi\,\,\,\, \text{is the time $t$ flow of} \,\,\,\,\frac{\varepsilon (1+x^2+y^2)^2f_y}{f_{\varepsilon}^2}\partial_x-\frac{\varepsilon (1+x^2+y^2)^2f_x}{f_{\varepsilon}^2}\partial_y.\]
Recall that if $P\partial_x+Q\partial_y$ is a smooth vector field on $\R^2$, vanishing at $(0,0)$, then the linearization of its time $t$ flow evaluated at the origin is represented by the $2\times2$ matrix $\exp(t X)$, with respect to the standard ordered basis $(\partial_x,\partial_y)$ of $T_{(0,0)}\R^2$, where \[X=\begin{pmatrix}P_x & P_y \\ Q_x & Q_y \end{pmatrix}.\]
Here, the partial derivatives of $P$ and $Q$ are implicitly assumed to be evaluated at the origin. In this spirit, set $P=\frac{\varepsilon (1+x^2+y^2)^2f_y}{f_{\varepsilon}^2}$ and $Q=\frac{-\varepsilon (1+x^2+y^2)^2f_x}{f_{\varepsilon}^2}$, and we compute \[X=\begin{pmatrix}P_x & P_y \\ Q_x & Q_y \end{pmatrix}=\frac{\varepsilon}{f_{\epsilon}(p)^2}\begin{pmatrix}f_{yx}(0,0) & f_{yy}(0,0) \\ -f_{xx}(0,0) & -f_{xy}(0,0)\end{pmatrix}=\frac{-\varepsilon}{f_{\varepsilon}(p)^2}J_0 H(f,\psi).\]
This implies that $B=D^{-1}\exp\big(\frac{-t\varepsilon}{f_{\varepsilon}(p)^2}J_0 H(f,\psi)\big)D$, where $D$ is a change of basis matrix relating $(v_1,v_2)$ and the pushforward of $(\partial_x,\partial_y)$ by $\psi$ in $T_pS^2$. Because $\psi$ is holomorphic, $D$ must equal $r\cdot\mathcal{R}(s)$ for some $r>0$ and some $s\in\R$. This provides that \begin{equation}\label{equation: B_t}B=\mathcal{R}(-s)\exp\bigg(\frac{-t\varepsilon}{f_{\varepsilon}(p)^2}J_0 H(f,\psi)\bigg)\mathcal{R}(s).\end{equation}

Finally, we combine \eqref{equation A_t} and \eqref{equation: B_t} to conclude
\begin{align*}
    \Phi_t=A\cdot B&=\mathcal{R}\bigg(\frac{2t}{f_{\varepsilon}(p)}\bigg)\cdot \mathcal{R}(-s)\exp\bigg(\frac{-t\varepsilon}{f_{\varepsilon}(p)^2}J_0 H(f,\psi)\bigg) \mathcal{R}(s) \\
    &=\mathcal{R}(-s)\mathcal{R}\bigg(\frac{2t}{f_{\varepsilon}(p)}\bigg)\exp\bigg(\frac{-t\varepsilon}{f_{\varepsilon}(p)^2}J_0 H(f,\psi)\bigg)\mathcal{R}(s).
\end{align*}
\end{proof}

\begin{corollary}\label{corollary: CZ sphere}
Fix $L>0$. Then there exists some $\varepsilon_0>0$ such that for $\varepsilon\in(0,\varepsilon_0]$, all Reeb orbits $\gamma\in\mathcal{P}^L(\lambda_{\varepsilon})$ are nondegenerate, take the form $\gamma_p^k$, a $k$-fold cover of an embedded Reeb orbit $\gamma_p$ as in in Lemma \ref{lemma: orbits}, and $\mu_{\CZ}(\gamma)=4k+\mbox{\em ind}_{f}(p)-1$
for $p\in\mbox{\em Crit}(f)$.
\end{corollary}

\begin{proof}
That there exists an $\varepsilon_0>0$ ensuring $\gamma$ is nondegenerate and projects to a critical point $p$ of $f$ whenever $\gamma\in\mathcal{P}^L(\lambda_{\varepsilon})$, for $\varepsilon\in(0,\varepsilon_0]$, is proven in \cite[Lemma 4.11]{N2}. To compute $\mu_{\CZ}(\gamma)$, we apply the naturality, loop, and signature properties of the Conley-Zehnder index (see Remark \ref{remark: properties of cz}) to our path $\Phi=\{\Phi_t\}\subset\Sp(2)$ from Proposition \ref{proposition: flow}. This family of matrices has a factorization, up to $\SO(2)$-conjugation,  $\Phi_t=L_t\Psi_t$, where $L$ is the loop of symplectic matrices $\R/2\pi kf_{\varepsilon}(p)\Z\to\Sp(2)$,  $t\mapsto\mathcal{R}(2t/f_{\varepsilon}(p))$, and $\Psi$ is the path of matrix exponentials $t\mapsto \exp(\frac{-t\varepsilon}{f_{\varepsilon}(p)^2}J_0 H(f,\psi))$, where $\psi$ denotes a choice of stereographic coordinates at $p$. In total,
\begin{align}
    \mu_{\CZ}(\gamma_p^k)&=\mu_{\CZ}(\Phi)\nonumber\\
    &=\mu_{\CZ}(L\Psi)\nonumber\\
    &=2\mu(L)+\mu_{\CZ}(\Psi)\label{equation: cz computation 1}\\
    &=2\cdot 2k+\frac{1}{2}\text{Sign}(-H(f,\psi))\label{equation: cz computation 2}\\
    &=4k+\text{ind}_f(p)-1.\label{equation: cz computation 3}
\end{align}
In line \eqref{equation: cz computation 1} we use the Loop property from Remark \ref{remark: properties of cz}, in line \eqref{equation: cz computation 2} we use the Signature property, and in line \eqref{equation: cz computation 3} we use that $\text{Sign}(H(f,\psi))=2-2\text{ind}_f(p)$. Additionally, $\mu$ in line \eqref{equation: cz computation 1} denotes the Maslov index of a loop of symplectic matrices, so that $\mu(L)=2k$ (see Remark \ref{remark: properties of cz} and \cite[\S 2]{MS}).
\end{proof}

\subsection{Geometry of $S^3/G$ and associated Reeb dynamics} \label{subsection: sfs geometry}
The process of perturbing a degenerate contact form on prequantization bundles, outlined in Section \ref{subsection: spherical geometry and associated Reeb dynamics}, is often used to compute Floer theories, for example, their embedded contact homology \cite{NW}. Although the quotients $S^3/G$ are not prequantization bundles, they do admit an $S^1$-action (with fixed points), and are examples of \emph{Seifert fiber spaces}.

Let $G\subset\SU(2)$ be a finite nontrivial group. This $G$ acts on $S^3$ without fixed points, thus $S^3/G$ inherits smooth structure. Because the quotient $\pi_G:S^3\to S^3/G$ is a universal cover, we see that $\pi_1(S^3/G)\cong G$ is completely torsion, and  $\beta_1(S^3/G)=0$. Because the $G$-action preserves $\lambda\in\Omega^1(S^3)$, we have a descent of $\lambda$ to a contact form on $S^3/G$, denoted $\lambda_G\in\Omega^1(S^3/G)$, with $\xi_G:=\text{ker}(\lambda_G)$. Because the actions of $S^1$ and $G$ on $S^3$ commute, we have an $S^1$-action on $S^3/G$, describing the Reeb flow of $\lambda_G$. Hence, $\lambda_G$ is degenerate.

Let $H\subset\SO(3)$ denote $P(G)$, the image of $G$ under $P:\SU(2)\to \SO(3)$. The $H$-action on $S^2$ has fixed points, and so the quotient $S^2/H$ inherits \emph{orbifold} structure. Lemma \ref{lemma: commutes} provides a unique map $\fp:S^3/G\to S^2/H$, making the following diagram commute

\begin{equation}\label{diagram: commuting square}\begin{tikzcd}
S^3 \arrow{r}{\pi_G} \arrow{d}{\fP}
& S^3/G \arrow{d}{\fp} \\
S^2 \arrow{r}{\pi_H}
& S^2/H
\end{tikzcd}
\end{equation}
where $\pi_G$ is a finite universal cover, $\pi_H$ is an orbifold cover, $\fP$ is a projection of a prequantization bundle, and $\fp$ is identified with the Seifert fibration.

\begin{remark}
(Global trivialization of $\xi_G$) Recall the $\SU(2)$-invariant vector fields $V_i$ spanning $\xi$ on $S^3$ \eqref{equation: global trivialization}. Because these $V_i$ are $G$-invariant, they descend to smooth sections of $\xi_G$, providing a global symplectic trivialization, $\tau_G$, of $\xi_G$, hence $c_1(\xi_G)=0$. Given a Reeb orbit $\gamma$ of some contact form on $S^3/G$,  $\mu_{\CZ}(\gamma)$ denotes the Conley-Zehnder index of $\gamma$ with respect to this global trivialization. That is, $\mu_{\CZ}(\gamma):=\mu_{\CZ}^{\tau_G}(\gamma)$.
\end{remark}

We now additionally assume that the Morse function $f:S^2\to\R$ from Section \ref{subsection: spherical geometry and associated Reeb dynamics} is $H$-invariant. This invariance provides a descent of $f$ to an orbifold Morse function, $f_H:S^2/H\to\R$, in the language of \cite{CH}. Furthermore, the $H$-invariance of $f$ provides that the smooth $F=f\circ\fP$ is $G$-invariant, and descends to a smooth function, $F_G:S^3/G\to\R$. We define, analogously to Notation \ref{notation: morse data}, \[f_{H,\varepsilon}:=1+\varepsilon f_H,\hspace{1.5cm} F_{G,\varepsilon}:=1+\varepsilon F_G,\hspace{1.5cm}\lambda_{G,\varepsilon}:=F_{G,\varepsilon}\lambda_G.\] For sufficiently small $\varepsilon$, $\lambda_{G,\varepsilon}$ is a contact form on $S^3/G$ with kernel $\xi_G$. The condition $\pi_G^*\lambda_{G,\varepsilon}=\lambda_{\varepsilon}$ implies that $\gamma:[0,T]\to S^3$ is an integral curve of $R_{\lambda_{\varepsilon}}$ if and only if $\pi_G\circ\gamma:[0,T]\to S^3/G$ is an integral curve of $R_{\lambda_{G,\varepsilon}}$.

\begin{remark}\label{remark: same local models} 
(Local models and Conley-Zehnder indices on $S^3$ and $S^3/G$ agree) Suppose $\gamma:[0,T]\to S^3/G$ is an integral curve of $R_{\lambda_{G,\varepsilon}}$, and let $\widetilde{\gamma}:[0,T]\to S^3$ be a $\pi_G$-lift of $\gamma$ so that $\widetilde{\gamma}$ is an integral curve of   $R_{\lambda_{\varepsilon}}$. For $t\in[0,T]$, let $\Phi_t\in\Sp(2)$ denote the time $t$ linearized Reeb flow of $\lambda_{\varepsilon}$ along $\widetilde{\gamma}$ with respect to $\tau$, and let $\Phi^G_t\in\Sp(2)$ denote that of $\lambda_{G,\varepsilon}$ along $\gamma$ with respect to $\tau_G$. Then $\Phi_t=\Phi^G_t$, because the local contactomorphism $\pi_G$ preserves the trivializing vector fields in addition to the contact forms. Thus, if $\gamma$ is a Reeb orbit in $S^3/G$ and $\widetilde{\gamma}$ is a lift to $S^3$, we have that $\mu_{\CZ}(\gamma)=\mu_{\CZ}(\Phi^G)=\mu_{\CZ}(\Phi)=\mu_{\CZ}(\widetilde{\gamma})$, and Corollary \ref{corollary: CZ sphere} computes the right hand side in many cases.
\end{remark}

For any $p\in S^2$, let $\mathcal{O}(p):=\{h\cdot p|\,h\in H\}$ denote the  \emph{orbit} of $p$, and $H_p\subset H$ denotes the \emph{isotropy subgroup} of $p$, the set $\{h\in H|\, h\cdot p=p\}$. Recall that $|\mathcal{O}(p)||H_p|=|H|$ for any $p$. A point $p\in S^2$ is a \emph{fixed point} if $|H_p|>1$. The set of fixed points of $H$ is $\text{Fix}(H)$. The point  $q\in S^2/H$ is an \emph{orbifold point} if $q=\pi_H(p)$ for some $p\in\text{Fix}(H)$. We now additionally assume that $f$ satisfies $\text{Crit}(f)=\text{Fix}(H)$, (this will be the case in  Section \ref{section: computation of filtered contact homology}). The Reeb orbit $\gamma_p\in\mathcal{P}(\lambda_{\varepsilon})$ from Lemma \ref{lemma: orbits} projects to $p\in\text{Crit}(f)$ under $\fP$, and thus $\pi_G\circ\gamma_p\in\mathcal{P}(\lambda_{G,\varepsilon})$ projects to the orbifold point $\pi_H(p)$ under $\fp$. Lemma \ref{lemma: covering multiplicity} computes $m(\pi_G\circ\gamma_p)$.

\begin{lemma}\label{lemma: covering multiplicity}
Let $\gamma_p\in\mathcal{P}(\lambda_{\varepsilon})$ be the embedded Reeb orbit in $S^3$ from Lemma \ref{lemma: orbits}. Then the multiplicity of $\pi_G\circ\gamma_p\in\mathcal{P}(\lambda_{G,\varepsilon})$  is $2|H_p|$ if $|G|$ is even, and is $|H_p|$ if $|G|$ is odd.
\end{lemma}

\begin{proof}
Recall that $|G|$ is even if and only if $|G|=2|H|$, and that $|G|$ is odd if and only if $|G|=|H|$ (the only element of $\SU(2)$ of order 2 is $-\text{Id}$, the generator of $\text{ker}(P)$). By the classification of finite subgroups of $\SU(2)$, if $|G|$ is odd, then $G$ is cyclic. 

Let $q:=\pi_H(p)\in S^2/H$, let $d:=|\mathcal{O}(p)|$ so that $d|H_p|=|H|$ and $|G|=rd|H_p|$, where $r=2$ when $|G|$ is even and $r=1$ when $|G|$ is odd. Label the points of $\mathcal{O}(p)$ by $p_1=p, p_2, \dots, p_d$. Now $\fP^{-1}(\mathcal{O}(p))$ is a disjoint union of $d$ Hopf fibers, $C_i$, where $C_i=\fP^{-1}(p_i)$. Let $C$  denote the embedded circle $\fp^{-1}(q)\subset S^3/G$. By commutativity of   \eqref{diagram: commuting square}, we have that $\pi_G^{-1}(C)=\sqcup_{i}C_i=\fP^{-1}(\mathcal{O}(p))$. We have the following commutative diagram of circles and points:
\[\begin{tikzcd}
\sqcup_{i=1}^dC_i=\pi_G^{-1}(C) \arrow{r}{\pi_G} \arrow{d}{\fP}
& C=\fp^{-1}(q) \arrow{d}{\fp} \\
\{p_1,\dots, p_d\}=\mathcal{O}(p) \arrow{r}{\pi_H}
& \{q\}
\end{tikzcd}\]
We must have that $\pi_G:\sqcup_{i=1}^dC_i\to C$ is a $|G|=rd|H_p|$-fold cover from the disjoint union of $d$ Hopf fibers to one embedded circle. The restriction of $\pi_G$ to any one of these circles $C_i$ provides a smooth covering map, $C_i\to C$; let $n_i$ denote the degree of this cover.  Because $G$ acts transitively on these circles, all of the degrees $n_i$ are equal to some $n$. Thus, $\sum_{i}^dn_i=dn=|G|=rd|H_p|$ which implies that $n=r|H_p|$ is the covering multiplicity of $\pi_G\circ\gamma_p$.
\end{proof}

We conclude this section with an analogue of Corollary $\ref{corollary: CZ sphere}$ for the Reeb orbits of $\lambda_{G,\varepsilon}$.

\begin{lemma}\label{lemma: ActionThresholdLink}
Fix $L>0$. Then there exists some $\varepsilon_0>0$ such that, for  $\varepsilon\in(0,\varepsilon_0]$, all Reeb orbits  $\gamma\in\mathcal{P}^L(\lambda_{G,\varepsilon})$ are nondegenerate, project to an orbifold critical point of $f_{H}$ under $\fp:S^3/G\to S^2/H$, and $\mu_{\CZ}(\gamma)=4k+\mbox{\em ind}_f(p)-1$ whenever $\gamma$ is contractible with a lift to some orbit $\gamma_p^k$ in $S^3$ as in Lemma \ref{lemma: orbits}, where $p\in\mbox{\em Crit}(f)$.
\end{lemma}

\begin{proof}
Let $L':=|G|L$ and take the corresponding $\varepsilon_0$ as appearing in Corollary \ref{corollary: CZ sphere}, applied to $L'$. Now, for $\varepsilon\in(0,\varepsilon_0]$, elements of $\mathcal{P}^{L'}(\lambda_{\varepsilon})$ are nondegenerate and project to critical points of $f$. Let $\gamma\in\mathcal{P}^L(\lambda_{G,\varepsilon})$ and let $n\in\N$ denote the order of $[\gamma]$ in $\pi_1(S^3/G) \cong G$. Now we see that $\gamma^n$ is contractible and lifts to an orbit $\widetilde{\gamma}\in\mathcal{P}^{L'}(\lambda_{\varepsilon})$, which must be  nondegenerate and must project to some critical point $p$ of $f$ under $\fP$. 

If the orbit $\gamma$ is degenerate, then $\gamma^n$ is degenerate and, by the discussion in Remark \ref{remark: same local models}, $\widetilde{\gamma}$ would be degenerate. Commutativity of  \eqref{diagram: commuting square} implies that $\gamma$ projects to the orbifold critical point $\pi_H(p)$ of $f_H$ under $\fp$. Finally, if $n=1$ then again by Remark \ref{remark: same local models}, the local model $\{\Phi^G_t\}$ of the Reeb flow along $\gamma$ matches that of $\widetilde{\gamma}$, $\{\Phi_t\}$, and thus $\mu_{\CZ}(\gamma)=\mu_{\CZ}(\widetilde{\gamma})$. The latter index is computed in Corollary \ref{corollary: CZ sphere}.

\end{proof}

\subsection{Cylinders over orbifold Morse trajectories}\label{subsection: cylinders over orbifold Morse trajectories}

In Section \ref{section: computation of filtered contact homology} we will compute the action filtered cylindrical contact homology groups using the above computations and models using specific choices of Morse functions $f$. We will see that the grading of any generator of the filtered chain complex is even\footnote{Recall that the degree of a generator $\gamma$ is given by $|\gamma|=\mu_{\CZ}(\gamma)-1$.}, implying that the differential \emph{vanishes} for our perturbing functions $f_H$. 

However, this does not mean that the moduli spaces of holomorphic cylinders are  empty. In this section, we spell out a correspondence between moduli spaces of holomorphic cylinders and the moduli spaces of \emph{orbifold} Morse trajectories in the base; the latter is often nonempty. This is the orbifold version of the correspondence between cylinders and flow lines in the context of prequantization bundles, established in \cite[\S 5]{N2}. The reader should consider this a bonus section, as nothing presented here is necessary in proving Theorem \ref{theorem: main}, but is helpful for visualization purposes.

As discussed in Section \ref{subsection: cylindrical contact homology as an analogue of orbifold Morse homology}, the Seifert projection $\fp:S^3/G\to S^2/H$ highlights many of the interplays between orbifold Morse theory and cylindrical contact homology. In particular, the projection geometrically relates holomorphic cylinders in $\R\times S^3/G$ to the orbifold Morse trajectories in $S^2/H$. This necessitates a discussion about the complex structure on $\xi_{G}$ that we will use.

\begin{remark}\label{remark: the specific almost complex structure}(The canonical complex structure $J$ on $\xi_{G}$) Recall that the $G$-action on $S^3$ preserves the standard complex structure $J_{\C^2}$ on $\xi\subset TS^3$. Thus, $J_{\C^2}$ descends to a complex structure on $\xi_G$, which we will denote simply by $J$ in this section. Note that for any (small) $\varepsilon$, and for any $f$ on $S^2$, that $J_{\C^2}$ is $\lambda_{\varepsilon}$-compatible, and so $J$ is $\lambda_{G,\varepsilon}$-compatible as well.  This $J$ might not be one of the generic $J_N$ used to compute the filtered homology groups in the later Sections \ref{subsection: cyclic}, \ref{subsection: dihedral}, or \ref{subsection: polyhedral}; this is fine. A generic choice of $J_N$ is necessary in ensuring tranversality of cylinders appearing in symplectic cobordisms used to define the chain maps later in Section \ref{section: direct limits of filtered homology}, but has no impact on the results of this discussion. We use this $J$ for the rest of this section due to its compatibility with $J_{\C^2}$ on $S^3$ and with $j$ on $S^2$.
\end{remark}

\begin{remark}\label{remark: cylinders over flow lines}($J_{\C^2}$-holomorphic cylinders in $\R\times S^3$ over Morse flow lines in $S^2$) 

Fix critical points $p$ and $q$ of a Morse-Smale function $f$ on $(S^2,\omega_{\FS}(\cdot,j\cdot))$. Fix $\varepsilon>0$ sufficiently small. By \cite[Propositions 5.4, 5.5]{N2}, we have a bijective correspondence between $\mathcal{M}(p,q)$ and $\mathcal{M}^{J_{\C^2}}(\gamma_p^k,\gamma_q^k)/\R$, for any $k\in\N$, where $\gamma_p$ and $\gamma_q$ are the embedded Reeb orbits in $S^3$ projecting to $p$ and $q$ under $\fP$. Given a Morse trajectory $x\in\mathcal{M}(p,q)$, the components of the corresponding cylinder $u_x:\R\times S^1\to\R\times S^3$ are explicitly written down in \cite[\S 5]{N2} in terms of a parametrization  of $x$, the Hopf action on $S^3$, the Morse function $f$, and the horizontal lift of its gradient to $\xi$. The resultant $u_x\in\mathcal{M}^{J_{\C^2}}(\gamma_p,\gamma_q)/\R$\footnote{We are abusing notation by conflating the parametrized cylindrical map $u_x$ with the equivalence class $[u_x]\in\mathcal{M}^{J_{\C^2}}(\gamma_p,\gamma_q)/\R$; we will continue to abuse notation in this way.} is $J_{\C^2}$-holomorphic. Furthermore, the Fredholm index of $u_x$ agrees with that of $x$. The image of the composition  \[\R\times S^1\xrightarrow{u_x}\R\times S^3\xrightarrow{\pi_{S^3}}S^3\xrightarrow{\fP}S^2\] equals the image of $x$ in $S^2$. We call $u_x$ the \emph{cylinder over} $x$. \end{remark}

The following procedure uses Remark \ref{remark: cylinders over flow lines} and Diagram \ref{diagram: commuting square} to establish a similar correspondence between moduli spaces of \emph{orbifold} flow lines of $S^2/H$ and moduli spaces of $J$-holomorphic cylinders in $\R\times S^3/G$, where $J$ is taken to be the $J_{\C^2}$-descended complex structure on $\xi_{G}$ from Remark \ref{remark: the specific almost complex structure}.
\begin{enumerate}
    \item Take $x\in\mathcal{M}(p,q)$, for orbifold Morse critical points $p, q\in S^2/H$ of $f_H$.
    \item Take a $\pi_H$-lift, $\widetilde{x}:\R\to S^2$ of $x$, from $\widetilde{p}$ to $\widetilde{q}$ in $S^2$, for some preimages $\widetilde{p}$ and $\widetilde{q}$ of $p$ and $q$. We have $\widetilde{x}\in\mathcal{M}(\widetilde{p},\widetilde{q})$.
    \item Let $u_{\widetilde{x}}$ be the $J_{\C^2}$-holomorphic cylinder in $\R\times S^3$ over $\widetilde{x}$ (see Remark \ref{remark: cylinders over flow lines}). We now have $u_{\widetilde{x}}\in\mathcal{M}^{J_{\C^2}}(\gamma_{\widetilde{p}},\gamma_{\widetilde{q}})/\R$.
    \item Let $u_x:\R\times S^1\to\R\times S^3/G$ denote the composition \[\R\times S^1\xrightarrow{u_{\widetilde{x}}}\R\times S^3\xrightarrow{\text{Id}\times\pi_G}\R\times S^3/G.\] Because $J$ is the $\pi_G$-descent of $J_{\C^2}$, we have that \[\text{Id}\times\pi_G:(\R\times S^3,J_{\C^2})\to(\R\times S^3/G,J)\] is a holomorphic map. This implies that $u_x$ is $J$-holomorphic; \begin{equation}\label{equation: induced cylinder}
        u_x\in\mathcal{M}^J(\pi_G\circ\gamma_{\widetilde{p}},\pi_G\circ\gamma_{\widetilde{q}})/\R.
    \end{equation}
\end{enumerate}
Note that $\pi_G\circ\gamma_{\widetilde{p}}$ and $\pi_G\circ\gamma_{\widetilde{q}}$ are contractible Reeb orbits of $\lambda_{G,\varepsilon}$ projecting to $p$ and $q$, respectively. Thus, if $\gamma_p$ and $\gamma_q$ are the embedded (potentially non-contractible) Reeb orbits of $\lambda_{G,\varepsilon}$ in $S^3/G$ over $p$ and $q$, then we have that \[\pi_G\circ\gamma_{\widetilde{p}}=\gamma_p^{m_p},\,\,\,\text{and}\,\,\,\pi_G\circ\gamma_{\widetilde{q}}=\gamma_q^{m_q},\]
where $m_p, m_q\in\N$ are the orders of $[\gamma_p]$ and $[\gamma_q]$ in $\pi_1(S^3/G)$. In particular, we can simplify equation \eqref{equation: induced cylinder}: \[u_x\in\mathcal{M}^J(\gamma_p^{m_p},\gamma_q^{m_q})/\R.\] This allows us to establish an orbifold version of the correspondence in \cite[\S 5]{N2}:
\begin{align}
    \mathcal{M}(p,q)&\cong\mathcal{M}^J(\gamma_p^{m_p},\gamma_q^{m_q})/\R \label{equation: induced cylinder 2}\\
    x&\sim u_x\nonumber.
\end{align}

\section{Computation of filtered cylindrical contact homology}\label{section: computation of filtered contact homology}

From Section \ref{section: introduction}, a finite subgroup of $\SU(2)$ is either cyclic, conjugate to the binary dihedral group $\D^*_{2n}$, or is one of the binary polyhedral groups $\T^*$, $\Oc^*$, or $\I^*$. If  subgroups $G_1$ and $G_2$ satisfy $A^{-1}G_2A=G_1$, for $A\in\SU(2)$, then the map $S^3\to S^3$, $p\mapsto A\cdot p$, descends to a strict contactomorphism   $(S^3/G_1,\lambda_{G_1}) \to (S^3/G_2,\lambda_{G_2})$, preserving Reeb dynamics. Thus, we are free to compute the contact homology of $S^3/G$ for a particular selection $G$ in a given conjugacy class.

The action threshold used to compute the filtered homology will depend on the type of subgroup $G$. In the rest of this section, for $N\in\N$, $L_N\in\R$ denotes 
\begin{enumerate}
     \itemsep-.35em
    \item $2\pi N-\frac{\pi}{n}$, when $G$ is cyclic of order $n$,
    \item $2\pi N-\frac{\pi}{2n}$, when $G$ is conjugate to $\D_{2n}^*$,
    \item $2\pi N-\frac{\pi}{10}$, when $G$ is conjugate to $\T^*$, $\Oc^*$, or $\I^*$.
\end{enumerate}

\subsection{Cyclic subgroups} \label{subsection: cyclic}

From \cite[Theorem 1.5]{AHNS}, the positive $S^1$-equivariant symplectic homology of the link of the $A_n$ singularity, $L_{A_n}\subset\C^3$, with contact structure $\xi_0=TL_{A_n}\cap J_{\C^3}(TL_{A_n})$, satisfies:
\[SH_*^{+,S^1}(L_{A_n},\xi_0)=\begin{cases} \Q^{n} & *=1,\\ \Q^{n+1} & *\geq3\,\, \mbox{and odd} \\ 0 &\mbox{else.}\end{cases}\]
Furthermore, \cite{BO} prove that there are restricted classes of contact manifolds whose cylindrical contact homology (with a degree shift) is isomorphic to its positive $S^1$-equivariant symplectic homology, when both are defined over $\Q$ coefficients. Indeed, we note an isomorphism by inspection when we compare this symplectic homology with the cylindrical contact homology of $(L_{A_n},\xi_0)\cong(S^3/G,\xi_G)$ for $G\cong\Z_{n+1}$ from Theorem \ref{theorem: main}. Although this cylindrical contact homology is computed in \cite[Theorem 1.36]{N2}, we recompute these groups using a direct limit of filtered contact homology to present the general structure of the computations to come in the dihedral and polyhedral cases.

Let $G\cong\Z_{n}$ be a finite cyclic subgroup of $G$ of order $n$. If $|G|=n$ is even, with $n=2m$, then $P:G\to H:=P(G)$ has nontrivial two element kernel, and $H$ is cyclic of order $m$. Otherwise, $n$ is odd, $P:G\to H=P(G)$ has trivial kernel, and $H$ is cyclic of order $n$. 

By conjugating $G$ if necessary, we can assume that $H$ acts on $S^2$ by rotations around the vertical axis through $S^2$. The height function  $f:S^2\to[-1,1]$ is Morse, $H$-invariant, and provides precisely two fixed points; the north pole, $\mathfrak{n}\in S^2$ featuring $f(\mathfrak{n})=1$,  and the south pole, $\mathfrak{s}\in S^2$ where $f(\mathfrak{s})=-1$.  For small $\varepsilon$, we can expect to see iterates of two embedded Reeb orbits, denoted $\gamma_{\mathfrak{n}}$ and $\gamma_{\mathfrak{s}}$, of $\lambda_{G,\varepsilon}:=(1+\varepsilon\mathfrak{p}^*f_H)\lambda_{G}$ in $S^3/G$ as the only generators of the filtered chain groups. Both $\gamma_{\mathfrak{n}}$ and $\gamma_{\mathfrak{s}}$ are elliptic and parametrize the exceptional fibers in $S^3/G$ over the two orbifold points of $S^2/H$.

Select $N\in \N$. Lemma \ref{lemma: ActionThresholdLink} produces an $\varepsilon_N>0$ for which if $\varepsilon\in(0,\varepsilon_N]$, then all orbits in $\mathcal{P}^{L_N}(\lambda_{G,\varepsilon})$ are nondegenerate and are iterates of $\gamma_{\mathfrak{s}}$ or $\gamma_{\mathfrak{n}}$,  whose actions satisfy 
\begin{equation}\label{equation: action cyclic}
    \mathcal{A}(\gamma_{\mathfrak{s}}^k)=\frac{2\pi k(1-\varepsilon)}{n},\,\,\,\,\,\mathcal{A}(\gamma_{\mathfrak{n}}^k)=\frac{2\pi k(1+\varepsilon)}{n},
\end{equation} which tells us that the $L_N$-filtered chain complex is $\Q$-generated by the Reeb orbits $\gamma_{\mathfrak{s}}^k$ and $\gamma_{\mathfrak{n}}^k$, for $1\leq k<nN$. With respect to $\tau_G$, the rotation numbers $\theta_{\mathfrak{s}}^k$ and $\theta_{\mathfrak{n}}^k$ of $\gamma^k_{\mathfrak{s}}$ and $\gamma^k_{\mathfrak{n}}$  satisfy \[\theta^k_{\mathfrak{s}}=\frac{2k}{n}-\varepsilon\frac{k}{n(1-\varepsilon)},\,\,\,\,\,\,\theta^k_{\mathfrak{n}}=\frac{2k}{n}+\varepsilon\frac{k}{n(1+\varepsilon)}.\] See Remark \ref{remark: rotation numbers} for a definition of rotation numbers and their role in computing Conley-Zehnder indices. Thus, if  $\varepsilon_N$ is sufficiently small: \begin{equation}\label{equation: CZ indices cyclic}
    \mu_{\CZ}(\gamma_{\mathfrak{s}}^k)=2\Bigl\lceil \frac{2k}{n} \Bigr\rceil-1,\,\,\,\,\,\mu_{\CZ}(\gamma_{\mathfrak{n}}^k)=2\Bigl\lfloor \frac{2k}{n} \Bigr\rfloor+1.
\end{equation}
For $i\in\Z$, let $c_i$ denote the number of $\gamma\in\mathcal{P}^{L_N}(\lambda_{G,\varepsilon_N})$  with $|\gamma|=\mu_{\CZ}(\gamma)-1=i$ (recall that $|\gamma|$ denotes the integral grading of $\gamma$ in the chain complex). Then:
\begin{itemize}
    \itemsep-.35em
    \item $c_0=n-1$,    
    \item $c_{2i}=n$ for $0<i<2N-1$,
    \item $c_{4N-2}=n-1$,
    \item $c_i=0$ for all other $i$ values.
\end{itemize} \begin{remark} \label{remark: CZcyclic}
Although the above description of $c_i$ is independent of the parity of $n$, the contributions by iterates of $\gamma_{\mathfrak{s}}$ versus $\gamma_{\mathfrak{n}}$ \emph{do} depend on this parity. If $a_i$ denotes the number of $\gamma^k_{\mathfrak{s}}$ with $|\gamma^k_{\mathfrak{s}}|=i$, and $b_i$ denotes the number of $\gamma^k_{\mathfrak{n}}$ with $|\gamma^k_{\mathfrak{n}}|=i$ (so that $a_i+b_i=c_i$ for all $i$), the values $a_i$ and $b_i$ as a function of $i$  depend on the parity of $n$; when $|G|=n$ is even $a_i$ and $b_i$ are functions of $i$ mod 2, however if $|G|=n$ is odd, these values of $a_i$ and $b_i$ depend on the value of $i$ mod 4. In the case that $n=2m$ is even we have the following decomposition of the integers $c_i$: 
\begin{itemize}
    \itemsep-.35em
   \item $a_0=m$ and $b_0=m-1$,
   \item $a_{2i}=m=b_{2i}$ for $0<i<2N-1$,
    \item $a_{4N-2}=m-1$ and $b_{4N-2}=m$.
\end{itemize}
In the case that $n=2p+1$ is odd, we have the following more involved decomposition of $c_i$:
\begin{itemize}
     \itemsep-.35em
    \item $a_0=p$ and $b_0=p$ (providing $c_0=2p=n-1$),
    \item $a_{4i}=p$ and $b_{4i}=p+1$ for $0<i<N$ (providing $c_{4i}=2p+1=n$),
    \item $a_{4i-2}=p+1$ and $b_{4i-2}=p$ for $0<i <N$ (providing $c_{4i-2}=2p+1=n$),
    \item $a_{4N-2}=p$ and $b_{4N-2}=p$ (providing $c_{4N-2}=2p=n-1$).
\end{itemize}
The first few Conley Zehnder indices appearing in $L_{A_2}\cong L(3,2)=S^3/\Z_3$ are listed \cite[Ex. 1.3]{AHNS} and illustrate this behavior in the odd case.
\end{remark}

By \eqref{equation: CZ indices cyclic}, if $\mu_{\CZ}(\gamma^k_{\mathfrak{n}})=1$, then $k<n/2$, so the orbit $\gamma^k_{\mathfrak{n}}$ is not contractible. If $\mu_{\CZ}(\gamma^k_{\mathfrak{s}})=1$, then by \eqref{equation: CZ indices cyclic}, $k\leq n/2$ and  $\gamma^k_{\mathfrak{s}}$ is not contractible. Thus,  $\lambda_N:=\lambda_{G,\varepsilon_N}$ is $L_N$-dynamically convex and so by \cite[Theorem 1.3]{HN},\footnote{One hypothesis of \cite[Th. 1.3]{HN} requires that all contractible Reeb orbits $\gamma$ satisfying $\mu_{\CZ}(\gamma)=3$ must be embedded. This fails in our case by considering the contractible $\gamma_{\mathfrak{s}}^n$, which is not embedded yet satisfies $\mu_{\CZ}(\gamma_{\mathfrak{s}}^n)=3$. See Lemma \ref{lemma: generic j} for an explanation of why we do not need this additional hypothesis.} a generic choice $J_N\in\mathcal{J}(\lambda_N)$ provides a well defined filtered chain complex, yielding the isomorphism 
\begin{align*}
    CH_*^{L_N}(S^3/G,\lambda_N, J_N)&\cong\bigoplus_{i=0}^{2N-1}\Q^{n-2}[2i]\oplus\bigoplus_{i=0}^{2N-2} H_*(S^2;\Q)[2i]\\
    &=\begin{cases} \Q^{n-1} & *=0, \, 4N-2 \\
\Q^{n} & *=2i, 0<i<2N-1 \\ 0 &\mbox{else.}\end{cases}
\end{align*}
 This follows from investigating the good contributions to  $c_i$, which is 0 for odd $i$, implying $\partial^{L_N}=0$. This supports our Theorem \ref{theorem: main} in the cyclic case, because $G$ is abelian and so $m=|\text{Conj}(G)|=n$. Theorem \ref{theorem: cobordisms induce inclusions} of Section \ref{section: direct limits of filtered homology} will permit taking a direct limit over inclusions of these  groups.

\subsection{Binary dihedral groups $\D^*_{2n}$}\label{subsection: dihedral}

The binary dihedral group $\D^*_{2n}\subset \SU(2)$ has order $4n$ and projects to the dihedral group $\D_{2n}\subset \SO(3)$, which has order $2n$, under the cover $P:\SU(2)\to \SO(3)$. With the quantity $n$ understood, these groups will respectively be denoted $\D^*$ and $\D$. The group $\D^*$ is generated by the two matrices\[A=\begin{pmatrix} \zeta_n & 0 \\ 0 & \overline{\zeta_n}\end{pmatrix} \hspace{1cm} B=\begin{pmatrix} 0 & -1 \\ 1 & 0\end{pmatrix},\]
where $\zeta_n:=\exp(i\pi/n)$ is a primitive $2n^{\text{th}}$ root of unity. These matrices satisfy the relations $B^2=A^n=-\text{Id}$ and $BAB^{-1}=A^{-1}$. The group elements may be enumerated as follows: \[\D^*=\{A^kB^l:0\leq k< 2n, \, 0\leq l\leq 1 \}.\] 
By applying \eqref{equation: P in coordinates}, the following matrices generate $\D\subset\SO(3)$:
\[a:=P(A)=\begin{pmatrix} 1 & 0 & 0\\ 0 & \cos{(2\pi/n)} & -\sin{(2\pi/n)} \\ 0 & \sin{(2\pi/n)} & \cos{(2\pi/n)} \end{pmatrix} \hspace{1cm} b:=P(B)=\begin{pmatrix}-1 & 0 & 0 \\ 0 & 1 & 0 \\ 0 & 0 & -1\end{pmatrix}.\]
 There are three types of fixed points in $S^2$ of the $\D$-action, categorized as follows, and depicted in Figure \ref{figure: dihedral repeat}:
 
 \begin{figure}[h]
    \centering
    \includegraphics[width=0.4\textwidth]{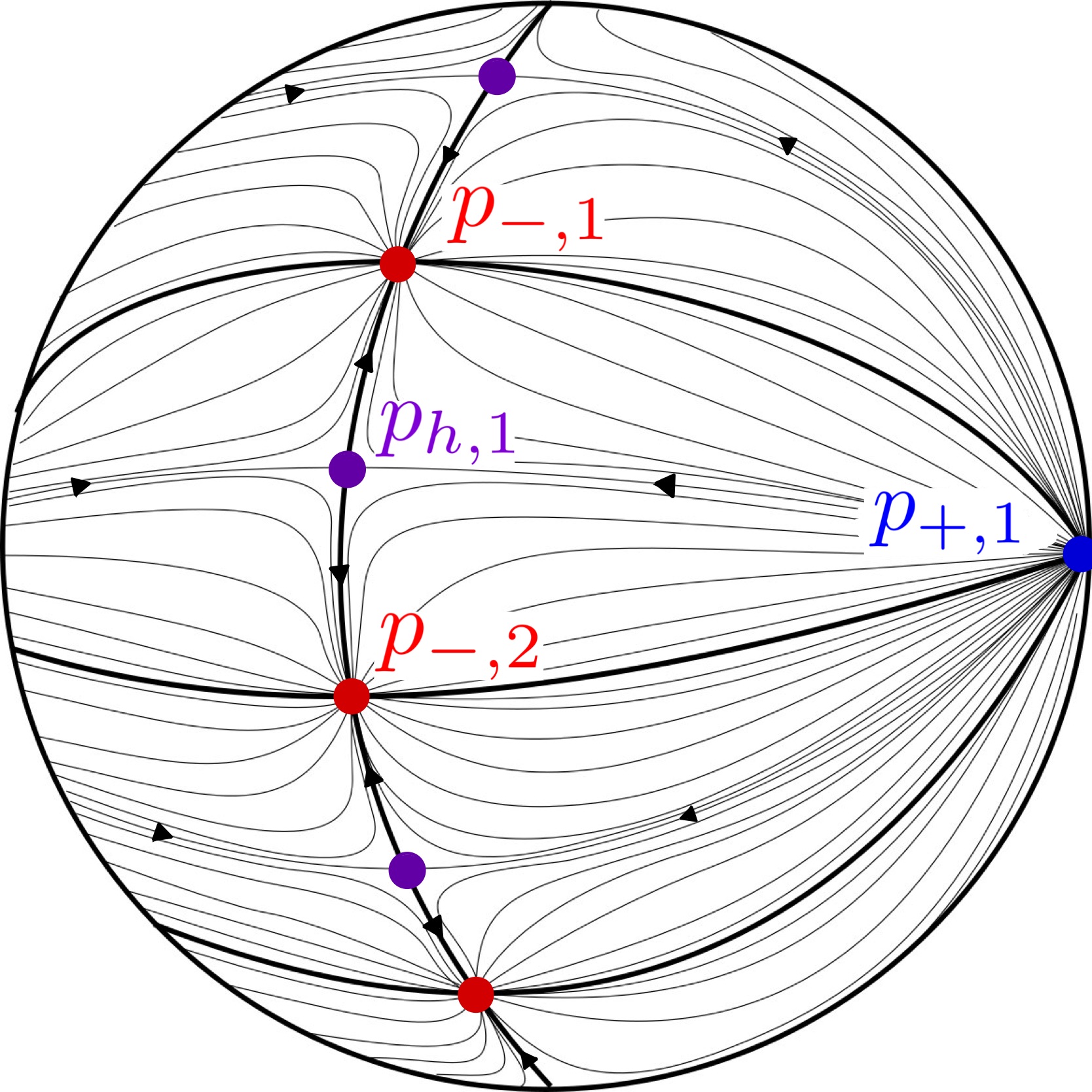}
    \caption{Fixed points of the $\D$ - action on $S^2$ and flow lines of a Morse function constructed in Section \ref{subsection: constructing morse functions}.}
    \label{figure: dihedral repeat}
\end{figure}

\textbf{Morse index 0}: $p_{-,k}:=(0,\cos{((\pi+2k\pi)/n)}, \sin{((\pi+2k\pi)/n)})\in \text{Fix}(\D)$, for $k\in\{1,\dots, n\}$. We have that $p_{-,n}$ is a fixed point of $ab$ and $p_{-,k}=a^k\cdot p_{-,n}$. Thus, $p_{-,k}$ is a fixed point of $a^k(ab)a^{-k}=a^{2k+1}b$. These $n$ points enumerate a $\D$-orbit in $S^2$, and so the isotropy subgroup of $\D$ associated to any of the $p_{-,k}$ is of order 2 and is generated by $a^{2k+1}b\in\D$. The point $p_-\in S^2/\D$ denotes the image of any $p_{-,k}$ under $\pi_{\D}$.
    
    \vspace{.25cm}

    \textbf{Morse index 1}: $p_{h,k}:=(0,\cos{(2k\pi/n)}, \sin{(2k\pi/n)})\in \text{Fix}(\D)$, for $k\in\{1,\dots, n\}$. We have that $p_{h,n}$ is a fixed point of of $b$ and  $p_{h,k}=a^k\cdot p_{h,n}$. Thus, $p_{h,k}$ is a fixed point of $a^k(b)a^{-k}=a^{2k}b$. These $n$ points enumerate a $\D$-orbit in $S^2$, and so the isotropy subgroup of $\D$ associated to any of the $p_{h,k}$ is of order 2 and is generated by  $a^{2k}b\in\D$. The point $p_h\in S^2/\D$ denotes the image of any $p_{h,k}$ under $\pi_{\D}$. 
    
        \vspace{.25cm}

    \textbf{Morse index 2}: $p_{+,1}=(1,0,0)$ and $p_{+,2}=(-1,0,0)$. These  are the fixed points of $a^k$, for $0<k<n$, and together enumerate a single two element $\D$-orbit. The isotropy subgroup associated to either of the points is cyclic of order $n$ in $\D$, generated by $a$.  The point $p_+\in S^2/\D$ denotes the image of any one of these two points under $\pi_{\D}$.
     \vspace{.25cm}

There exists a $\D$-invariant, Morse-Smale function $f$ on $(S^2,\omega_{\FS}(\cdot,j \cdot))$, with $\text{Crit}(f)=\text{Fix}(\D)$, which descends to an orbifold Morse function $f_{\D}:S^2/\D\to\R$, constructed in Section \ref{subsection: constructing morse functions}. Furthermore, there are stereographic coordinates at
\begin{enumerate}
    \itemsep-.35em
    \item the points $p_{-,k}$ in which $f$ takes the form $(x^2+y^2)/2-1$ near $(0,0)$,
    \item the points $p_{h,k}$ in which $f$  takes the form $(x^2-y^2)/2$  near $(0,0)$,
    \item the  points $p_{+,k}$ in which $f$ takes the form $1-(x^2+y^2)/2$  near $(0,0)$.
\end{enumerate}
The orbifold surface $S^2/\D$ is homeomorphic to $S^2$ and has three orbifold points (explained in Section \ref{subsection: visualizing holomorphic cylinders: an example}, see Figure \ref{figure: fundamental domain}). Lemma \ref{lemma: CZdihedral} identifies the Reeb orbits of $\lambda_{\D^*, \varepsilon}=(1+\varepsilon \fp^*f_{\D})\lambda_{\D^*}$ that appear in the filtered chain complex and computes their Conley-Zehnder indices. Recall that the grading of a generator $\gamma$ is given by $|\gamma|=\mu_{\CZ}(\gamma)-1$.

\begin{lemma}\label{lemma: CZdihedral}
Fix $N\in\N$. Then there exists an $\varepsilon_N>0$ such that for all $\varepsilon\in(0,\varepsilon_N]$, every $\gamma\in\mathcal{P}^{L_N}(\lambda_{\D^*,\varepsilon})$ is nondegenerate and projects to an orbifold critical point of $f_{\D}$ under $\fp$, where $L_N=2\pi N-\pi/2n$. If $c_i$ denotes the number of $\gamma\in \mathcal{P}^{L_N}(\lambda_{\D^*,\varepsilon})$ with $|\gamma|=i$, then 
\begin{itemize}
 \itemsep-.35em
    \item $c_i=0$ if $i<0$ or $i> 4N-2$,
    \item $c_i=n+2$ for $i=0$ and $i=4N-2$, with all $n+2$ contributions by good Reeb orbits,
    \item $c_i=n+3$ for even $i$, $0<i<4N-2$, with all $n+3$ contributions by good Reeb orbits,
    \item $c_i=1$ for odd $i$, $0<i<4N-2$ and this contribution is by a bad Reeb orbit.
\end{itemize}
\end{lemma}

\begin{proof}

Apply Lemma \ref{lemma: ActionThresholdLink} to $L_N=2\pi N-\frac{\pi}{2n}$ to obtain $\varepsilon_N$. Now, if $\varepsilon\in(0,\varepsilon_N]$, we have that every $\gamma\in\mathcal{P}^{L_N}(\lambda_{\P^*,\varepsilon})$ is nondegenerate and projects to an orbifold critical point of $f_{\D}$. We now study the actions and indices of these  orbits.

\vspace{.25cm}
    \textbf{Orbits over $p_-$}: Let $e_-$ denote the embedded Reeb orbit of $\lambda_{\D^*,\varepsilon}$ which projects to $p_-\in S^2/\D$. By Lemmas \ref{lemma: covering multiplicity} and \ref{lemma: orbits}, $e_-^4$ lifts to an embedded Reeb orbit of $\lambda_{\varepsilon}$ in $S^3$ with action $2\pi(1-\varepsilon)$, projecting to some $p_{-,j}\in S^2$. Thus, $\mathcal{A}(e_- )=\pi(1-\varepsilon)/2$, so $\mathcal{A}(e^k_-)=k\pi (1-\varepsilon)/2$. Hence, we include iterates $e_-^k$ for  $k=1,2,\dots, 4N-1$ in the $L_N$-filtered chain complex. Using Remark \ref{remark: same local models} and Proposition \ref{proposition: flow}, we see that the linearized Reeb flow of $\lambda_{\D^*,\varepsilon}$ along $e^k_-$ with respect to trivialization $\tau_{\D^*}$ is given by the family of matrices $\Phi_t=\mathcal{R}(2t-t\varepsilon)$ for $t\in[0,k\pi(1-\varepsilon)/2]$, where we have used that $f(p_{-,j})=-1$ and that we have stereographic coordinates $\psi$ at the point $p_{-,j}$ such that $H(f,\psi)=\text{Id}$. We see that $e_-^k$ is elliptic with rotation number $\theta_-^k=k/2-k\varepsilon/4(1-\varepsilon)$, thus
    \[\mu_{\CZ}(e_-^k)=2\Bigl\lceil \frac{k}{2}-\frac{k\varepsilon}{4(1-\varepsilon)} \Bigr\rceil-1=2\Bigl\lceil \frac{k}{2} \Bigr\rceil-1,\] 
    where the last step is valid by reducing $\varepsilon_{N}$ if necessary. See Remark \ref{remark: rotation numbers} for a review of rotation numbers and their role in computing Conley-Zehnder indices.
    
    \vspace{.25cm}
    
    \textbf{Orbits over} $p_h$: Let $h$ denote the embedded Reeb orbit of $\lambda_{\D^*,\varepsilon}$ which projects to $p_h\in S^2/\D$. By Lemmas \ref{lemma: covering multiplicity} and \ref{lemma: orbits}, $h^4$ lifts to an embedded Reeb orbit of $\lambda_{\varepsilon}$ in $S^3$ with action $2\pi$, projecting to some $p_{h,j}\in S^2$. Thus, $\mathcal{A}(h)=\pi/2$, so $\mathcal{A}(h^k)=k\pi/2$. Hence, we include (only the good) iterates $h^k$, for  $k=1,2,\dots, 4N-1$ in our $L_N$-filtered chain  complex.
    
    To see that $h$ is a hyperbolic Reeb orbit, we  consider its 4-fold cover $h^4$. By Remark \ref{remark: same local models} and Proposition \ref{proposition: flow}, we compute the linearized Reeb flow, noting that the lifted Reeb orbit projects to $p_{h,j}\in S^2$, where $f(p_{h,j})=0$, and that we have stereographic coordinates $\psi$ at $p_{h,j}$ such that $H(f,\psi)=\text{Diag}(1,-1)$. We evaluate the matrix at $t=2\pi$ to see that the linearized return map associated to $h^4$ is \[\exp\begin{pmatrix}0 & 2\pi\varepsilon \\ 2\pi\varepsilon&0\end{pmatrix}=\begin{pmatrix}\cosh{(2\pi\varepsilon)} &\sinh{(2\pi\varepsilon)}\\\sinh{(2\pi\varepsilon)}&\cosh{(2\pi\varepsilon)}\end{pmatrix}.\] The eigenvalues of this matrix are $\cosh{(2\pi\varepsilon)}\pm\sinh{(2\pi\varepsilon)}$. So long as $\varepsilon$ is small, these eigenvalues are real and positive, so that $h^4$ is positive hyperbolic, implying that $h$ is also hyperbolic. If $\mu_{\CZ}(h):=I\in\Z$, then $\mu_{\CZ}(h^4)=4I$, because $\mu_{\CZ}$ grows linearly  for hyperbolic Reeb orbits, by Remark \ref{remark: CZ indices of iterates}.  By Corollary \ref{corollary: CZ sphere} and Remark \ref{remark: same local models}, we know $\mu_{\CZ}(h^4)=4\cdot1+1-1=4$. Hence, $I=1$ and thus, $h$ is negative hyperbolic (by Remark \ref{remark: parity of Conley-Zehnder indices}). And so, by definition, the even iterates of $h$ are \emph{bad} Reeb orbits.
    \vspace{.25cm}
    
    \textbf{Orbits over $p_+$}: Let $e_+$ denote the embedded Reeb orbit of $\lambda_{\D^*,\varepsilon}$ which projects to $p_+\in S^2/\D$. By Lemmas \ref{lemma: covering multiplicity} and \ref{lemma: orbits}, the  $2n$-fold cover $e_-^{2n}$ lifts to some embedded Reeb orbit of $\lambda_{\varepsilon}$ in $S^3$ with action $2\pi(1+\varepsilon)$, projecting to some $p_{+,j}\in S^2$. Thus, $\mathcal{A}(e_+)=\pi(1+\varepsilon)/n$, so $\mathcal{A}(e^k_-)=k\pi(1+\varepsilon)/n$. Hence, we include the iterates $e_+^k$, for  $k=1,2,\dots, 2nN-1$ in our $L_N$-filtered complex. Using Remark \ref{remark: same local models} and Proposition \ref{proposition: flow}, we see that the linearized Reeb flow of $\lambda_{\D^*,\varepsilon}$ along $e^k_+$ with respect to trivialization $\tau_{\D^*}$ is given by the family of matrices \[\Phi_t=\mathcal{R}\bigg(\frac{2t}{1+\varepsilon}+\frac{t\varepsilon}{(1+\varepsilon)^2}\bigg)\,\,\, \text{for}\,\,\, t\in[0,k\pi(1+\varepsilon)/n],\] where we have used that $f(p_{+,j})=1$ and that we have stereographic coordinates $\psi$ at  $p_{+,j}$ such that $H(f,\psi)=-\text{Id}$. Thus, $e_+^k$ is elliptic with rotation number $\theta_+^k$  and Conley-Zehnder index given by \[\theta_+^k=\frac{k}{n}+\frac{\varepsilon k}{2n(1+\varepsilon)}, \,\,\mu_{\CZ}(e_+^k)=1+2\Bigl\lfloor \frac{k}{n}+\frac{\varepsilon k}{2n(1+\varepsilon)} \Bigr\rfloor=1+2\Bigl\lfloor \frac{k}{n} \Bigr\rfloor,\]
    where the last step is valid for sufficiently small $\varepsilon_N$. Again, see Remark \ref{remark: rotation numbers} for a definition of rotation numbers and their role in computing Conley-Zehnder indices.
\end{proof}

Lemma \ref{lemma: CZdihedral} produces the sequence $(\varepsilon_N)_{N=1}^{\infty}$, which we can assume  decreases monotonically to 0 in $\R$. Define the sequence of 1-forms $(\lambda_N)_{N=1}^{\infty}$ on $S^3/\D^*$ by $\lambda_N:=\lambda_{\D^*,\varepsilon_N}$.

\begin{summary*}(Dihedral data) We have
\begin{equation} \label{equation: CZ indices dihedral}
    \mu_{\CZ}(e_-^k)=2\Bigl\lceil \frac{k}{2} \Bigr\rceil-1,\,\,\,\mu_{\CZ}(h^k)=k,\,\,\,\mu_{\CZ}(e_+^k)=2\Bigl\lfloor \frac{k}{n} \Bigr\rfloor+1,
\end{equation}
\begin{equation}\label{equation: action dihedral}
        \mathcal{A}(e_-^k)=\frac{k\pi(1-\varepsilon)}{2},\,\,\,\mathcal{A}(h^k)=\frac{k\pi}{2},\,\,\,\mathcal{A}(e_+^k)=\frac{k\pi(1+\varepsilon)}{n}.
\end{equation}
\end{summary*}

\begin{table}[h!]
\centering
 \begin{tabular}{||c | c | c | c ||} 
 \hline
Grading & Index & Orbits & $c_i$ \\ [0.5ex] 
 \hline\hline
 0 & 1 & $e_-,e_-^2,h,e_+,e_+^2,\dots,e_+^{n-1}$ & $n+2$\\ 
 \hline
 1 & 2 & $h^2$ & 1\\
 \hline
 2 & 3 & $e_-^3,e_-^4, h^3, e_+^{n}, \dots, e_+^{2n-1}$ & $n+3$\\
  \hline
 \vdots & \vdots & \vdots & \vdots \\
  \hline
 $4N-4$ & $4N-3$ & $e_-^{4N-3},e_-^{4N-2},h^{4N-3}, e_+^{(2N-2)n},\dots e_+^{(2N-1)n-1}$ & $n+3$\\
  \hline
 $4N-3$ & $4N-2$ & $h^{4N-2}$ & 1\\
    \hline
 $4N-2$ & $4N-1$ & $e_-^{4N-1},h^{4N-1}, e_+^{(2N-1)n},\dots,e_+^{2Nn-1}$ & $n+2$\\ 
 \hline
\end{tabular}
\caption{Reeb orbits of $\mathcal{P}^{L_N}(\lambda_{\D_{2n}^*,\varepsilon_N})$ and their Conley-Zehnder indices}
\label{table: Dihedral CZ indices}
\end{table}

None of the orbits in the first two rows of Table \ref{table: Dihedral CZ indices} are contractible, thus $\lambda_N=\lambda_{\D^*,\varepsilon_N}$ is $L_N$-dynamically convex and so by \cite[Theorem 1.3]{HN},\footnote{One hypothesis of \cite[Th. 1.3]{HN} requires that all contractible Reeb orbits $\gamma$ satisfying $\mu_{\CZ}(\gamma)=3$ must be embedded. This fails in our case by considering the contractible $e_-^4$, which is not embedded yet satisfies $\mu_{\CZ}(e_-^4)=3$. See Lemma \ref{lemma: generic j} for an explanation of why we do not need this additional hypothesis.} a generic choice $J_N\in\mathcal{J}(\lambda_N)$ provides a well defined filtered chain complex, yielding the  isomorphism of $\Z$-graded vector spaces 
\begin{align*}
    CH_*^{L_N}(S^3/\D^*,\lambda_N, J_N)&\cong\bigoplus_{i=0}^{2N-1}\Q^{n+1}[2i]\oplus\bigoplus_{i=0}^{2N-2} H_*(S^2;\Q)[2i]\\
    &=\begin{cases} \Q^{n+2} & *=0, \, 4N-2 \\
\Q^{n+3} & *=2i, 0<i<2N-1 \\ 0 &\mbox{else.}\end{cases}.
\end{align*}
 This follows from investigating the good contributions to  $c_i$, which is 0 for odd $i$, implying $\partial^{L_N}=0$. This supports our Theorem \ref{theorem: main} in the dihedral case, because $|\text{Conj}(\D_{2n}^*)|=n+3$. Theorem \ref{theorem: cobordisms induce inclusions} of Section \ref{section: direct limits of filtered homology} will permit taking a direct limit over inclusions of these  groups.

\subsection{Binary polyhedral groups $\T^*$, $\Oc^*$, and $\I^*$} \label{subsection: polyhedral}

Unlike the dihedral case, we will not work with explicit matrix generators of the polyhedral groups, because computing the coordinates of the fixed points is more involved. Instead, we will take a more geometric approach. Let $\P^*\subset\SU(2)$ be some binary polyhedral group so that it is congruent to either $\T^*$, $\Oc^*$ or $\I^*$, with $|\P^*|=24$, 48, or 120, respectively. Let $\P\subset\SO(3)$ denote the image of $\P^*$ under the group homomorphism $P$. This group $\P$ is conjugate to one of $\T$, $\Oc$, or $\I$ in $\SO(3)$, and its order satisfies $|\P^*|=2|\P|$. The $\P$-action on $S^2$ is given by the symmetries of a regular polyhedron inscribed in $S^2$. The fixed point set, $\text{Fix}(\P)$, is partitioned into three $\P$-orbits. Let the number of vertices, edges, and faces of the polyhedron in question be $\mathscr{V}$, $\mathscr{E}$, and $\mathscr{F}$ respectively (see Table \ref{table:polyhedral xyz}). We describe the types of $\P$-fixed points in $S^2$ now, depicted in Figures \ref{figure: tetraheadral repeat}, \ref{figure: octahedral repeat}, and \ref{figure: icosahedral repeat}. Note that the flow lines shown between the fixed points are of Morse functions constructed in Section \ref{subsection: constructing morse functions}.
 
\begin{figure}[!htb]
    \centering
    \begin{minipage}{.45\textwidth}
        \centering
        \includegraphics[width=0.8\textwidth]{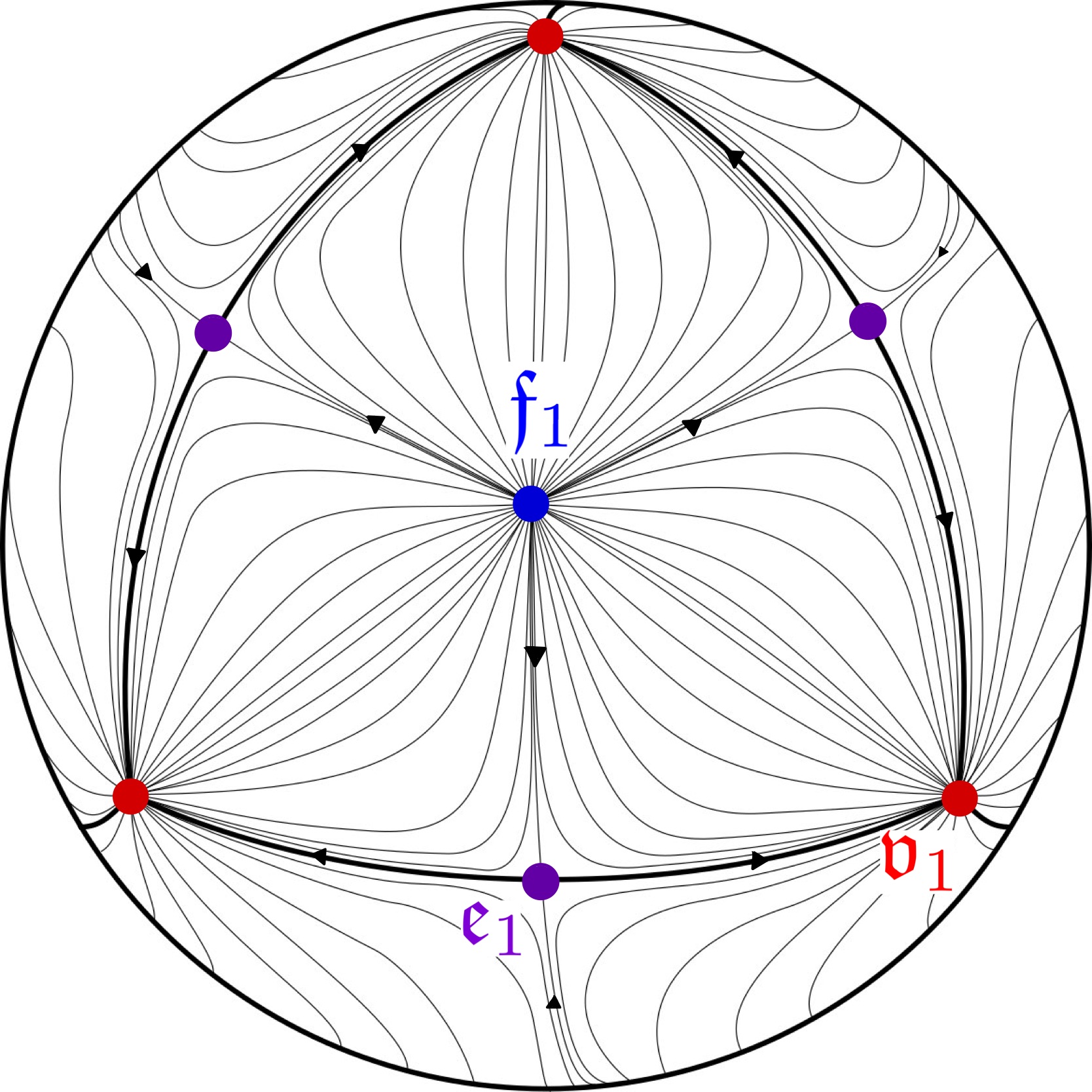}
        \caption{Tetrahedral fixed points}
        \label{figure: tetraheadral repeat}
    \end{minipage}%
    \begin{minipage}{0.45\textwidth}
        \centering
        \includegraphics[width=0.8\textwidth]{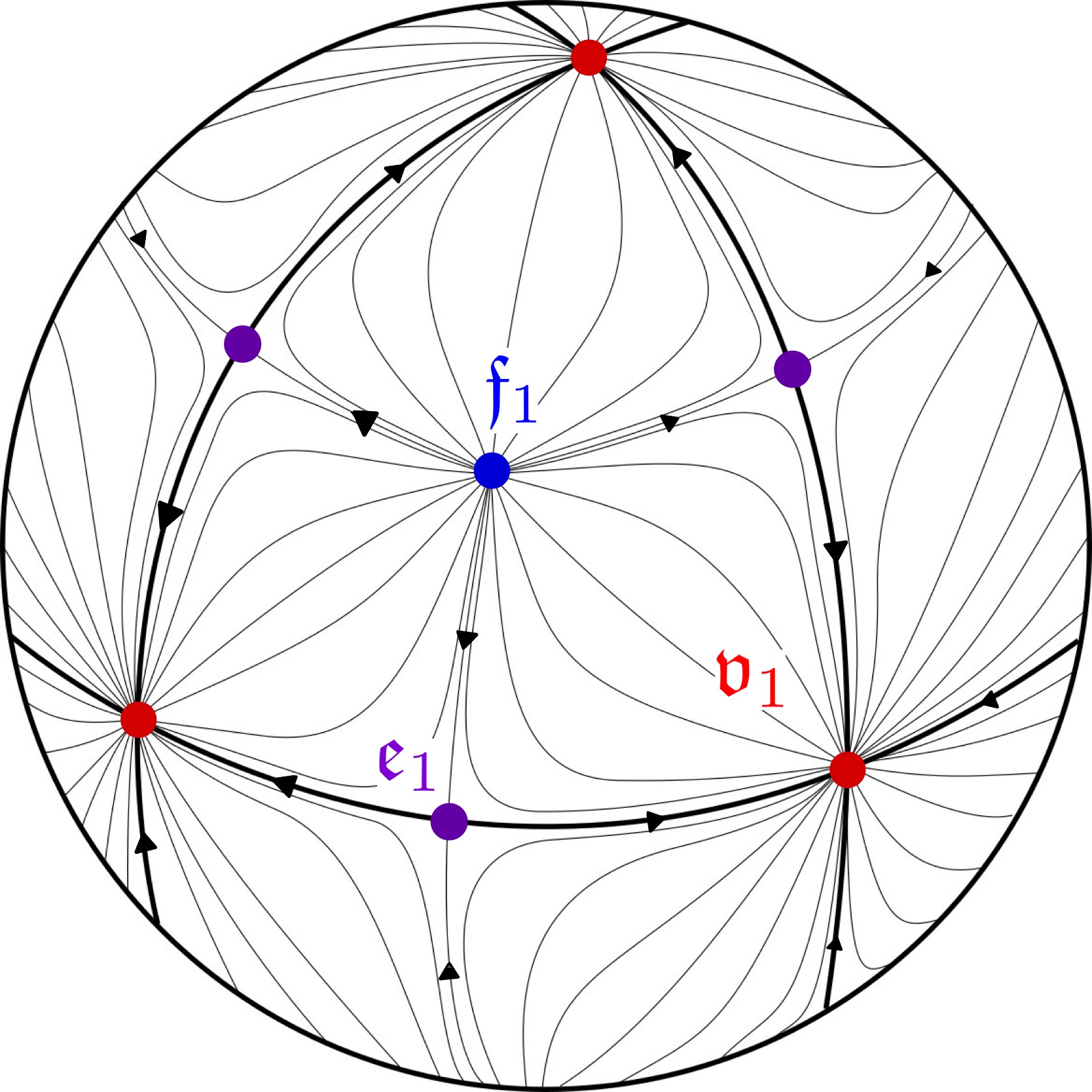}
        \caption{Octahedral fixed points}
        \label{figure: octahedral repeat}
    \end{minipage}
    \begin{minipage}{.45\textwidth}
        \centering
        \includegraphics[width=0.8\textwidth]{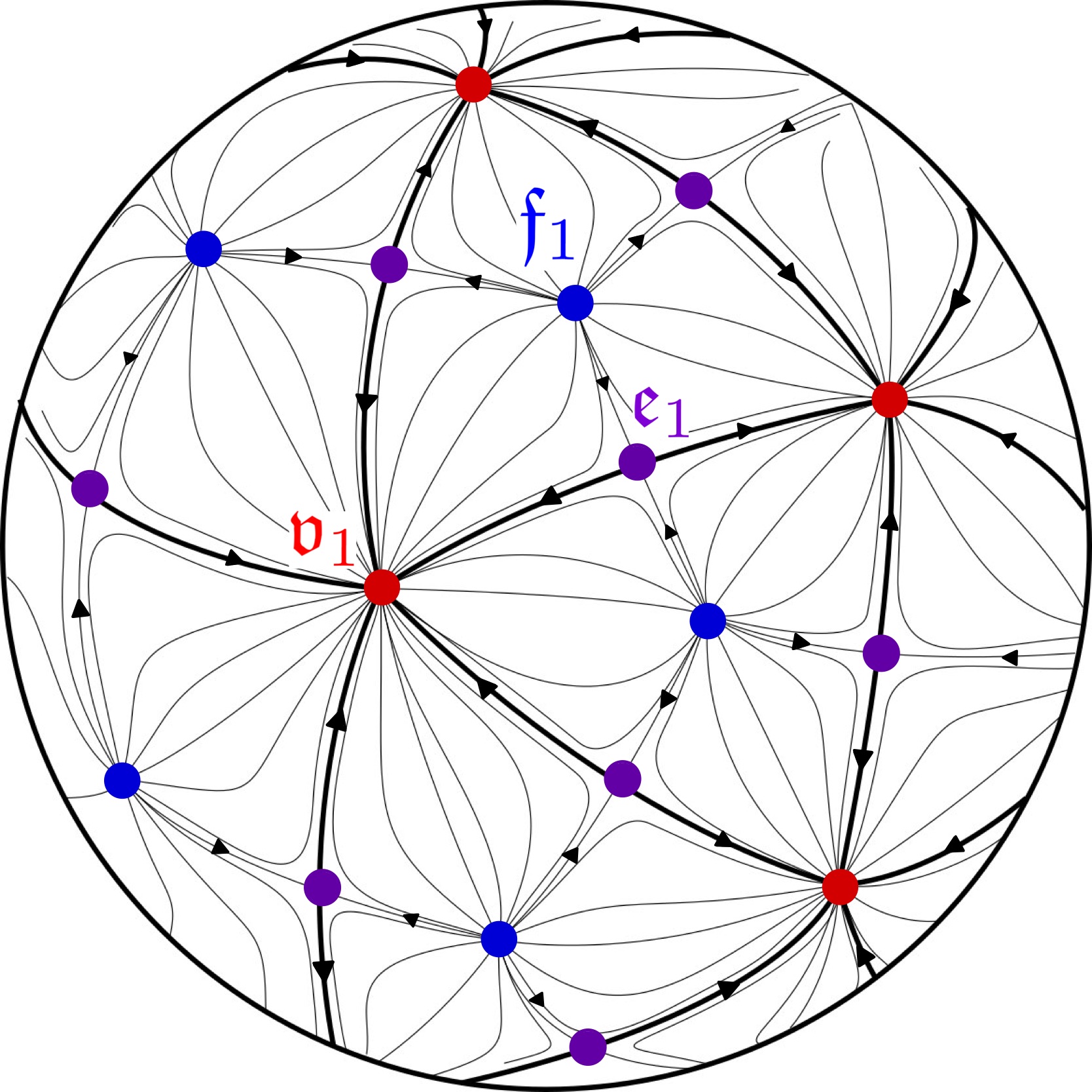}
        \caption{Icosahedral fixed points}
        \label{figure: icosahedral repeat}
    \end{minipage}%
\end{figure}

\textbf{Vertex type fixed points}: The set $\{\mathfrak{v}_1, \mathfrak{v}_2, \dots, \mathfrak{v}_\mathscr{V}\}\subset \text{Fix}(\P)$ constitutes a single $\P$-orbit, where each $\mathfrak{v}_i$ is an inscribed vertex of the polyhedron in $S^2$. Let $\mathbf{I}_{\mathscr{V}}\in\N$ denote $|\P|/\mathscr{V}$, so that the  isotropy subgroup associated to any of the $\mathfrak{v}_i$ is cyclic of order $\mathbf{I}_{\mathscr{V}}$. The quantity $\mathbf{I}_{\mathscr{V}}$ is 3, 4, or 5, for $\P=\T$, $\Oc$, or $\I$, respectively (i.e., the number of edges meeting at any vertex). Let $\mathfrak{v}\in S^2/\P$ denote the image of any of the $\mathfrak{v}_i$ under the orbifold covering map $\pi_{\P}:S^2\to S^2/\P$.

\vspace{.25cm}

\textbf{Edge type fixed points}: The set $\{\mathfrak{e}_1, \mathfrak{e}_2, \dots, \mathfrak{e}_\mathscr{E}\}\subset \text{Fix}(\P)$  constitutes a single $\P$-orbit, where each $\mathfrak{e}_i$ is the image of a midpoint of one of the edges of the polyhedron under the radial projection $\R^3\setminus\{0\}\to S^2$. Let $\mathbf{I}_{\mathscr{E}}\in\N$ denote $|\P|/\mathscr{E}$, so that the isotropy subgroup associated to any of the $\mathfrak{e}_i$ is cyclic of order $\mathbf{I}_{\mathscr{E}}$. One can see that $\mathbf{I}_{\mathscr{E}}=2$, for any choice of $\P$. Let $\mathfrak{e}\in S^2/\P$ denote the image of any of the $\mathfrak{e}_i$ under the orbifold covering map $\pi_{\P}:S^2\to S^2/\P$.

\vspace{.25cm}

\textbf{Face type fixed points}: The set $\{\mathfrak{f}_1, \mathfrak{f}_2, \dots, \mathfrak{f}_{\mathscr{F}}\}\subset \text{Fix}(\P)$ constitutes a single $\P$-orbit, where each $\mathfrak{f}_i$ is the image of a barycenter of one of the faces of the polyhedron under the radial projection $\R^3\setminus\{0\}\to S^2$. Let $\mathbf{I}_{\mathscr{F}}$ denote $|\P|/\mathscr{F}$, so that the isotropy subgroup associated to any of the $\mathfrak{f}_i$ is cyclic of order $\mathbf{I}_{\mathscr{F}}$. One can see that $\mathbf{I}_{\mathscr{F}}=3$, for any choice of $\P$. Let $\mathfrak{f}\in S^2/\P$ denote the image of any of the $\mathfrak{f}_i$ under the orbifold covering map $\pi_{\P}:S^2\to S^2/\P$.

\begin{table}[h!]
\centering
 \begin{tabular}{|| c | c | c | c | c | c | c | c | c ||} 
 \hline
 Group & Group order & $\mathscr{V}$ & $\mathscr{E}$ & $\mathscr{F}$ & $\mathbf{I}_{\mathscr{V}}$ & $\mathbf{I}_{\mathscr{E}}$ & $\mathbf{I}_{\mathscr{F}}$  & $|\text{Conj}(\P^*)|$\\ [0.5ex] 
 \hline\hline
 $\T$ & 12 & 4 & 6 & 4  & 3 & 2 & 3 & 7\\ 
 \hline
 $\Oc$ & 24 & 6 & 12 & 8 & 4 & 2 & 3 & 8\\
 \hline
 $\I$ & 60 & 12 & 30 & 20 & 5 & 2 & 3 & 9\\
 \hline
\end{tabular}
 \caption{Polyhedral quantities. Note $|\text{Cong}(\P^*)|=\mathbf{I}_{\mathscr{V}}+\mathbf{I}_{\mathscr{E}}+\mathbf{I}_{\mathscr{F}}-1.$}
 \label{table:polyhedral xyz}
\end{table}

\begin{remark}
(Dependence on choice of $\P^*$) The fixed point set of $\P$ is determined by the initial selection of $\P^*\subset\SU(2)$. More precisely, if $\P_1^*$ and $\P_2^*$ are conjugate binary polyhedral subgroups of $\SU(2)$, then $A^{-1}\P_2^*A=\P_1^*$  for some $A\in\SU(2)$. In this case, the rigid motion of $\R^3$ given by $P(A)\in\SO(3)$ takes the fixed point set of $\P_1$ to that of $\P_2$.
\end{remark}

There exists a $\P$-invariant, Morse-Smale function $f$ on $(S^2,\omega_{\FS}(\cdot,j \cdot))$, with $\text{Crit}(f)=\text{Fix}(\P)$, which descends to an orbifold Morse function $f_{\P}:S^2/\P\to\R$, constructed in Section \ref{subsection: constructing morse functions}. Furthermore, there are stereographic coordinates at
\begin{enumerate}
    \itemsep-.35em
    \item the points $\mathfrak{v}_i$ in which $f$ takes the form $(x^2+y^2)/2-1$ near $(0,0)$,
    \item the points $\mathfrak{e}_i$ in which $f$ takes the form $(x^2-y^2)/2$ near $(0,0)$,
    \item the points $\mathfrak{f}_i$ in which $f$ takes the form $1-(x^2+y^2)/2$ near $(0,0)$.
\end{enumerate}
The orbifold surface $S^2/\P$ is homeomorphic to $S^2$ and has three orbifold points (explained in Section \ref{subsection: visualizing holomorphic cylinders: an example}, see Figure \ref{figure: fundamental domain}). Lemma \ref{lemma: CZpolyhedral} identifies the Reeb orbits of $\lambda_{\P^*, \varepsilon}=(1+\varepsilon \fp^*f_{\P})\lambda_{\P^*}$ that appear in the filtered chain complex and computes their Conley-Zehnder indices. Let $m\in \N$ denote the integer $m:=\mathbf{I}_{\mathscr{V}}+\mathbf{I}_{\mathscr{E}}+\mathbf{I}_{\mathscr{F}}-1$ (equivalently, $m=|\text{Conj}(\P^*)|$, see Table \ref{table:polyhedral xyz}).

\begin{lemma} \label{lemma: CZpolyhedral}
Fix $N\in\N$. Then there exists an $\varepsilon_N>0$ such that, for all $\varepsilon\in(0,\varepsilon_N]$, every $\gamma\in\mathcal{P}^{L_N}(\lambda_{\P^*,\varepsilon})$ is nondegenerate and projects to an orbifold critical point of $f_{\P}$ under $\fp$, where $L_N:=2\pi N-\pi/10$. If $c_i$ denotes the number of $\gamma\in \mathcal{P}^{L_N}(\lambda_{\P^*,\varepsilon})$ with $|\gamma|=i$, then

\begin{enumerate}
    \itemsep-.35em
    \item $c_i=0$ if $i<0$ or $i> 4N-2$,
    \item $c_i=m-1$ for $i=0$ and $i=4N-2$ with all contributions by good Reeb orbits,
    \item $c_i=m$ for even $i$, $1<i<4N-1$ with all contributions by good Reeb orbits,
     \item $c_i=1$ for odd $i$, $0<i<4N-2$ and this contribution is by a bad Reeb orbit.
\end{enumerate}
\end{lemma}

\begin{proof}
Apply Lemma \ref{lemma: ActionThresholdLink} to $L_N=2\pi N-\frac{\pi}{10}$ to obtain $\varepsilon_N$. If $\varepsilon\in(0,\varepsilon_N]$, then every $\gamma\in\mathcal{P}^{L_N}(\lambda_{\P^*,\varepsilon})$ is nondegenerate and projects to an orbifold critical point of $f_{\P}$. We  investigate the actions and Conley-Zehnder indices of these three types of orbits. Our reasoning will largely follow that used in the proof of Lemma \ref{lemma: CZdihedral}, and so some details will be omitted.

\vspace{.25cm}

    \textbf{Orbits over $\mathfrak{v}$}: Let $\mathcal{V}$ denote the embedded Reeb orbit of $\lambda_{\P^*,\varepsilon}$ in $S^3/\P^*$ which projects to $\mathfrak{v}\in S^2/\P$. One computes that $\mathcal{A}(\mathcal{V}^k)=k\pi(1-\varepsilon)/\mathbf{I}_{\mathscr{V}}$, and so the iterates $\mathcal{V}^k$ are included for all $k<2N\mathbf{I}_{\mathscr{V}}$. The orbit $\mathcal{V}^k$ is elliptic with rotation number $\theta_{\mathcal{V}}^k$ and Conley-Zehnder index
    \[\theta_{\mathcal{V}}^k=\frac{k}{\mathbf{I}_{\mathscr{V}}}-\frac{\varepsilon k}{2\mathbf{I}_{\mathscr{V}}(1-\varepsilon)}, \,\,\,\,\,\mu_{\CZ}(\mathcal{V}^k)=2\Bigl\lceil \frac{k}{\mathbf{I}_{\mathscr{V}}} \Bigr\rceil-1.\]
    
    \vspace{.25cm}

    \textbf{Orbits over $\mathfrak{e}$}: Let $\mathcal{E}$ denote the embedded Reeb orbit of $\lambda_{\P^*,\varepsilon}$ in $S^3/\P^*$ which projects to $\mathfrak{e}\in S^2/\P$. By a similar study of the orbit $h$ of Lemma \ref{lemma: CZdihedral}, one sees that $\mathcal{A}(\mathcal{E}^k)=k\pi/2$, so the iterates $\mathcal{E}^k$ are included for all $k<4N$. Like the dihedral Reeb orbit $h$, $\mathcal{E}$ is negative hyperbolic with $\mu_{\CZ}(\mathcal{E})=1$, thus $\mu_{\CZ}(\mathcal{E}^k)=k$. The even iterates of $\mathcal{E}$ are bad Reeb orbits.
    
    \vspace{.25cm}
    
    \textbf{Orbits over $\mathfrak{f}$}: Let $\mathcal{F}$ denote the embedded Reeb orbit of $\lambda_{\P^*,\varepsilon}$ in $S^3/\P^*$ which projects to $\mathfrak{f}\in S^2/\P$. One computes that $\mathcal{A}(\mathcal{F}^k)=k\pi(1+\varepsilon)/3$, and so the iterates $\mathcal{F}^k$ are included for all $k<6N$. The orbit $\mathcal{F}^k$ is elliptic with rotation number $\theta_{\mathcal{F}}^k$ and Conley-Zehnder index
    \[\theta_{\mathcal{F}}^k=\frac{k}{3}+\frac{\varepsilon k}{6(1+\varepsilon)}, \,\,\,\,\,\mu_{\CZ}(\mathcal{F}^k)=2\Bigl\lfloor \frac{k}{3} \Bigr\rfloor+1=2\Bigl\lfloor \frac{k}{3} \Bigr\rfloor+1.\]
\end{proof}

Lemma \ref{lemma: CZpolyhedral} produces the (monotonically decreasing) sequence $(\varepsilon_N)_{N=1}^{\infty}$. Define the sequence of 1-forms $(\lambda_N)_{N=1}^{\infty}$ on $S^3/\P^*$  by $\lambda_N:=\lambda_{\P^*,\varepsilon_N}$.

\begin{summary*}
(Polyhedral data) We have
\begin{equation} \label{equation: CZ indices polyhedral}
    \mu_{\CZ}(\mathcal{V}^k)=2\Bigl\lceil \frac{k}{\mathbf{I}_{\mathscr{V}}} \Bigr\rceil-1,\,\,\,\mu_{\CZ}(\mathcal{E}^k)=k,\,\,\,\mu_{\CZ}(\mathcal{F}^k)=2\Bigl\lfloor \frac{k}{3} \Bigr\rfloor+1,
\end{equation}

\begin{equation}\label{equation: action polyhedral}
        \mathcal{A}(\mathcal{V}^k)=\frac{k\pi(1-\varepsilon)}{\mathbf{I}_{\mathscr{V}}},\,\,\,\mathcal{A}(\mathcal{E}^k)=\frac{k\pi}{2},\,\,\,\mathcal{A}(\mathcal{F}^k)=\frac{k\pi(1+\varepsilon)}{3}.
\end{equation}
\end{summary*}

\begin{table}[h!]
\centering
 \begin{tabular}{||c | c | c | c ||} 
 \hline
Grading & Index & Orbits & $c_i$ \\ [0.5ex]
 \hline\hline
 0 & 1 & $\mathcal{V}, \dots, \mathcal{V}^{\mathbf{I}_{\mathscr{V}}}, \mathcal{E}, \mathcal{F}, \mathcal{F}^2$ & $m-1$\\ 
 \hline
 1 & 2 & $\mathcal{E}^2$ & 1\\
 \hline
 2 & 3 & $\mathcal{V}^{\mathbf{I}_{\mathscr{V}}+1}, \dots, \mathcal{V}^{2\mathbf{I}_{\mathscr{V}}}, \mathcal{E}^3, \mathcal{F}^3, \mathcal{F}^4, \mathcal{F}^5$ & $m$\\
  \hline
 \vdots & \vdots & \vdots & \vdots \\
  \hline
 $4N-4$ & $4N-3$ & $\mathcal{V}^{(2N-2)\mathbf{I}_{\mathscr{V}}+1}, \dots, \mathcal{V}^{(2N-1)\mathbf{I}_{\mathscr{V}}}, \mathcal{E}^{4N-3}, \mathcal{F}^{6N-6}, \mathcal{F}^{6N-5}, \mathcal{F}^{6N-4}$ & $m$\\
  \hline
 $4N-3$ & $4N-2$ & $\mathcal{E}^{4N-2}$ & 1\\
    \hline
 $4N-2$ & $4N-1$ & $\mathcal{V}^{(2N-1)\mathbf{I}_{\mathscr{V}}+1}, \dots, \mathcal{V}^{2N\mathbf{I}_{\mathscr{V}}-1}, \mathcal{E}^{4N-1}, \mathcal{F}^{6N-3}, \mathcal{F}^{6N-2}, \mathcal{F}^{6N-1}$ & $m-1$\\ 
 \hline
\end{tabular}
\caption{Reeb orbits of $\mathcal{P}^{L_N}(\lambda_{\P^*,\varepsilon_N})$ and their Conley-Zehnder indices}
\label{table: polyhedral CZ indices}
\end{table}

None of the orbits in the first two rows of Table \ref{table: polyhedral CZ indices} are contractible, so $\lambda_N=\lambda_{\P^*,\varepsilon_N}$ is $L_N$-dynamically convex and so by \cite[Theorem 1.3]{HN},\footnote{One hypothesis of \cite[Th. 1.3]{HN} requires that all contractible Reeb orbits $\gamma$ satisfying $\mu_{\CZ}(\gamma)=3$ must be embedded. This fails in our case by considering the contractible $\mathcal{V}^{2\mathbf{I}_{\mathscr{V}}}$, which is not embedded yet satisfies $\mu_{\CZ}(\mathcal{V}^{2\mathbf{I}_{\mathscr{V}}})=3$. See Lemma \ref{lemma: generic j} for an explanation of why we do not need this additional hypothesis.} a generic choice $J_N\in\mathcal{J}(\lambda_N)$ provides a well defined filtered chain complex, yielding the isomorphism of $\Z$-graded vector spaces

\begin{align*}
    CH_*^{L_N}(S^3/\P^*,\lambda_N, J_N)&\cong\bigoplus_{i=0}^{2N-1}\Q^{m-2}[2i]\oplus\bigoplus_{i=0}^{2N-2} H_*(S^2;\Q)[2i]\\
    &=\begin{cases} \Q^{m-1} & *=0, \, 4N-2 \\
\Q^{m} & *=2i, 0<i<2N-1 \\ 0 &\mbox{else.}\end{cases}
\end{align*}
 This follows from investigating the good contributions to  $c_i$, which is 0 for odd $i$, implying $\partial^{L_N}=0$. Note that $m=|\text{Conj}(\P^*)|$, and so this supports Theorem \ref{theorem: main} (see Table \ref{table:polyhedral xyz}). Theorem \ref{theorem: cobordisms induce inclusions} of Section \ref{section: direct limits of filtered homology} will permit taking a direct limit over inclusions.

\subsection{Construction of $H$-invariant Morse-Smale functions}\label{subsection: constructing morse functions}

We  now produce the $H$-invariant, Morse-Smale functions on $S^2$ used in the above Sections \ref{subsection: dihedral} and \ref{subsection: polyhedral}, for $H=\D_{2n}$ and $H=\P$ respectively. Table \ref{table: morse points} describes three finite subsets, $X_0$, $X_1$, and $X_2$, of $S^2$ which depend on $H\subset\SO(3)$. We construct an $H$-invariant, Morse-Smale function $f$ on $(S^2, \omega_{\FS}(\cdot, j\cdot))$, whose set of critical points of index $i$ is $X_i$, so that $\text{Crit}(f)=X:=X_0\cup X_1\cup X_2$. Additionally, $X=\text{Fix}(H)$, the fixed point set of the $H$-action on $S^2$. This constructed $f$ is \emph{perfect} in the sense that it features the minimal number of required critical points, because $\text{Fix}(H)\subset\text{Crit}(f)$ must always hold. In the case that $H$ is a polyhedral group, $X_0$ is the set of vertex points, $X_1$ is the set of edge midpoints, and $X_2$ is the set of face barycenters.
\begin{table}[h!]
\centering
 \begin{tabular}{|| c | c | c | c ||} 
 \hline
 $H$ & $X_0$ & $X_1$ & $X_2$\\ [0.5ex] 
 \hline\hline
 $\D_{2n}$ & $\{p_{-,k}\,|\,1\leq k\leq n\}$ & $\{p_{h,k}\,|\,1\leq k\leq n\}$ & $\{p_{+,k}\,|\, 1\leq k\leq 2\}$\\
 \hline
 $\P$ & $\{\mathfrak{v}_k\,|\,1\leq k\leq \mathscr{V}\}$ & $\{\mathfrak{e}_k\,|\,1\leq k\leq \mathscr{E}\}$   & $\{\mathfrak{f}_k\,|\,1\leq k\leq \mathscr{F}\}$ \\
 \hline
\end{tabular}
 \caption{Fixed points of $H$, sorted by Morse index}
 \label{table: morse points}
\end{table}
\begin{lemma}\label{lemma: morse}
Let $H\subset\mbox{\em SO}(3)$ be either $\D_{2n}$ or $\P$. Then there exists an $H$-invariant, Morse function $f$ on $S^2$, with $\mbox{\em Crit}(f)=X$, such that $\mbox{\em ind}_f(p)=i$ if $p\in X_i$. Furthermore, there are stereographic coordinates defined in a small neighborhood of $p\in X$ in which $f$ takes the form  
\begin{enumerate}[(i)]
 \itemsep-.35em
    \item $q_0(x,y):=(x^2+y^2)/2-1$,\,\,\, if $p\in X_0$
    \item $q_1(x,y):=(y^2-x^2)/2$,\,\,\,\,\,\,\,\,\,\,\,\,\, if $p\in X_1$
    \item $q_2(x,y):=1-(x^2+y^2)/2$,\,\,\,\,\,if $p\in X_2$
\end{enumerate}
\end{lemma}
\begin{proof}
We first produce an auxiliary Morse function $\widetilde{f}$ which might not be $H$-invariant, then $f$ is taken to be the $H$-average of $\widetilde{f}$. Fix $\delta>0$. For $p\in X$, let $D_p\subset S^2$ be the open geodesic disc centered at $p$ with radius $\delta$ with respect to the metric $\omega_{\FS}(\cdot, j\cdot)$. Define $\widetilde{f}$ on $D_p$ to be the pullback of $q_0$, $q_1$, or $q_2$, for $p\in X_0$, $X_1$, or $X_2$ respectively, by stereographic coordinates at $p$. Set $D:=\cup_{p\in X}D_p$. For $\delta$ small, $D$ is a \emph{disjoint} union and $\widetilde{f}:D\to\R$ is Morse.  Note that 
\begin{enumerate}[(a)]
    \itemsep-.35em
    \item $D$, the $\delta$-neighborhood of $X$ in $S^2$, is an $H$-invariant set, and for all $p\in X$, the $H_p$-action restricts to an action on $D_p$, where $H_p\subset H$ denotes the stabilizer subgroup,
        \item For all $p\in X$, $\widetilde{f}|_{D_p}:D_p\to\R$ is $H_p$-invariant, and
    \item $\widetilde{f}:D\to\R$ is an $H$-invariant Morse function, with $\text{Crit}(\widetilde{f})=X$.
\end{enumerate}
 The $H$-invariance and $H_p$-invariance in (a) hold,  because $H\subset\SO(3)$ acts on $S^2$ by $\omega_{\FS}(\cdot,j\cdot)$-\emph{isometries} (rotations about axes through $p\in X$), and because $X$ is an $H$-invariant set. The $H_p$-invariance of (b) holds because, in stereographic coordinates, the $H_p$-action pulled back to $\R^2$ is always generated by some linear rotation about the origin, $\mathcal{R}(\theta)$. Both $q_0$ and $q_2$ are invariant with respect to any $\mathcal{R}(\theta)$, whereas $q_1$ is invariant with respect to the action generated by $\mathcal{R}(\pi)$, which is precisely the action by $H_p$ when $p\in X_1$, so (b) holds. Finally, the $H$-invariance in (c) holds directly by the $H_p$ invariance from (b). Now, extend the domain of $\widetilde{f}$ from $D$ to all of $S^2$ so that $\widetilde{f}$ is smooth and Morse, with $\text{Crit}(\widetilde{f})=X$. Figures \ref{figure: dihedral}, \ref{figure: tetrahedral}, \ref{figure: octahedral}, and \ref{figure: icosahedral} depict possible extensions $\tilde{f}$ for each type of symmetry group $H$.

For $h\in H$, let $\phi_h:S^2\to S^2$ denote the group action, $p\mapsto h\cdot p$. Define \[f:=\frac{1}{|H|}\sum_{h\in H}\phi_h^*\widetilde{f},\] where $|H|\in\N$ is the group order of $H$. This $H$-invariant $f$ is smooth and agrees with $\widetilde{f}$ on $D$. If no critical points are created in the averaging process of $\widetilde{f}$, then we have that $\text{Crit}(f)=X$, implying that $f$ is Morse, and we are done.

We say that the extension $\widetilde{f}$ to $S^2$ from $D$ is \emph{roughly} $H$-invariant, if for any $p\in S^2\setminus X$ and $h\in H$, the angle between the nonzero gradient vectors \[\text{grad}(\widetilde{f})\,\,\, \text{and}\,\,\, \text{grad}(\phi_h^*\widetilde{f})\] in $T_pS^2$ is less than $\pi/2$. If $\widetilde{f}$ if roughly $H$-invariant, then  for $p\notin X$, $\text{grad}(f)(p)$ is an average of a collection of nonzero vectors in the same convex half space of $T_pS^2$ and must be nonzero, implying $p\notin\text{Crit}(f)$. That is, if $\widetilde{f}$ is roughly $H$-invariant, then $\text{Crit}(f)=X$, as desired. The extensions $\widetilde{f}$ in Figures \ref{figure: dihedral}, \ref{figure: tetrahedral}, \ref{figure: octahedral}, and \ref{figure: icosahedral} are all roughly $H$-invariant by inspection, and the proof is complete. 
\end{proof}

\begin{figure}[!htb]
    \centering
    \begin{minipage}{.5\textwidth}
        \centering
        \includegraphics[width=0.9\textwidth]{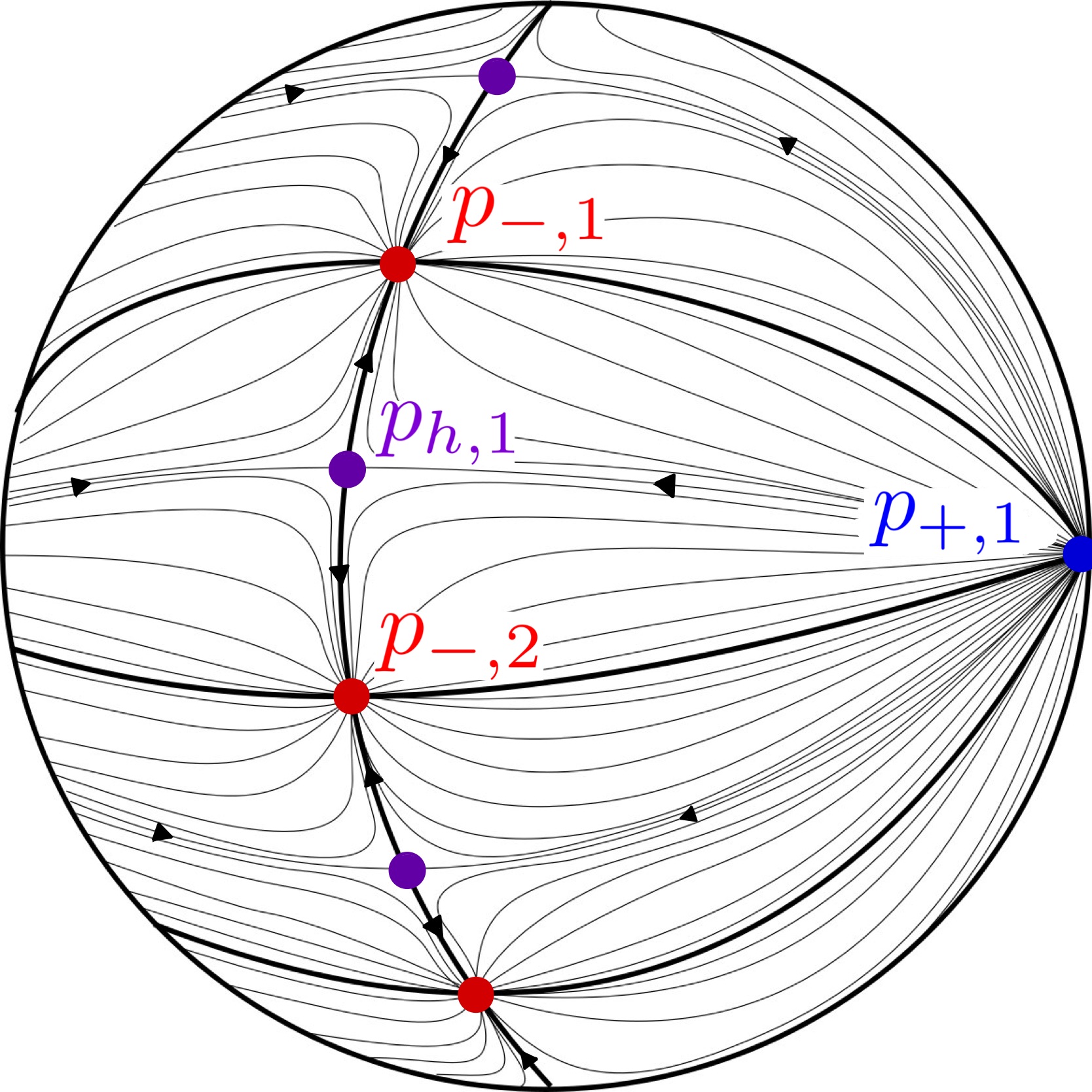}
        \caption{A possible dihedral $\widetilde{f}$}
        \label{figure: dihedral}
    \end{minipage}%
    \begin{minipage}{0.5\textwidth}
        \centering
        \includegraphics[width=0.9\textwidth]{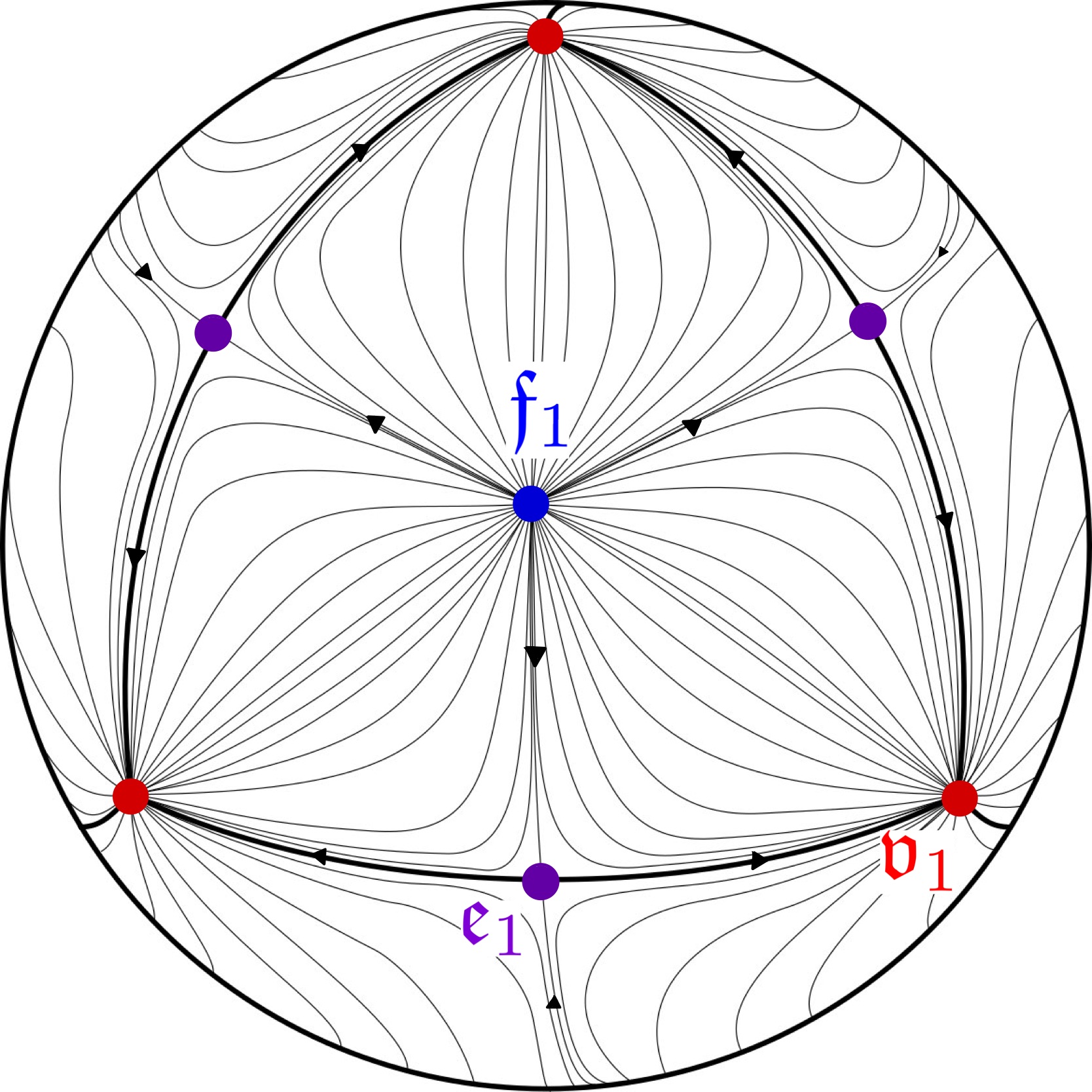}
        \caption{A possible tetrahedral $\widetilde{f}$}
        \label{figure: tetrahedral}
    \end{minipage}
\end{figure}

\begin{figure}[!htb]
    \centering
    \begin{minipage}{.5\textwidth}
        \centering
        \includegraphics[width=0.9\textwidth]{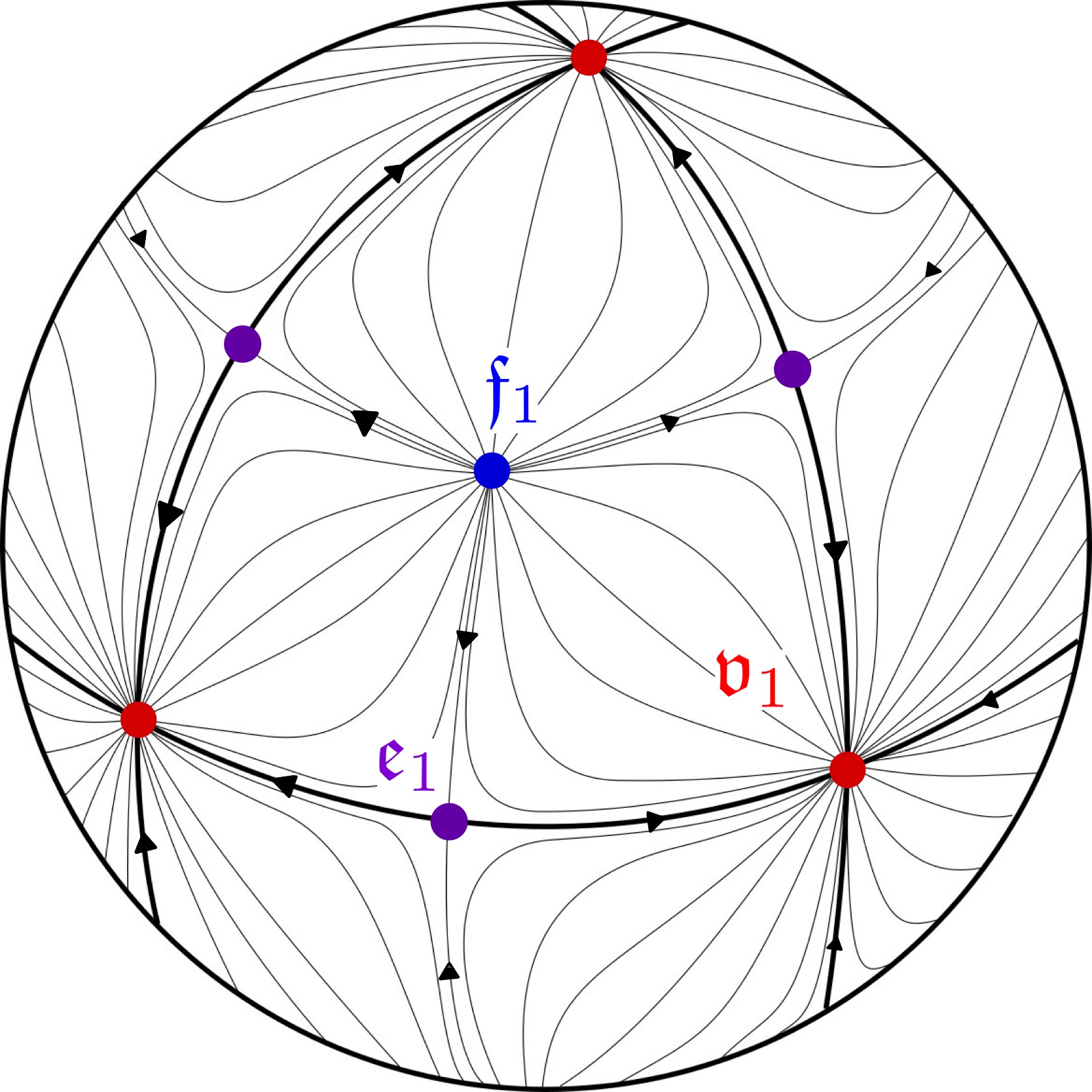}
        \caption{A possible octahedral $\widetilde{f}$}
        \label{figure: octahedral}
    \end{minipage}%
    \begin{minipage}{0.5\textwidth}
        \centering
        \includegraphics[width=0.9\textwidth]{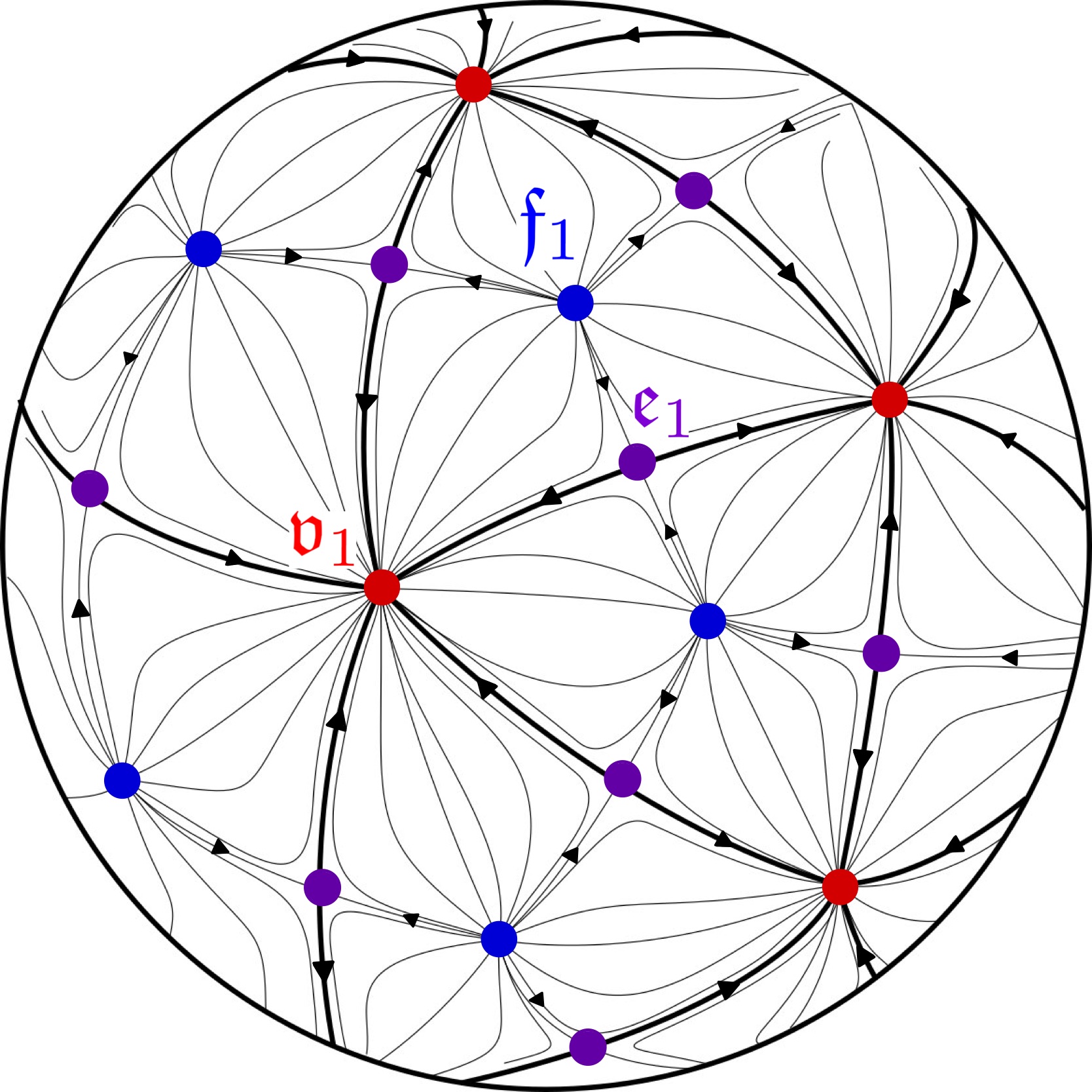}
        \caption{A possible icosahedral $\widetilde{f}$}
        \label{figure: icosahedral}
    \end{minipage}
\end{figure}
In each of Figures \ref{figure: dihedral}, \ref{figure: tetrahedral}, \ref{figure: octahedral}, and \ref{figure: icosahedral}, the blue, violet, and red critical points are of Morse index 2, 1, and 0 respectively.

\begin{lemma}\label{lemma: smale}
If $f$ is a Morse function on a 2-dimensional manifold $S$ such that $f(p_1)=f(p_2)$ for all $p_1, p_2\in\mbox{\em Crit}(f)$ with Morse index 1, then $f$ is Smale, given any metric on $S$.
\end{lemma}
\begin{proof}
Given metric $g$ on $S$, $f$ fails to be Smale with respect to $g$ if and only if there are two distinct critical points of $f$ of Morse index 1 that are connected by a gradient flow line of $f$. Because all such critical points have the same $f$ value, no such flow line exists.
\end{proof}
By Lemma \ref{lemma: smale}, the Morse function $f$ provided in Lemma \ref{lemma: morse} is Smale for $\omega_{\FS}(\cdot, j\cdot)$.

\subsection{Orbifold and contact interplays: an example}\label{subsection: visualizing holomorphic cylinders: an example}
Before we compute direct limits of filtered homology groups in Section \ref{section: direct limits of filtered homology}, we  demonstrate the analogies between the contact data of $S^3/G$ and the orbifold Morse data of $S^2/H$, described in Section \ref{subsection: cylindrical contact homology as an analogue of orbifold Morse homology} through examples and illustrations. 

As previously mentioned, for any finite $H\subset\SO(3)$, the quotient $S^2/H$ is an orbifold 2-sphere. When $H$ is cyclic, the orbifold 2-sphere $S^2/H$ resembles a lemon shape, featuring two orbifold points, and is immediately homeomorphic to $S^2$. If not cyclic, $H$ is dihedral, or polyhedral. A fundamental domain for the $H$-action on $S^2$ in these two latter cases can be taken to be an isosceles, closed, geodesic triangle, denoted $\Delta_{H}\subset S^2$. These geodesic triangles $\Delta_{H}$ are identified by the shaded regions of $S^2$ in Figures \ref{figure: dihedral fundamental domain} (for $H=\D_{2n}$) and \ref{figure: tetrahedral fundamental domain} (For $H=\T$);
\begin{figure}[!htb]
    \centering
    \begin{minipage}{.5\textwidth}
        \centering
        \includegraphics[width=0.8\textwidth]{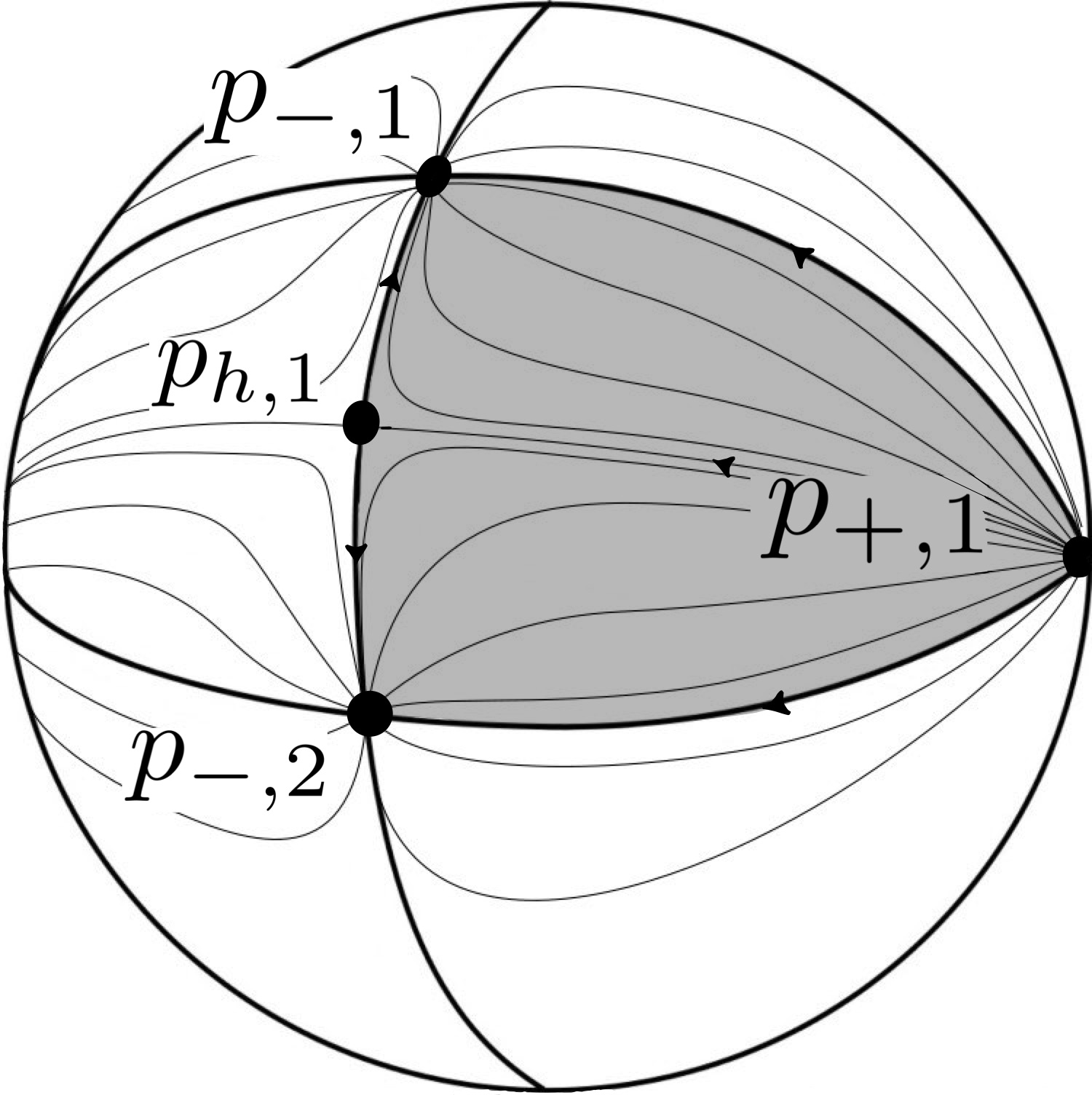}
        \caption{Fundamental domain $\Delta_{\D_{2n}}$}
        \label{figure: dihedral fundamental domain}
    \end{minipage}%
    \begin{minipage}{0.5\textwidth}
        \centering
        \includegraphics[width=0.8\textwidth]{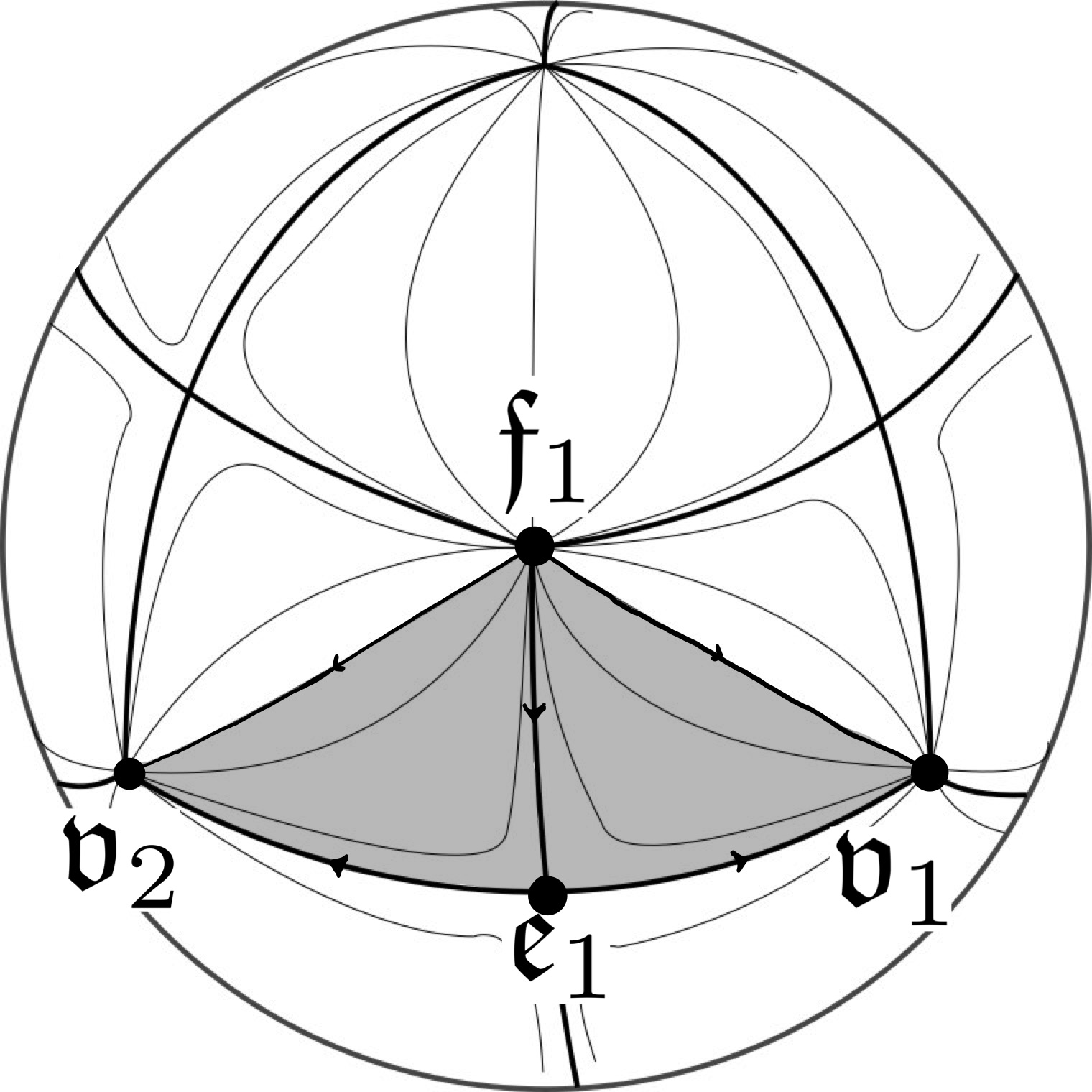}
        \caption{Fundamental domain $\Delta_{\T}$}
        \label{figure: tetrahedral fundamental domain}
    \end{minipage}
\end{figure}

\begin{itemize}
    \item In Figure \ref{figure: dihedral fundamental domain}, we can take the three vertices of $\Delta_{\D_{2n}}$ to be $p_{+,1}$, $p_{-,1}$, and $p_{-,2}$. We have that $S^2$ is tessellated by $|\D_{2n}|=2n$ copies of $\Delta_{\D_{2n}}$ and the internal angles of $\Delta_{\D_{2n}}$ are $\pi/2$, $\pi/2$, and $2\pi/n$; $\Delta_{\D_{2n}}$ is isosceles.
    \item In Figure \ref{figure: tetrahedral fundamental domain}, we can take the three vertices of $\Delta_{\T}$ to be $\mathfrak{f}_1$, $\mathfrak{v}_1$, and $\mathfrak{v}_2$. We have that $S^2$ is tessellated by $|\T|=12$ copies of $\Delta_{\T}$, and the internal angles of $\Delta_{\T}$ are $\pi/3$, $\pi/3$, and $2\pi/3$; $\Delta_{\T}$ is isosceles.
\end{itemize}

The triangular fundamental domains $\Delta_{\Oc}$ and $\Delta_{\I}$ for the $\Oc$ and $\I$ actions on $S^2$ are constructed  analogously to $\Delta_{\T}$. Ultimately, in every (non-cyclic) case, we have a closed, isosceles, geodesic triangle serving as a fundamental domain for the $H$-action on $S^2$. Applying the $H$-identifications on the boundary of $\Delta_{H}$ produces $S^2/H$, a quotient that is homeomorphic to $S^2$ with three orbifold points. Specifically, under the surjective quotient map restricted to the closed $\Delta_H$ (depicted in Figure \ref{figure: fundamental domain}), \[\pi_{H}|_{\Delta_H}:\Delta_{H}\subset S^2\to S^2/H,\]
\begin{itemize}
    \item (blue maximum) one orbifold point of $S^2/H$ has a preimage consisting of a single vertex of $\Delta_H$, this is an index 2 critical point in $S^2$,
    \item (violet saddle) one orbifold point of $S^2/H$ has a preimage consisting of a single midpoint of an edge of $\Delta_H$, this is an index 1 critical point in $S^2$,
    \item (red minimum) one orbifold point of $S^2/H$ has a preimage consisting of two vertices of $\Delta_H$, both are index 0 critical points in $S^2$. 
\end{itemize}
Figure \ref{figure: fundamental domain} depicts the attaching map for $\Delta_H$ along the boundary and these points.

\begin{figure}[h]
    \centering
    \includegraphics[width=0.8\textwidth]{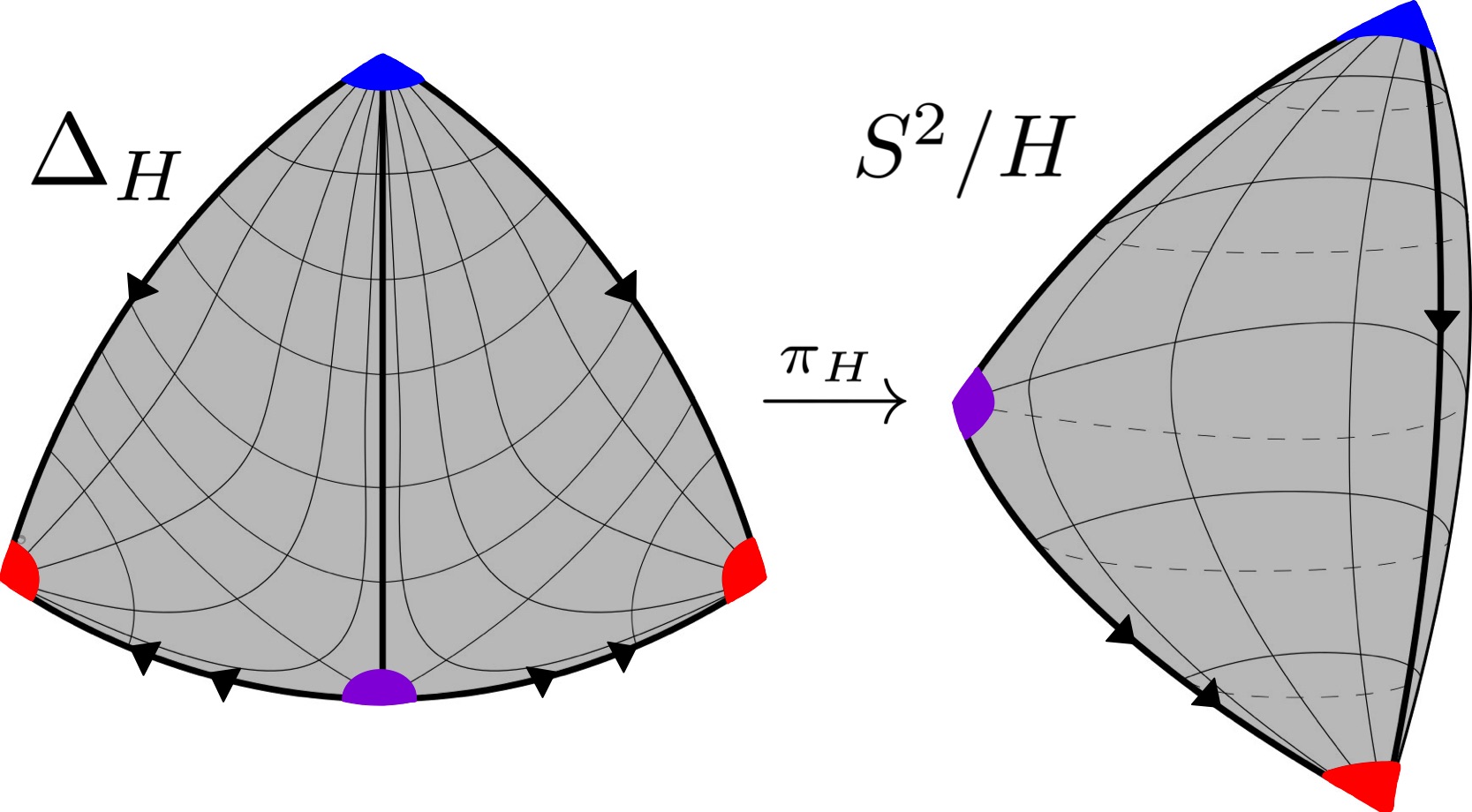}
    \caption{A triangular fundamental domain $\Delta_H$ produces $S^2/H$, a topological $S^2$ with three orbifold points (when $H$ is non-cyclic). Note that that the directional markers on the boundary of $\Delta_H$ indicating the $H$-identifications simultaneously depict orbifold Morse trajectories.}
    \label{figure: fundamental domain}
\end{figure}

We now specialize to the case $H=\T$ and study the geometry of $S^3/\T^*$ and $S^2/\T$. This choice $H=\T$ makes the examples and diagrams concrete - note that a choice of $H=\D$, $\Oc$, or $\I$ produces similar geometric scenarios. From Section \ref{subsection: polyhedral}, the three orbifold points of $S^2/\T$ are denoted $\mathfrak{v}$, $\mathfrak{e}$, and $\mathfrak{f}$, which are critical points of the orbifold Morse function $f_{\T}$ of index 0, 1, and 2, respectively. Furthermore, for $\varepsilon>0$ small, we have three embedded nondegenerate Reeb orbits, $\mathcal{V}$, $\mathcal{E}$, and $\mathcal{F}$ in $S^3/\T^*$ of $\lambda_{\T^*,\varepsilon}$, projecting to the respective orbifold critical points under $\fp:S^3/\T^*\to S^2/\T$. Figure \ref{figure: calzone} illustrates this data. 

\begin{figure}[h]
    \centering
    \includegraphics[width=0.85\textwidth]{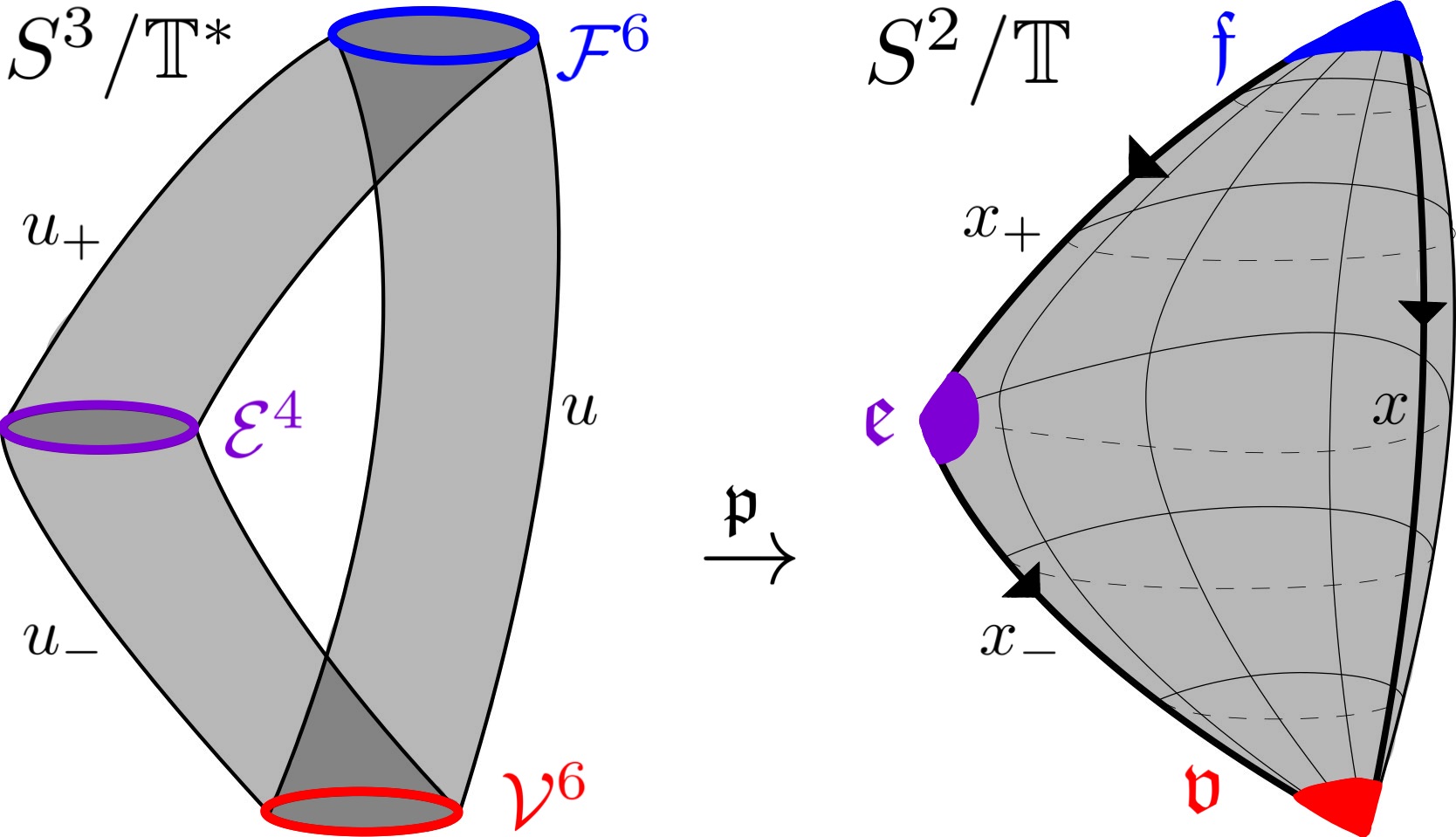}
    \caption{The Seifert projection $\fp$ takes Reeb orbits and cylinders of $S^3/\T^*$ to orbifold critical points and Morse trajectories of $S^2/\T$. The depicted cylinders in $S^3/\T^*$ should be understood as the images of infinite cylinders in the symplectization under the projection $S^3/\T^*\times\R\to S^3/\T^*$.}
    \label{figure: calzone}
\end{figure}

In Section \ref{subsection: cylindrical contact homology as an analogue of orbifold Morse homology} we argued that bad Reeb orbits in cylindrical contact homology are analogous to non-orientable critical points in orbifold Morse theory. We can see how $\fp:S^3/\T^*\to S^2/\T$ geometrically realizes this analogy. In Section \ref{subsection: polyhedral}, we saw that the even iterates $\mathcal{E}^{2k}$ are examples of bad Reeb orbits, and by Remark \ref{remark: nonorientable1}, $\mathfrak{e}$ is a non-orientable critical point of $f_{\T}$. The projection $\fp$ directly takes the bad Reeb orbits $\mathcal{E}^{2k}$ to the non-orientable critical point $\mathfrak{e}$.

 Next we consider the analogies between the moduli spaces. The orders of $\mathcal{V}$, $\mathcal{E}$, and $\mathcal{F}$ in $\pi_1(S^3/\T^*)$ are 6, 4, and 6, respectively (see Section \ref{subsubsection: binary polyhedral}). Thus, by the correspondence \eqref{equation: induced cylinder 2} in Section \ref{subsection: cylinders over orbifold Morse trajectories}, we have the following identifications between moduli spaces of orbifold Morse flow lines of $S^2/\T$ and $J$-holomorphic cylinders in $\R\times S^3/\T^*$ (where we use the  complex structure $J$ described in Remark \ref{remark: the specific almost complex structure}):
\begin{align}
    \mathcal{M}(\mathfrak{f},\mathfrak{e}) &\cong\mathcal{M}^J(\mathcal{F}^6,\mathcal{E}^4)/\R \label{equation: correspondence +}\\
    \mathcal{M}(\mathfrak{e},\mathfrak{v}) &\cong\mathcal{M}^J(\mathcal{E}^4,\mathcal{V}^6)/\R \label{equation: correspondence -}\\
    \mathcal{M}(\mathfrak{f},\mathfrak{v}) &\cong\mathcal{M}^J(\mathcal{F}^6,\mathcal{V}^6)/\R. \label{equation: correspondence 0}
\end{align}
Correspondences \eqref{equation: correspondence +} and \eqref{equation: correspondence -} are between singleton sets. Indeed, let $x_+$ be the unique orbifold Morse flow line from $\mathfrak{f}$ to $\mathfrak{e}$ in $S^2/\T$, and let $x_-$ be the unique orbifold Morse flow line from $\mathfrak{e}$ to $\mathfrak{v}$, depicted in Figure \ref{figure: calzone}. Then we have corresponding cylinders $u_+$ and $u_-$ from $\mathcal{F}^6$ to $\mathcal{E}^4$, and from $\mathcal{E}^4$ to $\mathcal{V}^6$, respectively:

\begin{align}
    \{x_+\}=\mathcal{M}(\mathfrak{f},\mathfrak{e}) &\cong\mathcal{M}^J(\mathcal{F}^6,\mathcal{E}^4)/\R=\{u_+\}\label{equation: singleton1}\\
    \{x_-\}=\mathcal{M}(\mathfrak{e},\mathfrak{v}) &\cong\mathcal{M}^J(\mathcal{E}^4,\mathcal{V}^6)/\R=\{u_-\}\label{equation: singleton2}
\end{align}
These cylinders $u_{\pm}$ are depicted in Figure \ref{figure: calzone}. Additionally, note that the indices of the corresponding objects agree:
\begin{align*}
    \text{ind}(x_+)&=\text{ind}_{f_{\T}}(\mathfrak{f})-\text{ind}_{f_{\T}}(\mathfrak{e})=2-1=1\\
    \text{ind}(u_+)&=\mu_{\CZ}(\mathcal{F}^6)-\mu_{\CZ}(\mathcal{E}^4)=5-4=1,
\end{align*}
and 
\begin{align*}
    \text{ind}(x_-)&=\text{ind}_{f_{\T}}(\mathfrak{e})-\text{ind}_{f_{\T}}(\mathfrak{v})=1-0=1\\
    \text{ind}(u_-)&=\mu_{\CZ}(\mathcal{E}^4)-\mu_{\CZ}(\mathcal{V}^6)=4-3=1.
\end{align*}

Next, consider the third correspondence of moduli spaces, \eqref{equation: correspondence 0}. From Figure \ref{figure: calzone}, we can geometrically see that $\mathcal{M}(\mathfrak{f},\mathfrak{v})$ is diffeomorphic to a 1-dimensional open interval. Indeed, for any $x\in\mathcal{M}(\mathfrak{f},\mathfrak{v})$, we have that \[\text{ind}(x)=\text{ind}_{f_{\T}}(\mathfrak{f})-\text{ind}_{f_{\T}}(\mathfrak{v})=2,\] algebraically verifying that the moduli space of orbifold flow lines must be 1-dimensional. On the other hand, take any cylinder $u\in\mathcal{M}^J(\mathcal{F}^6,\mathcal{V}^6)/\R$. Now, \[\text{ind}(u)=\mu_{\CZ}(\mathcal{F}^6)=\mu_{\CZ}(\mathcal{V}^6)=5-3=2,\] verifying that this moduli space of cylinders is $1$-dimensional, as expected by \eqref{equation: correspondence 0}. 

Note that both of the open, 1-dimensional moduli spaces in \eqref{equation: correspondence 0} admit a compactification by broken objects. We can see explicitly from Figure \ref{figure: calzone} that both ends of the 1-dimensional moduli space $\mathcal{M}(\mathfrak{f},\mathfrak{v})$ converge to the \emph{same} once-broken orbifold Morse trajectory, $(x_+,x_-)\in\mathcal{M}(\mathfrak{f},\mathfrak{e})\times\mathcal{M}(\mathfrak{e},\mathfrak{v})$. In particular, the compactification $\overline{\mathcal{M}(\mathfrak{f},\mathfrak{v})}$ is a topological $S^1$, obtained by adding the single point $(x_+,x_-)$ to an open interval, and we write \[\overline{\mathcal{M}(\mathfrak{f},\mathfrak{v})}=\mathcal{M}(\mathfrak{f},\mathfrak{v})\,\bigsqcup\, \bigg(\mathcal{M}(\mathfrak{f},\mathfrak{e})\times\mathcal{M}(\mathfrak{e},\mathfrak{v})\bigg). \]

An identical phenomenon occurs for the compactification of the cylinders. That is, both ends of the 1-dimensional interval $\mathcal{M}^J(\mathcal{F}^6,\mathcal{V}^6)/\R$ converge to the \emph{same} once broken building of cylinders $(u_+,u_-)$. The compactification $\overline{\mathcal{M}^J(\mathcal{F}^6,\mathcal{V}^6)/\R}$ is a topological $S^1$, obtained by adding a single point $(u_+,u_-)$ to an open interval, and we write \[\overline{\mathcal{M}^J(\mathcal{F}^6,\mathcal{V}^6)/\R}=\mathcal{M}^J(\mathcal{F}^6,\mathcal{V}^6)/\R\,\bigsqcup\,\bigg(\mathcal{M}^J(\mathcal{F}^6,\mathcal{E}^4)/\R\times\mathcal{M}^J(\mathcal{E}^4,\mathcal{V}^6/\R)\bigg).\]

Recall that in Section \ref{subsection: cylindrical contact homology as an analogue of orbifold Morse homology}, we argued that the differentials of cylindrical contact homology and orbifold Morse homology are structurally identical due to the similarities in the compactifications of the 1-dimensional moduli spaces. This is due to the fact that in both theories, a once broken building can serve as a limit of \emph{multiple} ends of a 1-dimensional moduli space. Our examples depict this phenomenon:
\begin{itemize}
    \item The broken building $(x_+,x_-)$ of orbifold Morse flow lines serves as the limit of \emph{both} ends of the open interval $\mathcal{M}(\mathfrak{f},\mathfrak{v})$.
    \item The broken building $(u_+,u_-)$ of $J$-holomorphic cylinders serves as the limit of \emph{both} ends of the open interval $\mathcal{M}^J(\mathcal{F}^6,\mathcal{V}^6)/\R$.
\end{itemize}

Another analogy is highlighted in this example. In both homology theories, it is possible for a sequence of flow lines/cylinders between orientable objects can break along an intermediate non-orientable object (see \cite[Example 2.10]{CH}). For example:
\begin{itemize}
    \item There is a  sequence $x_n\in\mathcal{M}(\mathfrak{f},\mathfrak{v})$ converging to the broken building $(x_+,x_-)$, which breaks at $\mathfrak{e}$. The critical points $\mathfrak{f}$ and $\mathfrak{v}$ are orientable, whereas $\mathfrak{e}$ is non-orientable.
    \item There is a  sequence $u_n\in\mathcal{M}^J(\mathcal{F}^6,\mathcal{V}^6)/\R$ converging to the broken building $(u_+,u_-)$, which breaks along the orbit $\mathcal{E}^4$. The Reeb orbits $\mathcal{F}^6$ and $\mathcal{V}^6$ are good, whereas $\mathcal{E}^4$ is bad.
\end{itemize}

Recall from Remark \ref{remark: nonorientable2}, that we cannot assign a value $\epsilon(x_{\pm})\in\{\pm1\}$ nor can we assign a value $\epsilon(u_{\pm})\in\{\pm1\}$ to the objects $x_{\pm}$ and $u_{\pm}$, because they have a non-orientable limiting object. This difficulty in general  provides a hiccup in showing that $\partial^2=0$, as we cannot write the signed count as a sum involving terms of the form $\epsilon(x_+)\epsilon(x_-)$ or $\epsilon(u_+)\epsilon(u_-)$. One gets around this by showing that a once broken building with break at a non-orientable object is utilized by an \emph{even} number of ends of the 1-dimensional moduli space, and that a cyclic group action on this set of even number of ends interchanges the orientations. Using this fact, one shows $\partial^2=0$ (see \cite[Remark 5.3]{CH} and \cite[\S 4.4]{HN}). We see this explicitly in our examples:
\begin{itemize}
    \item The once broken building $(x_+,x_-)$ is the limit of \emph{two} ends of the 1-dimensional moduli space of orbifold trajectories; one positive and one negative end.
    \item The once broken building $(u_+,u_-)$ is the limit of \emph{two} ends of the 1-dimensional moduli space of orbifold trajectories; one positive and one negative end.
\end{itemize}
\begin{remark}\label{remark: when d squared is not zero} (Including non-orientable objects complicates $\partial^2=0$) Recall in Section \ref{subsection: cylindrical contact homology as an analogue of orbifold Morse homology} we mentioned that one discards bad Reeb orbits and non-orientable orbifold critical points as generators in the hopes of achieving $\partial^2=0$ in both theories. Using our understanding of moduli spaces from  Figure \ref{figure: calzone}, we show why $\partial^2=0$ could not reasonably hold if we were to include $\mathfrak{e}$ and $\mathcal{E}^4$ in the corresponding chain complexes. Suppose that we have some coherent way of assigning $\pm1$ to the trajectories $x_{\pm}$ and cylinders $u_{\pm}$. Now, due to equations \eqref{equation: singleton1} and \eqref{equation: singleton2}, we would have in the orbifold case
\begin{align*}
    \partial^{\text{orb}}\mathfrak{f}&=\dfrac{\epsilon(x_+)|\T_{\mathfrak{f}}|}{|\T_{x_+}|}\mathfrak{e}=3\epsilon(x_+)\mathfrak{e}\\
    \partial^{\text{orb}}\mathfrak{e}&=\dfrac{\epsilon(x_-)|\T_{\mathfrak{e}}|}{|\T_{x_-}|}\mathfrak{v}=2\epsilon(x_-)\mathfrak{v}\\
    \implies &\langle\partial^{\text{orb}}\mathfrak{f},\mathfrak{v}\rangle=6\epsilon(x_+)\epsilon(x_-)\neq0,
\end{align*}
where we have used that $|\T_{\mathfrak{f}}|=3$, $|\T_{\mathfrak{e}}|=2$, and $|\T_{x_{\pm}}|=1$. Similarly, again due to equations \eqref{equation: singleton1} and \eqref{equation: singleton2}, we would have
\begin{align*}
    \partial\mathcal{F}^6&=\dfrac{\epsilon(u_+)m(\mathcal{F}^6)}{m(u_+)}\mathcal{E}^4=3\epsilon(u_+)\mathcal{E}^4\\
    \partial\mathcal{E}^4&=\dfrac{\epsilon(u_-)m(\mathcal{E}^4)}{m(u_-)}\mathcal{V}^6=2\epsilon(u_-)\mathcal{V}^6\\
    \implies &\langle\partial\mathcal{F}^6,\mathcal{V}^6\rangle=6\epsilon(u_+)\epsilon(u_-)\neq0,
\end{align*}
where we have used that $m(\mathcal{F}^6)=6$, $m(\mathcal{E}^4)=4$, and $m(u_{\pm})=2$.
\end{remark}

Recall that the multiplicity of a $J$-holomorphic cylinder $u$ divides the multiplicity of the limiting Reeb orbits $\gamma_{\pm}$. The following remark uses the example of this section to demonstrate it need not be the case that $m(u)=\text{GCD}(m(\gamma_+),m(\gamma_-))$. 
\begin{remark}\label{remark: multiplicity of u is not always gcd} 
By \eqref{equation: correspondence 0}, $\mathcal{M}^J(\mathcal{F}^6,\mathcal{V}^6)/\R$ is nonempty, and must contain some $u$. We see that $m(\mathcal{F}^6)=m(\mathcal{V}^6)=6$, so that \[\text{GCD}(m(\mathcal{F}^6),m(\mathcal{V}^6))=6.\] Suppose for contradiction's sake that that $m(u)=6$, and consider the underlying somewhere injective $J$-holomorphic cylinder $v$. It must be the case that $u=v^6$ and that $v\in\mathcal{M}^J(\mathcal{F},\mathcal{V})/\R$. The existence of such a $v$ implies that $\mathcal{F}$ and $\mathcal{V}$ represent the same free homotopy class of loops in $S^3/\T^*$. However, we will determine in Section \ref{subsubsection: binary polyhedral} that these Reeb orbits represent distinct homotopy classes (See Table \ref{table: Tetrahedral homotopy classes of Reeb orbits}). Thus, $m(u)$ is not equal to the GCD.
\end{remark}

\section{Direct limits of filtered  cylindrical contact homology}\label{section: direct limits of filtered homology}
In the previous section we obtained a description for the  filtered cylindrical contact homology of the links of the simple singularities. For any finite, nontrivial subgroup $G\subset\SU(2)$, we have a sequence $(L_N,\lambda_N,J_N)_{N=1}^{\infty}$, where $L_N\to\infty$ monotonically in $\R$, $\lambda_N$ is an $L_N$-dynamically convex contact form on $S^3/G$ with kernel $\xi_G$, and $J_N\in\mathcal{J}(\lambda_N)$ is generically chosen so that 
\begin{equation}\label{equation: identification}
    CH_*^{L_N}(S^3/G,\lambda_N, J_N)\cong\bigoplus_{i=0}^{2N-1}\Q^{m-2}[2i]\oplus\bigoplus_{i=0}^{2N-2} H_*(S^2;\Q)[2i],
\end{equation}
where $m=|\text{Cong}(G)|$. For $N\leq M$, there is a natural inclusion of the vector spaces on the right hand side of \eqref{equation: identification}. Theorem \ref{theorem: cobordisms induce inclusions} yields that symplectic cobordisms induce well defined maps on filtered homology, which can be identified with these inclusions. Thus,
\[\varinjlim_N CH_*^{L_N}(S^3/G,\lambda_N, J_N)\cong\bigoplus_{i\geq0}\Q^{m-2}[2i]\oplus\bigoplus_{i\geq0} H_*(S^2;\Q)[2i].\]

We provide an overview of cobordisms and the maps they induce now. An \emph{exact symplectic cobordism} is a compact symplectic manifold with boundary $\overline{X}$, equipped with a Liouville form $\lambda$. The exact symplectic form $d\lambda$ orients $\overline{X}$ and, given contact manifolds $(Y_{\pm},\lambda_{\pm})$, we say that $\overline{X}$ is \emph{an exact symplectic cobordism from $(Y_+,\lambda_+)$ to $(Y_-,\lambda_-)$} if both $\partial \overline{X}=Y_+\sqcup-Y_-$ as oriented manifolds and $\lambda|_{Y_{\pm}}=\lambda_{\pm}$. For a generic \emph{cobordism compatible} almost complex structure $J$ on the \emph{completed symplectic cobordism} $X$, matching $J_{\pm}\in\mathcal{J}(\lambda_{\pm})$ on the positive and negative ends, respectively, one hopes that the moduli spaces of Fredholm index zero $J$-holomorphic cylinders in $X$, denoted $\mathcal{M}_0^J(\gamma_+,\gamma_-)$, are compact 0-manifolds for $\gamma_{\pm}\in\mathcal{P}_{\text{good}}(\lambda_{\pm})$. If this holds, then the map $\Phi$, defined on generators  $\gamma_{\pm}\in\mathcal{P}_{\text{good}}(\lambda_{\pm})$ by \[\Phi:CC_*(Y_+,\lambda_+,J_+)\to CC_*(Y_-,\lambda_-,J_-),\,\,\,\,\,\,\,\langle\Phi(\gamma_+),\gamma_-\rangle:=\sum_{u\in\mathcal{M}_0^J(\gamma_+,\gamma_-)}\epsilon(u)\frac{m(\gamma_+)}{m(u)},\]
is well defined, because the sums are finite. That $\Phi$ is a chain map follows from a careful analysis of the moduli spaces of index 1 cylinders that appear in these completed cobordisms. These chain maps induce \emph{continuation homomorphisms} on the cylindrical contact homology groups. Our main result in Section \ref{section: direct limits of filtered homology} is the following:

\begin{theorem}\label{theorem: cobordisms induce inclusions}
A generic exact completed symplectic cobordism $(X,\lambda, J)$ from $(S^3/G,\lambda_N,J_N)$ to $(S^3/G,\lambda_M,J_M)$, for $N\leq M$, induces a well defined chain map between filtered chain complexes. The induced maps $\Psi_N^M$ on homology \[\Psi_N^M:CH_*^{L_N}(S^3/G,\lambda_N,J_N)\to CH_*^{L_M}(S^3/G,\lambda_M,J_M)\] may be identified with the standard inclusions \[\bigoplus_{i=0}^{2N-1}\Q^{m-2}[2i]\oplus\bigoplus_{i=0}^{2N-2} H_*(S^2;\Q)[2i]\hookrightarrow \bigoplus_{i=0}^{2M-1}\Q^{m-2}[2i]\oplus\bigoplus_{i=0}^{2M-2} H_*(S^2;\Q)[2i]\]
and form a directed system of graded $\Q$-vector spaces over $\N$.
\end{theorem}

The most involved step in proving Theorem  \ref{theorem: cobordisms induce inclusions} is to argue compactness of the 0-dimensional moduli spaces, shown in Section \ref{subsection: holomorphic buildings in cobordisms},  supported via a case-by-case study of Reeb orbits according to their free homotopy classes in Section \ref{subsection: free homotopy classes represented by reeb orbits}. We will make use of the following notation in this section. \begin{notation}\label{notation: equivalence of reeb orbits}
For $\gamma_+\in\mathcal{P}^{L_N}(\lambda_N)$ and $\gamma_-\in\mathcal{P}^{L_M}(\lambda_M)$, we write $\gamma_+\sim\gamma_-$ whenever $m(\gamma_+)=m(\gamma_-)$ and both orbits project to the same orbifold point  under $\fp$. If either condition does not hold, we write $\gamma_+\nsim\gamma_-$.  Note that $\sim$ defines an equivalence relation on the disjoint union $\bigsqcup_{N\in\N}\mathcal{P}^{L_N}(\lambda_N)$. If $\gamma$ is contractible in $S^3/G$, then we write $[\gamma]=0$ in $[S^1,S^3/G]$.
\end{notation}
The proof of Theorem \ref{theorem: cobordisms induce inclusions} is as follows. Automatic transversality will be used in Corollary \ref{corollary: moduli spaces are finitie} to prove that $\mathcal{M}_0^J(\gamma_+,\gamma_-)$ is a 0-dimensional manifold. By  Proposition \ref{proposition: buildings in cobordisms},  $\mathcal{M}_0^J(\gamma_+,\gamma_-)$ will be shown to be compact. From this, one concludes that $\mathcal{M}_0^J(\gamma_+,\gamma_-)$ is a finite set for $\gamma_+\in\mathcal{P}_{\text{good}}^{L_N}(\lambda_N)$ and $\gamma_-\in\mathcal{P}_{\text{good}}^{L_M}(\lambda_M)$.

This finiteness tells us that, for $N\leq M$, we have the well defined chain map: \[\Phi_N^M:\big(CC_*^{L_N}(S^3/G,\lambda_N,J_N),\partial^{L_N}\big)\to\big(CC_*^{L_N}(S^3/G,\lambda_M,J_M),\partial^{L_N}\big).\] If $\gamma_+\sim\gamma_-$, then by the implicit function theorem, the weighted count $\langle \Phi_N^M(\gamma_+),\gamma_-\rangle$ of $\mathcal{M}_0^J(\gamma_+,\gamma_-)$, the finite set of  cylinders in $X$,  equals that of  $\mathcal{M}^{J_N}(\gamma_+,\gamma_+)$ and that of $\mathcal{M}^{J_M}(\gamma_-,\gamma_-)$, the moduli spaces of cylinders in the symplectizations of $\lambda_N$ and $\lambda_M$, which is known to be 1, given by the  contribution of a single trivial cylinder. 

If $\gamma_+\nsim\gamma_-$, then Corollary \ref{corollary: moduli spaces are finitie} will imply that $\mathcal{M}_0^J(\gamma_+,\gamma_-)$ is empty, and so $\langle \Phi_N^M(\gamma_+),\gamma_-\rangle=0$. Ultimately we conclude that, given $\gamma_+\in\mathcal{P}_{\text{good}}^{L_N}(\lambda_N)$, our chain map $\Phi_N^M$ takes the form $\Phi_N^M(\gamma_+)=\gamma_-$, where $\gamma_-\in \mathcal{P}_{\text{good}}^{L_N}(\lambda_M)$ is the unique Reeb orbit satisfying $\gamma_+\sim\gamma_-$.

We let $\iota_N^M$ denote the chain map given by the inclusion of subcomplexes:
\[\iota_N^M:\big(CC_*^{L_N}(S^3/G,\lambda_M,J_M),\partial^{L_N}\big)\hookrightarrow\big(CC_*^{L_M}(S^3/G,\lambda_M,J_M),\partial^{L_M}\big).\]
The composition $\iota_N^M\circ\Phi_N^M$ a chain map. Let $\Psi_N^M$ denote the map on homology induced by this composition, that is, $\Psi_N^M=(\iota_N^M\circ\Phi_N^M)_*$. We see that \[\Psi_N^M:CH_*^{L_N}(S^3/G,\lambda_N,J_N)\to CH_*^{L_M}(S^3/G,\lambda_M,J_M)\]
satisfies $\Psi_N^M([\gamma_+])=[\gamma_-]$ whenever $\gamma_+\sim\gamma_-$, and thus $\Psi_N^M$ takes the form of the standard inclusions after making the identifications \eqref{equation: identification}.

\subsection{Holomorphic buildings in cobordisms}\label{subsection: holomorphic buildings in cobordisms}

A \emph{holomorphic building} $B$ is a list $(u_1,u_2,\dots,u_n)$, where each $u_i$ is a potentially disconnected holomorphic curve in a completed symplectic cobordism between two closed contact manifolds equipped with a cobordism compatible almost complex structure. The curve $u_i$ is the $i^{\text{th}}$ \emph{level of} $B$. Each $u_i$ has a set of positive (negative) ends which are positively (negatively) asymptotic to a set of  Reeb orbits. For  $i\in\{1,\dots, n-1\}$, there is a bijection between the negative ends of $u_i$ and the positive ends of $u_{i+1}$, such that paired ends are asymptotic to the same Reeb orbit. The \emph{height} of $B$ is $n$, and a \emph{positive end} (\emph{negative end}) of $B$ is a positive end (negative end) of $u_1$ (of $u_n$). The \emph{genus} of $B$ is the genus of the Riemann surface $S$ obtained by attaching ends of the domains of the $u_i$ to those of $u_{i+1}$ according to the bijections; $B$ is \emph{connected} if this $S$ is connected. The \emph{index} of $B$, $\text{ind}(B)$, is defined to be $\sum_{i=1}^n\text{ind}(u_i)$.

All buildings that we will study will have at most one level in a nontrivial symplectic cobordism, the rest in a symplectization.  Unless otherwise stated, we require that no level of $B$ is a union of trivial cylinders (see Definition \ref{definition: trivial cylinders}) in a symplectization (that is,  $B$ is \emph{without trivial levels}).

\begin{remark}\label{remark: index of buildings}
(Index of buildings) The index of a connected,  genus 0 building $B$ with one positive end asymptotic to $\alpha$, and with $k$ negative ends, whose $i^{\text{th}}$ negative end is  asymptotic to $\beta_i$, is \begin{equation}\label{equation: index of buildings}\text{ind}(B)=k-1+\mu_{\CZ}(\alpha)-\sum_{i=1}^k\mu_{\CZ}(\beta_i).\end{equation}
This fact follows from an inductive argument applied to the height of $B$.  \end{remark}

The following proposition considers the relationships between the Conley-Zehnder indices and actions of a pair of Reeb orbits representing the same free homotopy class in  $[S^1,S^3/G]$.

\begin{proposition} \label{proposition: CZ and action}
Suppose $[\gamma_+]=[\gamma_-]\in[S^1,S^3/G]$ for $\gamma_+\in\mathcal{P}^{L_N}(\lambda_N)$ and $\gamma_-\in\mathcal{P}^{L_M}(\lambda_M)$.
\begin{enumerate}[(a)]
\itemsep-.35em
    \item If $\mu_{\CZ}(\gamma_+)=\mu_{\CZ}(\gamma_-)$, then $\gamma_+\sim\gamma_-$.
    \item If $\mu_{\CZ}(\gamma_+)<\mu_{\CZ}(\gamma_-)$, then $\mathcal{A}(\gamma_+)<\mathcal{A}(\gamma_-)$.
\end{enumerate}
\end{proposition}

The proof of Proposition \ref{proposition: CZ and action} is postponed to Section \ref{subsection: free homotopy classes represented by reeb orbits}, where we will show that it holds for cyclic, dihedral, and polyhedral groups $G$ (Lemmas \ref{lemma: cyclic cobordisms}, \ref{lemma: dihedral cobordisms}, and \ref{lemma: polyhedral cobordisms}). For any exact symplectic cobordism $(X,\lambda,J)$, Proposition \ref{proposition: CZ and action} (a) implies that $\mathcal{M}_0^J(\gamma_+,\gamma_-)$ is empty whenever $\gamma_+\nsim\gamma_-$, and (b) crucially implies that there do not exist cylinders of negative Fredholm index in $X$. Using Proposition \ref{proposition: CZ and action}, we now prove a compactness argument:

\begin{proposition}\label{proposition: buildings in cobordisms}
Fix $N< M$, $\gamma_+\in\mathcal{P}^{L_N}(\lambda_N)$, and $\gamma_-\in\mathcal{P}^{L_M}(\lambda_M)$. For $n_{\pm}\in\Z_{\geq0}$, consider a connected, genus zero building $B=(u_{n_+},\dots,  u_0, \dots, u_{-n_-})$, where $u_i$ is in the symplectization of $\lambda_N$ for $i>0$, $u_i$ is in the symplectization of $\lambda_M$ for $i<0$, and $u_0$ is in a generic, completed,  exact symplectic cobordism $(X,\lambda, J)$ from $(\lambda_N,J_N)$ to $(\lambda_M,J_M)$. If  $\mbox{\em ind}(B)=0$, with single positive puncture at $\gamma_+$ and single negative puncture at $\gamma_-$, then $n_+=n_-=0$ and $u_0\in\mathcal{M}_0^J(\gamma_+,\gamma_-)$.
\end{proposition}
\begin{proof}
This building $B$ provides the following sub-buildings, some of which may be empty, depicted in Figure \ref{figure: building}:

\begin{itemize}
    \item $B_+$, a  building in the symplectization of $\lambda_N$, with no level consisting entirely of trivial cylinders, with one positive puncture at $\gamma_+$, $k+1$ negative punctures at  $\alpha_i\in\mathcal{P}^{L_N}(\lambda_N)$, for $i=0,\dots, k$, with $[\alpha_i]=0$ for $i>0$, and $[\gamma_+]=[\alpha_0]$.
    \item $B_0$, a height 1 building in the cobordism with one positive puncture at $\alpha_0$, $l+1$ negative punctures at $\beta_i\in\mathcal{P}^{L_N}(\lambda_M)$, for $i=0,\dots,l$, with $[\beta_i]=0$ for $i>0$, and $[\alpha_0]=[\beta_0]$.
    \item $B_i$, a  building with one positive end at $\alpha_i$, without negative ends, for $i=1,\dots, k$.
    \item $C_i$, with one positive end at $\beta_i$, without negative ends, for $i=1,2,\dots,l$.
    \item $B_-$, a  building in the symplectization of $\lambda_M$, with one positive puncture at $\beta_0$ and one negative puncture at $\gamma_-$.
\end{itemize}

\begin{figure}[h]
    \centering
    \includegraphics[width=0.5\textwidth]{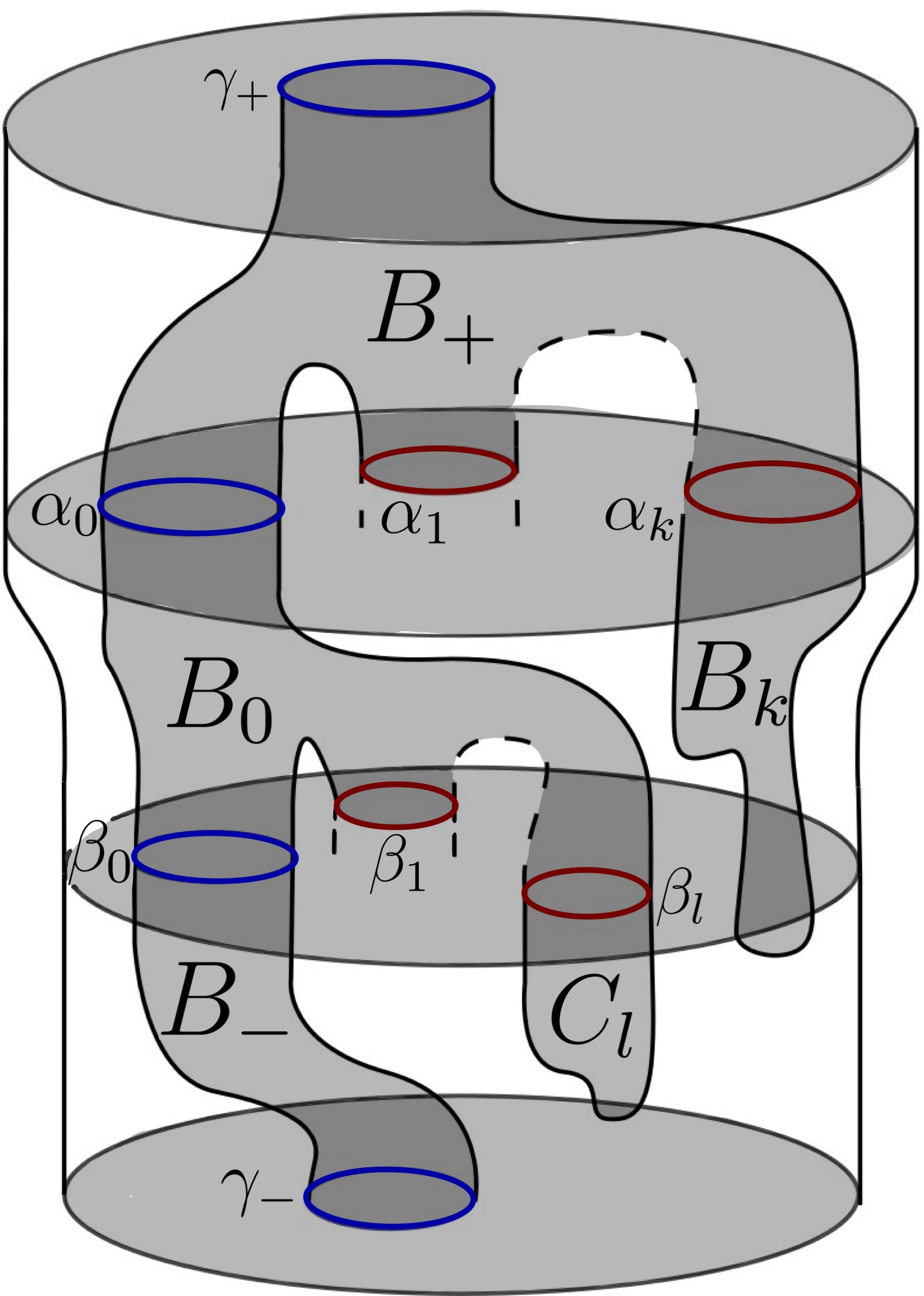}
    \caption{A limit of a sequence of cylinders in a cobordism}
    \label{figure: building}
\end{figure}

Contractible Reeb orbits are shown in red, while Reeb orbits representing the free homotopy class $[\gamma_{\pm}]$ are shown in blue. Each sub-building is connected and has genus zero. Although no level of $B$ consists entirely of trivial cylinders, some of these sub-buildings may have entirely trivial levels in a symplectization. Note that $0=\text{ind}(B)$ equals the sum of indices of the above buildings. Write $0=\text{ind}(B)=U+V+W$, where
\[U:=\text{ind}(B_+)+\sum_{i=1}^k\text{ind}(B_i),\,\,\,\,V:=\text{ind}(B_0)+\sum_{i=1}^l\text{ind}(C_i),\,\,\text{and}\,\, W:=\text{ind}(B_-).\]
We will first argue that $U$, $V$, and $W$ are non-negative integers. 

To see $U\geq0$, apply the index formula \eqref{equation: index of buildings} to each summand to compute $U=\mu_{\CZ}(\gamma_+)-\mu_{\CZ}(\alpha_0)$. If $U<0$, then  Proposition \ref{proposition: CZ and action} (b) implies $\mathcal{A}(\gamma_+)<\mathcal{A}(\alpha_0)$, which would violate the fact that action decreases along holomorphic buildings. We must have that $U\geq0$. 

To see $V\geq0$, again apply the index formula \eqref{equation: index of buildings} to find $V=\mu_{\CZ}(\alpha_0)-\mu_{\CZ}(\beta_0)$. Suppose $V<0$. Now Proposition \ref{proposition: CZ and action} (b) implies $\mathcal{A}(\alpha_0)<\mathcal{A}(\beta_0)$, contradicting the decrease of action. 

To see $W\geq 0$, consider that either $B_-$ consists entirely of trivial cylinders, or it doesn't. In the former case $W=0$. In the latter case, \cite[Proposition 2.8]{HN} implies that $W>0$. 

Because 0 is written as the sum of three non-negative integers, we conclude that $U=V=W=0$. We will combine this fact with Proposition \ref{proposition: CZ and action} to conclude that $B_{\pm}$ are empty buildings and that $B_0$ is a cylinder, finishing the proof. 

Note that $U=0$ implies $\mu_{\CZ}(\gamma_+)=\mu_{\CZ}(\alpha_0)$. Because $[\gamma_+]=[\alpha_0]$, Proposition \ref{proposition: CZ and action} (a) implies that $\gamma_+\sim\alpha_0$. This is enough to conclude $\gamma_+=\alpha_0\in\mathcal{P}(\lambda_N)$, and importantly, $\mathcal{A}(\gamma_+)=\mathcal{A}(\alpha_0)$. Noting that $\mathcal{A}(\gamma_+)\geq\sum_{i=0}^k\mathcal{A}(\alpha_i)$ (again, by decrease of action), we must have that $k=0$ and that this inequality is an equality. Thus, the buildings $B_i$ are empty for $i\neq0$, and the building $B_+$ has index 0 with only one negative end, $\alpha_0$. If $B_+$ has some nontrivial levels then \cite[Proposition 2.8]{HN} implies $0=\text{ind}(B_+)>0$. Thus, $B_+$ consists only of trivial levels. Because $B_+$ has no trivial levels, it is empty, and $n_+=0$.

Similarly, $V=0$ implies $\mu_{\CZ}(\alpha_0)=\mu_{\CZ}(\beta_0)$. Again, because $[\alpha_0]=[\beta_0]$, Proposition \ref{proposition: CZ and action} (a) implies that $\alpha_0\sim\beta_0$. Although we cannot write $\alpha_0=\beta_0$, we \emph{can} conclude that the difference $\mathcal{A}(\alpha_0)-\mathcal{A}(\beta_0)$ may be made arbitrarily small, by rescaling  $\{\varepsilon_K\}_{K=1}^{\infty}$ by some $c\in(0,1)$. Thus, the inequality $\mathcal{A}(\alpha_0)\geq\sum_{i=0}^l\mathcal{A}(\beta_i)$ forces $l=0$, implying that each $C_i$ is empty, and that $B_0$ has a single negative puncture at $\beta_0$, i.e. $B_0=(u_0)$ for $u_0\in\mathcal{M}_0^J(\alpha_0,\beta_0)$.

Finally, we consider our index 0 building $B_-$ in the symplectization of $\lambda_M$, with one positive end at $\beta_0$ and one negative end at $\gamma_-$. Again, \cite[Proposition 2.8]{HN} tells us that if $B_-$ has nontrivial levels, then $\text{ind}(B_-)>0$. Thus, all levels of $B_-$ must be trivial (this implies $\beta_0=\gamma_-$). However, because the $B_i$ and $C_j$ are empty, a trivial level of $B_-$ is a trivial level of $B$ itself, contradicting our hypothesis on $B$. Thus, $B_-$ is empty and $n_-=0$.
\end{proof}

\begin{corollary}\label{corollary: moduli spaces are finitie}
The moduli space $\mathcal{M}_0^J(\gamma_+,\gamma_-)$ is a compact, 0-dimensional manifold for generic $J$, where $\gamma_+\in\mathcal{P}^{L_N}(\lambda_N)$ and $\gamma_-\in\mathcal{P}^{L_M}(\lambda_M)$ are good, and $N<M$. Furthermore, $\mathcal{M}_0^J(\gamma_+,\gamma_-)$ is empty if $\gamma_+\nsim\gamma_-$.
\end{corollary}

\begin{proof}
To prove regularity of the cylinders $u\in\mathcal{M}_0^J(\gamma_+,\gamma_-)$, we invoke automatic transversality (\cite[Theorem 1]{W1}), providing regularity of all such $u$, because the inequality \[0=\text{ind}(u)>2b-2+2g(u)+h_+(u)=-2\]
holds.  Here, $b=0$ is the number of branched points of $u$ over its underlying somewhere injective cylinder, $g(u)=0$ is the genus of the curve, and $h_+(u)$ is the number of ends of $u$ asymptotic to positive hyperbolic Reeb orbits. Because good Reeb orbits in $\mathcal{P}^{L_N}(\lambda_N)$ and $\mathcal{P}^{L_M}(\lambda_M)$ are either elliptic or negative hyperbolic, we have that $h_+(u)=0$. To prove compactness of the moduli spaces, note that a sequence in $\mathcal{M}_0^J(\gamma_+,\gamma_-)$ has a subsequence converging in the sense of \cite{BEHWZ} to a building with the properties detailed in Proposition \ref{proposition: buildings in cobordisms}. Proposition \ref{proposition: buildings in cobordisms} proves that such an object is single cylinder in $\mathcal{M}_0^J(\gamma_+,\gamma_-)$, proving compactness of $\mathcal{M}_0^J(\gamma_+,\gamma_-)$. Finally, by Proposition \ref{proposition: CZ and action} (a), the existence of $u\in\mathcal{M}_0^J(\gamma_+,\gamma_-)$ implies that $\gamma_+\sim\gamma_-$. Thus, $\gamma_+\nsim\gamma_-$ implies that $\mathcal{M}_0^J(\gamma_+,\gamma_-)$ is empty.
\end{proof}

\subsection{Homotopy classes of Reeb orbits and proof of Proposition \ref{proposition: CZ and action}}\label{subsection: free homotopy classes represented by reeb orbits}

Proposition \ref{proposition: CZ and action} was key in arguing  compactness of $\mathcal{M}_0^J(\gamma_+,\gamma_-)$. To prove Proposition \ref{proposition: CZ and action}, we will make use of a bijection $[S^1,S^3/G]\cong\text{Conj}(G)$ to identify the free  homotopy classes represented by orbits in $S^3/G$ in terms of $G$. A loop in $S^3/G$ is a map $\gamma:[0,T]\to S^3/G$ satisfying  $\gamma(0)=\gamma(T)$. Selecting a lift $\widetilde{\gamma}:[0,T]\to S^3$ of $\gamma$ to $S^3$ determines a unique $g\in G$, for which $g\cdot\widetilde{\gamma}(0)=\widetilde{\gamma}(T)$. The map $[\gamma]\in[S^1,S^3/G]\mapsto[g]\in\text{Conj}(G)$ is well defined,  bijective, and respects iterations; that is, if $[\gamma]\cong[g]$, then  $[\gamma^m]\cong[g^m]$ for $m\in\N$.

\subsubsection{Cyclic subgroups}
We may assume that $G=\langle g\rangle\cong\Z_n$, where $g$ is the diagonal matrix $g=\text{Diag}(\epsilon, \overline{\epsilon})$, and $\epsilon:=\text{exp}{(2\pi i/n)}\in\C$. We have  $\text{Conj}(G)=\{[g^m]:0\leq m<n\}$, each class is a singleton because $G$ is abelian. We select explicit lifts $\widetilde{\gamma}_{\mathfrak{s}}$ and $\widetilde{\gamma}_{\mathfrak{n}}$ of $\gamma_{\mathfrak{s}}$ and $\gamma_{\mathfrak{n}}$ to $S^3$ given by:
\[\widetilde{\gamma}_{\mathfrak{n}}(t)=(e^{it},0),\,\,\, \text{and}\,\,\,\widetilde{\gamma}_{\mathfrak{s}}(t)=(0,e^{it}),\,\, \text{for}\,\, t\in[0,2\pi/n].\]
Because $g\cdot\widetilde{\gamma}_{\mathfrak{n}}(0)=\widetilde{\gamma}_{\mathfrak{n}}(2\pi/n)$,  $[\gamma_{\mathfrak{n}}]\cong[g]$,  and thus $[\gamma_{\mathfrak{n}}^m]\cong[g^m]$ for $m\in\N$. Similarly, because $g^{n-1}\cdot\widetilde{\gamma}_{\mathfrak{s}}(0)=\widetilde{\gamma}_{\mathfrak{s}}(2\pi/n)$, $[\gamma_{\mathfrak{s}}]\cong[g^{n-1}]=[g^{-1}]$, and thus $[\gamma_{\mathfrak{s}}^m]\cong[g^{-m}]$ for  $m\in\N$.

\begin{center}
    \begin{tabular}{||c|c||}
    \hline
    Class & Represented orbits\\
    \hline\hline
        $[g^m]$, for $0\leq m<n$ & $\gamma_{\mathfrak{n}}^{m+kn},\,\,\gamma_{\mathfrak{s}}^{n-m+kn}$ \\
        \hline
    \end{tabular}
\end{center}

\begin{lemma}\label{lemma: cyclic cobordisms}
Suppose $[\gamma_+]=[\gamma_-]\in[S^1,S^3/\Z_n]$ for $\gamma_+\in\mathcal{P}^{L_N}(\lambda_N)$ and $\gamma_-\in\mathcal{P}^{L_M}(\lambda_M)$.
\begin{enumerate}[(a)]
\itemsep-.35em
    \item If $\mu_{\CZ}(\gamma_+)=\mu_{\CZ}(\gamma_-)$, then $\gamma_+\sim\gamma_-$.
    \item If $\mu_{\CZ}(\gamma_+)<\mu_{\CZ}(\gamma_-)$, then $\mathcal{A}(\gamma_+)<\mathcal{A}(\gamma_-)$.
\end{enumerate}
\end{lemma}

\begin{proof}
Write $[\gamma_{\pm}]\cong[g^m]$, for some $0\leq m<n$. To prove (a), assume $\mu_{\CZ}(\gamma_+)=\mu_{\CZ}(\gamma_-)$. Recall the Conley-Zehnder index formulas \eqref{equation: CZ indices cyclic} from Section \ref{subsection: cyclic}: \begin{equation}\label{equation: CZ indices cyclic repeated}
    \mu_{\CZ}(\gamma_{\mathfrak{s}}^k)=2\Bigl\lceil \frac{2k}{n} \Bigr\rceil-1,\,\,\,\,\,\mu_{\CZ}(\gamma_{\mathfrak{n}}^k)=2\Bigl\lfloor \frac{2k}{n} \Bigr\rfloor+1.
\end{equation} We first argue that $\gamma_{\pm}$ both project to the same orbifold point. To see why,  note that an iterate of $\gamma_{\mathfrak{n}}$ representing the same homotopy class as an iterate of  $\gamma_{\mathfrak{s}}$ cannot share the same Conley-Zehnder indices, as  the equality $\mu_{\CZ}(\gamma_{\mathfrak{n}}^{m+k_1n})=\mu_{\CZ}(\gamma_{\mathfrak{s}}^{n-m+k_2n})$ implies, by \eqref{equation: CZ indices cyclic repeated}, that $2 \lfloor \frac{2m}{n}\rfloor+1=2+2k_2-2k_1$. This equality cannot hold because the left hand side is odd while the right hand side is even. So, without loss of  generality, suppose both $\gamma_{\pm}$ are iterates of $\gamma_{\mathfrak{n}}$. Then we must have that $\gamma_{\pm}=\gamma_{\mathfrak{n}}^{m+k_{\pm}n}$. Again by \eqref{equation: CZ indices cyclic repeated}, the equality $\mu_{\CZ}(\gamma_+)=\mu_{\CZ}(\gamma_-)$ implies $k_+=k_-$ and thus $\gamma_+\sim\gamma_-$.

To prove (b), we suppose $\mu_{\CZ}(\gamma_+)<\mu_{\CZ}(\gamma_-)$. Recall the action formulas \eqref{equation: action cyclic} from Section \ref{subsection: cyclic}: \begin{equation}\label{equation: action cyclic repeat}
    \mathcal{A}(\gamma_{\mathfrak{s}}^k)=\frac{2\pi k(1-\varepsilon)}{n},\,\,\,\,\,\mathcal{A}(\gamma_{\mathfrak{n}}^k)=\frac{2\pi k(1+\varepsilon)}{n}.
\end{equation} If both $\gamma_{\pm}$ project to the same point, then  $m(\gamma_+)<m(\gamma_-)$, because the Conley-Zehnder index is a non-decreasing function of the multiplicity, and thus, $\mathcal{A}(\gamma_+)<\mathcal{A}(\gamma_-)$. In the case that $\gamma_{\pm}$ project to different orbifold points there are two possibilities for the pair $(\gamma_+,\gamma_-)$:

\vspace{.25cm}

\emph{Case 1}:  $(\gamma_+,\gamma_-)=(\gamma_{\mathfrak{n}}^{m+nk_+},\gamma_{\mathfrak{s}}^{n-m+nk_-})$. By \eqref{equation: action cyclic repeat}, \[\mathcal{A}(\gamma_-)-\mathcal{A}(\gamma_+)=x+\delta,\,\,\,\, \text{where}\,\,\,\, x=2\pi\bigg(1+k_--k_+-\frac{2m}{n}\bigg)\] and $\delta$ can be made arbitrarily small, independent of $m$, $n$, and $k_{\pm}$. Thus, $0<x$ would imply $\mathcal{A}(\gamma_+)<\mathcal{A}(\gamma_-)$, after reducing $\varepsilon_N$ and $\varepsilon_M$ if necessary. By  \eqref{equation: CZ indices cyclic repeated}, the inequality $\mu_{\CZ}(\gamma_+)<\mu_{\CZ}(\gamma_-)$ yields \begin{equation}\label{equation: case 1 cyclic inequality}
    k_++\Bigl\lfloor\frac{2m}{n}\Bigr\rfloor\leq k_-.
\end{equation} 
If $2m<n$, then $\lfloor\frac{2m}{n}\rfloor=0$ and \eqref{equation: case 1 cyclic inequality} implies $k_+\leq k_-$. Now we see that \[x=2\pi\bigg(1+k_--k_+-\frac{2m}{n}\bigg)\geq2\pi\bigg(1-\frac{2m}{n}\bigg)>0,\] thus $\mathcal{A}(\gamma_+)<\mathcal{A}(\gamma_-)$. If $2m\geq n$, then $\lfloor\frac{2m}{n}\rfloor=1$ and \eqref{equation: case 1 cyclic inequality} implies $k_++1\leq k_-$. Now we see that \[x=2\pi\bigg(1+k_--k_+-\frac{2m}{n}\bigg)\geq2\pi\bigg(2-\frac{2m}{n}\bigg)>0,\] hence $\mathcal{A}(\gamma_+)<\mathcal{A}(\gamma_-)$.

\vspace{.25cm}

\emph{Case 2}: $(\gamma_+,\gamma_-)=(\gamma_{\mathfrak{s}}^{n-m+nk_+},\gamma_{\mathfrak{n}}^{m+nk_-})$. By  \eqref{equation: action cyclic repeat},  \[\mathcal{A}(\gamma_-)-\mathcal{A}(\gamma_+)=x+\delta,\,\,\,\, \text{where}\,\,\,\, x=2\pi\bigg(\frac{2m}{n}+k_--k_+-1\bigg)\] and $\delta$ is a small positive number. Thus, $0\leq x$ would imply $\mathcal{A}(\gamma_+)<\mathcal{A}(\gamma_-)$. Applying \eqref{equation: CZ indices cyclic repeated}, $\mu_{\CZ}(\gamma_+)<\mu_{\CZ}(\gamma_-)$ yields \begin{equation}\label{equation: case 2 cyclic inequality}
    k_+-\Bigl\lfloor\frac{2m}{n}\Bigr\rfloor<k_-.
\end{equation} 
If $2m<n$, then $\lfloor\frac{2m}{n}\rfloor=0$ and \eqref{equation: case 2 cyclic inequality} implies $k_++1\leq k_-$. Now we see that \[x=2\pi\bigg(\frac{2m}{n}+k_--k_+-1\bigg)\geq2\pi\bigg(\frac{2m}{n}\bigg)\geq0,\] thus $\mathcal{A}(\gamma_+)<\mathcal{A}(\gamma_-)$. If $2m\geq n$, then $\lfloor\frac{2m}{n}\rfloor=1$ and \eqref{equation: case 2 cyclic inequality} implies $k_+\leq k_-$. Now we see that \[x=2\pi\bigg(\frac{2m}{n}+k_--k_+-1\bigg)\geq2\pi\bigg(\frac{2m}{n}-1\bigg)\geq0,\] thus $\mathcal{A}(\gamma_+)<\mathcal{A}(\gamma_-)$.
\end{proof}

\subsubsection{Binary dihedral groups $\D_{2n}^*$}

The nonabelian group $\D_{2n}^*$ has $4n$ elements and has $n+3$  conjugacy classes. For  $0<m<n$,  the conjugacy class of $A^m$ is $[A^m]=\{A^m, A^{2n-m}\}$ (see Section \ref{subsection: dihedral} for notation). Because $-\text{Id}$ generates the center, we have $[-\text{Id}]=\{-\text{Id}\}$. We also have the conjugacy class $[\text{Id}]=\{\text{Id}\}$. The final two conjugacy classes are \[[B]=\{B, A^2B, \dots, A^{2n-2}B\}, \,\,\,\text{and}\,\,\,[AB]=\{AB, A^3B, \dots, A^{2n-1}B\}.\]
Note that $B^{-1}=B^3=-B=A^nB$ is conjugate to $B$ if and only if $n$ is even. Table \ref{table: lifts of dihedral reeb orbits} records lifts of our three embedded Reeb orbits, $e_-, h,$ and $e_+$, to paths $\widetilde{\gamma}$ in $S^3$, along with the group element $g\in\D^*_{2n}$ satisfying $g\cdot\widetilde{\gamma}(0)=\widetilde{\gamma}(T)$.

\begin{table}[h!]
\centering
\begingroup
\setlength{\tabcolsep}{10pt}
\renewcommand{\arraystretch}{1.4}
 \begin{tabular}{||c | c | c | c ||} 
 \hline
$\widetilde{\gamma}$ & $S^3$ expression & interval & $g\in\D^*_{2n}$ \\ [0.5ex] 
 \hline\hline
 $\widetilde{e_-}(t)$ & $t\mapsto \frac{e^{it}}{\sqrt{2}}\cdot (1,-i\epsilon)$ & $t\in[0,\frac{\pi}{2}]$ & $AB$ \\ 
 \hline
 $\widetilde{h}(t)$ & $t\mapsto \frac{e^{it}}{\sqrt{2}}\cdot(e^{i\pi/4},-e^{i\pi/4})$ &  $t\in[0,\frac{\pi}{2}]$ & $B$ \\
 \hline
 $\widetilde{e_+}(t)$ & $t\mapsto (e^{it},0)$ & $t\in[0,\frac{\pi}{n}]$ & $A$ \\
 \hline
\end{tabular}
\endgroup
 \caption{Lifts of Dihedral Reeb orbits to $S^3$}
 \label{table: lifts of dihedral reeb orbits}
\end{table}

The homotopy classes of $e_-$, $h$, and $e_+$ and their iterates are now determined by this data and are recorded in Table \ref{table: Dihedral homotopy classes of Reeb orbits}. We are now ready to prove Lemma \ref{lemma: dihedral cobordisms}, the dihedral case of Proposition \ref{proposition: CZ and action}.
\begin{table}[h!]
 \begin{tabular}{||c | c | c ||} 
 \hline
Free homotopy/conjugacy class & Represented orbits ($n$ even) &Represented orbits ($n$ odd) \\ [0.5ex] 
 \hline\hline
 $[\text{Id}]$ & $e_-^{4k},\,h^{4k},\,e_+^{2nk}$ & $e_-^{4k},\,h^{4k},\,e_+^{2nk}$ \\ 
 \hline
 $[-\text{Id}]$ &  $e_-^{2+4k},\,h^{2+4k},\,e_+^{n+2nk}$ & $e_-^{2+4k},\,h^{2+4k},\,e_+^{n+2nk}$ \\
 \hline
 $[A^m]$, for $0<m<n$ & $e_+^{m+2nk},\, e_+^{2n-m+2nk}$ & $e_+^{m+2nk},\, e_+^{2n-m+2nk}$ \\
 \hline
  $[B]$& $h^{1+4k},\,h^{3+4k}$ & $h^{1+4k},\,e_-^{3+4k}$ \\
 \hline
   $[AB]$ & $e_-^{1+4k},\,e_-^{3+4k}$ & $e_-^{1+4k},\,h^{3+4k}$ \\
 \hline
\end{tabular}
 \caption{Dihedral homotopy classes of Reeb orbits}
 \label{table: Dihedral homotopy classes of Reeb orbits}
\end{table}

\begin{lemma}\label{lemma: dihedral cobordisms}
Suppose $[\gamma_+]=[\gamma_-]\in[S^1,S^3/\D^*_{2n}]$ for $\gamma_+\in\mathcal{P}^{L_N}(\lambda_N)$ and $\gamma_-\in\mathcal{P}^{L_M}(\lambda_M)$.
\begin{enumerate}[(a)]
\itemsep-.35em
    \item If $\mu_{\CZ}(\gamma_+)=\mu_{\CZ}(\gamma_-)$, then $\gamma_+\sim\gamma_-$.
    \item If $\mu_{\CZ}(\gamma_+)<\mu_{\CZ}(\gamma_-)$, then $\mathcal{A}(\gamma_+)<\mathcal{A}(\gamma_-)$.
\end{enumerate}
\end{lemma}
\begin{proof}
We first prove (a). Recall the Conley-Zehnder index formulas \eqref{equation: CZ indices dihedral} from Section \ref{subsection: dihedral}: \begin{equation} \label{equation: CZ indices dihedral repeat}
    \mu_{\CZ}(e_-^k)=2\Bigl\lceil \frac{k}{2} \Bigr\rceil-1,\,\,\,\mu_{\CZ}(h^k)=k,\,\,\,\mu_{\CZ}(e_+^k)=2\Bigl\lfloor \frac{k}{n} \Bigr\rfloor+1.
\end{equation} By Table \ref{table: Dihedral homotopy classes of Reeb orbits}, there are five possible values of the class $[\gamma_{\pm}]$:

        \vspace{.25cm}

        \emph{Case 1:} $[\gamma_{\pm}]\cong[\text{Id}]$. Then $\{\gamma_{\pm}\}\subset\{e_-^{4k_1}, h^{4k_2}, e_+^{2nk_3}\,|\,k_i\in\N\}$. The Conley-Zehnder index modulo 4 of $e_-^{4k_1}$, $h^{4k_2}$, or $e_+^{2nk_3}$ is $-1$, 0, or $1$, respectively. Thus, $\gamma_{\pm}$ must project to the same orbifold point. Now, because (1) $\gamma_{\pm}$ are both type $e_-$, type $h$, or type $e_+$, and (2) share a homotopy class, $\mu_{\CZ}(\gamma_+)=\mu_{\CZ}(\gamma_-)$ and \eqref{equation: CZ indices dihedral repeat} imply $m(\gamma_+)=m(\gamma_-)$, thus $\gamma_+\sim\gamma_-$.
        
        \vspace{.25cm}
        
        \emph{Case 2:} $[\gamma_{\pm}]\cong[-\text{Id}]$. Then $\{\gamma_{\pm}\}\subset\{e_-^{2+4k_1}, h^{2+4k_2}, e_+^{n+2nk_3}\,|\,k_i\in\Z_{\geq0}\}$. The Conley-Zehnder index modulo 4 of $e_-^{4k_1}$, $h^{4k_2}$, or $e_+^{2nk_3}$ is $1$, 2, or $3$, respectively. By the reasoning as in the previous case, we obtain $\gamma_+\sim\gamma_-$.
        
        \vspace{.25cm}
        
        \emph{Case 3:} $[\gamma_{\pm}]\cong[A^m]$ for some $0<m<n$. If $\gamma_+\nsim\gamma_-$, then by Table \ref{table: Dihedral homotopy classes of Reeb orbits} and \eqref{equation: CZ indices dihedral repeat}, we must have  $k_1, k_2\in\Z$ such that $\mu_{\CZ}(e_+^{m+2nk_1})=\mu_{\CZ}(e_+^{2n-m+2nk_2})$. This equation becomes $\lfloor\frac{m}{n}\rfloor+\lceil\frac{m}{n}\rceil=2(1+k_2-k_1)$. By the bounds on $m$, the ratio $\frac{m}{n}$ is not an integer, and so the quantity on the left hand side is odd, which is impossible. Thus, we must have $\gamma_+\sim\gamma_-$.
        
        \vspace{.25cm}
        
        \emph{Case 4:} $[\gamma_{\pm}]\cong[B]$. If $n$ is even, then both $\gamma_{\pm}$ are iterates of $h$, and by \eqref{equation: CZ indices dihedral repeat}, these multiplicities agree, so $\gamma_+\sim\gamma_-$. If $n$ is odd and $\gamma_{\pm}$ project to the same orbifold point, then because their homotopy classes and Conley-Zehnder indices agree, Table \ref{table: Dihedral homotopy classes of Reeb orbits} and \eqref{equation: CZ indices dihedral repeat} imply their multiplicities must agree, so $\gamma_+\sim\gamma_-$.  If they project to different orbifold points then  $1+4k_+=\mu_{\CZ}(h^{1+4k_+})=\mu_{\CZ}(e_-^{3+4k_-})=3+4k_-$, which would imply $1=3$ mod 4, impossible.
        
        \vspace{.25cm}
        
        \emph{Case 5:} $[\gamma_{\pm}]\cong[AB]$. If $n$ is even, then both $\gamma_{\pm}$ are iterates of $e_-$, and by \eqref{equation: CZ indices dihedral repeat}, their multiplicities agree, thus $\gamma_+\sim\gamma_-$. If $n$ is odd and $\gamma_{\pm}$ project to the same point, then $\gamma_+\sim\gamma_-$ for the same reasons as in Case 4. Thus $\gamma_+\nsim\gamma_-$ implies  $3+4k_+=\mu_{\CZ}(h^{3+4k_+})=\mu_{\CZ}(e_-^{1+4k_-})=1+4k_-$, impossible mod 4.

        \vspace{.25cm}

To prove (b), let $\mu_{\CZ}(\gamma_+)<\mu_{\CZ}(\gamma_-)$. Recall the action formulas \eqref{equation: action dihedral} from Section \ref{subsection: dihedral}: \begin{equation}\label{equation: action dihedral repeat}
        \mathcal{A}(e_-^k)=\frac{k\pi(1-\varepsilon)}{2},\,\,\,\mathcal{A}(h^k)=\frac{k\pi}{2},\,\,\,\mathcal{A}(e_+^k)=\frac{k\pi(1+\varepsilon)}{n}.
\end{equation} First, note that if $\gamma_{\pm}$ project to the same orbifold critical point,  then we have $m(\gamma_+)<m(\gamma_-)$, because the Conley-Zehnder index as a function of multiplicity of the iterate is non-decreasing. This implies  $\mathcal{A}(\gamma_+)<\mathcal{A}(\gamma_-)$, because the action strictly increases as a function of the iterate. We now prove (b) for pairs of orbits $\gamma_{\pm}$ projecting to different orbifold critical points.

        \vspace{.25cm}
        
        \emph{Case 1:} $[\gamma_{\pm}]\cong[\text{Id}]$. Then $\{\gamma_{\pm}\}\subset\{e_-^{4k_1}, h^{4k_2}, e_+^{2nk_3}\,|\,k_i\in\N\}$. There are six  combinations of the value of $(\gamma_+,\gamma_-)$, whose index and action are compared using \eqref{equation: CZ indices dihedral repeat} and \eqref{equation: action dihedral repeat}:
        \begin{enumerate}[(i)]
        \itemsep-.35em
            \item $\gamma_+=e_-^{4k_+}$ and $\gamma_-=h^{4k_-}$. The inequality $4k_+-1=\mu_{\CZ}(\gamma_+)<\mu_{\CZ}(\gamma_-)=4k_-$ implies $k_+\leq k_-$. This verifies that $\mathcal{A}(\gamma_+)=2\pi k_+(1-\varepsilon_N)<2\pi k_-=\mathcal{A}(\gamma_-)$.
            \item $\gamma_+=h^{4k_+}$ and $\gamma_-=e_-^{4k_-}$. The inequality $4k_+=\mu_{\CZ}(\gamma_+)<\mu_{\CZ}(\gamma_-)=4k_--1$ implies $k_+< k_-$. This verifies that $\mathcal{A}(\gamma_+)=2\pi k_+<2\pi k_-(1-\varepsilon_M)=\mathcal{A}(\gamma_-)$.
            \item $\gamma_+=e_-^{4k_+}$ and $\gamma_-=e_+^{2nk_-}$. The inequality $4k_+-1=\mu_{\CZ}(\gamma_+)<\mu_{\CZ}(\gamma_-)=4k_-+1$ implies $k_+\leq k_-$. This verifies that $\mathcal{A}(\gamma_+)=2\pi k_+(1-\varepsilon_N)<2\pi k_-(1+\varepsilon_M)=\mathcal{A}(\gamma_-)$.
            \item $\gamma_+=e_+^{2nk_+}$ and $\gamma_-=e_-^{4k_-}$. The inequality $4k_++1=\mu_{\CZ}(\gamma_+)<\mu_{\CZ}(\gamma_-)=4k_--1$ implies $k_+< k_-$. This verifies that $\mathcal{A}(\gamma_+)=2\pi k_+(1+\varepsilon_N)<2\pi k_-(1-\varepsilon_M)=\mathcal{A}(\gamma_-)$.
            \item $\gamma_+=h^{4k_+}$ and $\gamma_-=e_+^{2nk_-}$. The inequality $4k_+=\mu_{\CZ}(\gamma_+)<\mu_{\CZ}(\gamma_-)=4k_-+1$ implies $k_+\leq k_-$. This verifies that $\mathcal{A}(\gamma_+)=2\pi k_+<2\pi k_-(1+\varepsilon_M)=\mathcal{A}(\gamma_-)$.
            \item $\gamma_+=e_+^{2nk_+}$ and $\gamma_-=h^{4k_-}$. The inequality $4k_++1=\mu_{\CZ}(\gamma_+)<\mu_{\CZ}(\gamma_-)=4k_-$ implies $k_+< k_-$. This verifies that $\mathcal{A}(\gamma_+)=2\pi k_+(1+\varepsilon_N)<2\pi k_-=\mathcal{A}(\gamma_-)$.
        \end{enumerate}
                
        \vspace{.25cm}
        
        \emph{Case 2:} $[\gamma_{\pm}]\cong[-\text{Id}]$. Then $\{\gamma_{\pm}\}\subset\{e_-^{2+4k_1}, h^{2+4k_2}, e_+^{n+2nk_3}\,|\,k_i\in\Z_{\geq0}\}$.  The possible values of $(\gamma_+,\gamma_-)$ and the  arguments for $\mathcal{A}(\gamma_+)<\mathcal{A}(\gamma_-)$ are identical to the above case.
                        
        \vspace{.25cm}
        
        \emph{Case 3:} $[\gamma_{\pm}]\cong[A^m]$ for $0<m<n$ and both $\gamma_{\pm}$ are iterates of $e_+$, so (b) holds.

        \vspace{.25cm}
        
        \emph{Case 4:} $[\gamma_{\pm}]\cong[B]$. If $n$ is even, then both $\gamma_{\pm}$ are iterates of $h$, and so (b) holds. Otherwise, $n$ is odd and the pair $(\gamma_+,\gamma_-)$ is either $(e_-^{3+4k_+},h^{1+4k_-})$ or $(h^{1+4k_+}, e_-^{3+4k_-})$. In the former case $\mu_{\CZ}(\gamma_+)<\mu_{\CZ}(\gamma_-)$ and \eqref{equation: CZ indices dihedral repeat} imply $k_+< k_-$, so, by \eqref{equation: action dihedral repeat},\\ $\mathcal{A}(\gamma_+)=\frac{(3+4k_+)\pi}{2}(1-\varepsilon_N)<\frac{(1+4k_-)\pi}{2}=\mathcal{A}(\gamma_-)$. In the latter case, $\mu_{\CZ}(\gamma_+)<\mu_{\CZ}(\gamma_-)$ and \eqref{equation: CZ indices dihedral repeat} imply $k_+\leq k_-$, thus $\mathcal{A}(\gamma_+)=\frac{(1+4k_+)\pi}{2}<\frac{(3+4k_-)\pi}{2}(1-\varepsilon_M)=\mathcal{A}(\gamma_-)$ by \eqref{equation: action dihedral repeat}. 

        \vspace{.25cm}
        
        \emph{Case 5:} $[\gamma_{\pm}]\cong[AB]$. If $n$ is even, then both $\gamma_{\pm}$ are iterates of $e_-$, and so $\mathcal{A}(\gamma_+)<\mathcal{A}(\gamma_-)$ holds.  Otherwise, $n$ is odd and the pair $(\gamma_+,\gamma_-)$ is either $(e_-^{1+4k_+},h^{3+4k_-})$ or $(h^{3+4k_+}, e_-^{1+4k_-})$. If the former holds, then $\mu_{\CZ}(\gamma_+)<\mu_{\CZ}(\gamma_-)$ and \eqref{equation: CZ indices dihedral repeat} imply $k_+\leq k_-$, and so, by \eqref{equation: action dihedral repeat},\\ $\mathcal{A}(\gamma_+)=\frac{(1+4k_+)\pi}{2}(1-\varepsilon_N)<\frac{(3+4k_-)\pi}{2}=\mathcal{A}(\gamma_-)$. If the latter holds, then $\mu_{\CZ}(\gamma_+)<\mu_{\CZ}(\gamma_-)$ and \eqref{equation: CZ indices dihedral repeat} imply $k_+< k_-$, and so $\mathcal{A}(\gamma_+)=\frac{(3+4k_+)\pi}{2}<\frac{(1+4k_-)\pi}{2}(1-\varepsilon_M)=\mathcal{A}(\gamma_-)$ by \eqref{equation: action dihedral repeat}. 
\end{proof}

\subsubsection{Binary polyhedral groups $\T^*$, $\Oc^*$, and $\I^*$}\label{subsubsection: binary polyhedral}
We will describe the homotopy classes of the Reeb orbits in $S^3/\P^*$ using a more geometric point of view than in the dihedral case. If a loop $\gamma$ in $S^3/\P^*$ and  $c\in\text{Conj}(G)$ satisfy $[\gamma]\cong c$, then the order of $\gamma$ in $\pi_1(S^3/\P^*)$ equals the \emph{group order} of $c$, defined to be the order of any $g\in\P^*$ representing $c$. If $\gamma$ has order $k$ in the fundamental group and $c\in\text{Conj}(\P^*)$ is the \emph{only} class with group order $k$, then we can immediately conclude that $[\gamma]\cong c$. Determining the free homotopy class of $\gamma$ via the group order is more difficult when there are multiple conjugacy classes of $\P^*$ with the same group order. 

Tables \ref{table: tetrahedral conjugacy}, \ref{table: octahedral conjugacy}, and \ref{table: icosahedral conjugacy} provide notation for the distinct conjugacy classes of $\T^*$, $\Oc^*$, and $\I^*$, along with their group orders. Our notation indicates when there exist multiple conjugacy classes featuring the same group order - the notation $P_{i,A}$ and $P_{i,B}$ provides labels for the two distinct conjugacy classes of $\P^*$ (for $P\in\{T,O,I\}$) of group order $i$. For $P\in\{T,O,I\}$, $P_{\text{Id}}$ and $P_{-\text{Id}}$ denote the singleton conjugacy classes $\{\text{Id}\}$ and $\{-\text{Id}\}$, respectively, and $P_i$ denotes the unique conjugacy class of group order $i$.

\begin{table}[h!]
\centering
    \begin{tabular}{||c ||c | c | c |c | c | c | c ||}
    \hline
       Conjugacy class  & $T_{\text{Id}}$ & $T_{-\text{Id}}$ & $T_4$ & $T_{6,A}$ & $T_{6,B}$ & $T_{3,A}$ & $T_{3,B}$ \\
    \hline
    Group order  & 1 & 2 & 4 & 6 & 6 & 3 & 3 \\
    \hline
    \end{tabular}
    \caption{The 7 conjugacy classes of $\T^*$}
    \label{table: tetrahedral conjugacy}
\end{table}
\begin{table}[h!]
\centering
    \begin{tabular}{||c ||c | c | c | c |c | c | c | c ||}
    \hline
       Conjugacy class  & $O_{\text{Id}}$ & $O_{-\text{Id}}$ & $O_{8,A}$ & $O_{8,B}$ & $O_{4,A}$ & $O_{4,B}$ & $O_{6}$ & $O_3$ \\
    \hline
    Group order  & 1 & 2 & 8 & 8 & 4 & 4 & 6 & 3 \\
    \hline
    \end{tabular}
        \caption{The 8 conjugacy classes of $\Oc^*$}
    \label{table: octahedral conjugacy}
\end{table}
\begin{table}[h!]
\centering
    \begin{tabular}{||c ||c | c | c | c |c | c | c | c | c ||}
    \hline
       Conjugacy class  & $I_{\text{Id}}$ & $I_{-\text{Id}}$ & $I_{10,A}$ & $I_{10,B}$ & $I_{5,A}$ & $I_{5,B}$ & $I_{4}$ & $I_6$  & $I_3$\\
    \hline
    Group order  & 1 & 2 & 10 & 10 & 5 & 5 & 4 & 6 & 3 \\
    \hline
    \end{tabular}
        \caption{The 9 conjugacy classes of $\I^*$}
    \label{table: icosahedral conjugacy}
\end{table}

The conclusions of Remarks \ref{remark: identify}, \ref{remark: distinguish 1}, and \ref{remark: distinguish 2} allow us to record the homotopy classes represented by any iterate of $\mathcal{V}$, $\mathcal{E}$, or $\mathcal{F}$ in Tables \ref{table: Tetrahedral homotopy classes of Reeb orbits}, \ref{table: octahedral homotopy classes of Reeb orbits}, and \ref{table: icosahedral homotopy classes of Reeb orbits}. We explain how the table is set up in the tetrahedral case: it must be true that $[\mathcal{V}]\neq[\mathcal{F}]$, otherwise taking the 2-fold iterate would violate Remark \ref{remark: distinguish 1} (1) so without loss of generality write $[\mathcal{V}]\cong T_{6,A}$ and $[\mathcal{F}]\cong T_{6,B}$, and similarly $[\mathcal{V}^2]\cong T_{3,A}$ and $[\mathcal{F}^2]\cong T_{3,B}$. By Remark \ref{remark: identify} (1), we must have that $[\mathcal{F}^5]\cong T_{6,A}$,  $[\mathcal{F}^4]\cong T_{3,A}$. Taking the 4-fold iterate of $[\mathcal{V}]=[\mathcal{F}^5]$ provides $[\mathcal{V}^4]=[\mathcal{F}^2]\cong T_{3,B}$, and the 5-fold iterate provides $[\mathcal{V}^5]=[\mathcal{F}^1]\cong T_{6,B}$, and we have resolved all ambiguity regarding the tetrahedral classes of group orders 6 and 3. Analogous arguments apply in the octahedral and icosahedral cases.

\begin{table}
\centering
\makebox[0pt][c]{\parbox{1.2\textwidth}{%
    \begin{minipage}[b]{0.32\hsize}\centering
        \begin{tabular}{||c | c ||} 
 \hline
Class & Represented orbits  \\ [0.5ex] 
 \hline\hline
 $T_{\text{Id}}$ & $\mathcal{V}^{6k},\,\mathcal{E}^{4k},\,\mathcal{F}^{6k}$ \\ 
 \hline
 $T_{-\text{Id}}$ &  $\mathcal{V}^{3+6k},\,\mathcal{E}^{2+4k},\,\mathcal{F}^{3+6k}$\\
 \hline
 $T_{4}$ &  $\mathcal{E}^{1+4k},\,\mathcal{E}^{3+4k}$\\
 \hline
  $T_{6,A}$& $\mathcal{V}^{1+6k},\,\mathcal{F}^{5+6k}$\\
 \hline
   $T_{6,B}$& $\mathcal{F}^{1+6k},\,\mathcal{V}^{5+6k}$\\
 \hline
   $T_{3,A}$& $\mathcal{V}^{2+6k},\,\mathcal{F}^{4+6k}$\\
 \hline
    $T_{3,B}$& $\mathcal{F}^{2+6k},\,\mathcal{V}^{4+6k}$\\
 \hline
\end{tabular}
 \caption{$S^3/\T^*$ Reeb Orbits}
 \label{table: Tetrahedral homotopy classes of Reeb orbits}
    \end{minipage}
    \hfill
    \begin{minipage}[b]{0.32\hsize}\centering
         \begin{tabular}{||c | c ||} 
 \hline
Class & Represented orbits  \\ [0.5ex] 
 \hline\hline
 $O_{\text{Id}}$ & $\mathcal{V}^{8k},\,\mathcal{E}^{4k},\,\mathcal{F}^{6k}$ \\ 
 \hline
 $O_{-\text{Id}}$ &  $\mathcal{V}^{4+8k},\,\mathcal{E}^{2+4k},\,\mathcal{F}^{3+6k}$\\
 \hline
 $O_{8,A}$ &  $\mathcal{V}^{1+8k},\,\mathcal{V}^{7+8k}$\\
 \hline
  $O_{8,B}$& $\mathcal{V}^{3+8k},\,\mathcal{V}^{5+8k}$\\
 \hline
  $O_{4,A}$ &  $\mathcal{V}^{2+8k},\,\mathcal{V}^{6+8k}$\\
 \hline
  $O_{4,B}$& $\mathcal{E}^{1+4k},\,\mathcal{E}^{3+4k}$\\
 \hline
   $O_{6}$& $\mathcal{F}^{1+6k},\,\mathcal{F}^{5+6k}$\\
 \hline
   $O_{3}$& $\mathcal{F}^{2+6k},\,\mathcal{F}^{4+6k}$\\
 \hline
\end{tabular}
\caption{$S^3/\Oc^*$ Reeb Orbits}
\label{table: octahedral homotopy classes of Reeb orbits}
    \end{minipage}
    \hfill
    \begin{minipage}[b]{0.32\hsize}\centering
        \begin{tabular}{||c | c ||} 
 \hline
Class & Represented orbits  \\ [0.5ex] 
 \hline\hline
 $I_{\text{Id}}$ & $\mathcal{V}^{10k},\,\mathcal{E}^{4k},\,\mathcal{F}^{6k}$ \\ 
 \hline
 $I_{-\text{Id}}$ &  $\mathcal{V}^{5+10k},\,\mathcal{E}^{2+4k},\,\mathcal{F}^{3+6k}$\\
 \hline
 $I_{10,A}$ &  $\mathcal{V}^{1+10k},\,\mathcal{V}^{9+10k}$\\
 \hline
  $I_{10,B}$& $\mathcal{V}^{3+10k},\,\mathcal{V}^{7+10k}$\\
 \hline
  $I_{5,A}$ &  $\mathcal{V}^{2+10k},\,\mathcal{V}^{8+10k}$\\
 \hline
  $I_{5,B}$& $\mathcal{V}^{4+10k},\,\mathcal{V}^{6+10k}$\\
 \hline
   $I_{4}$& $\mathcal{E}^{1+4k},\,\mathcal{E}^{3+4k}$\\
 \hline
   $I_{6}$& $\mathcal{F}^{1+6k},\,\mathcal{F}^{5+6k}$\\
 \hline
    $I_{3}$& $\mathcal{F}^{2+6k},\,\mathcal{F}^{4+6k}$\\
 \hline
\end{tabular}
\caption{$S^3/\I^*$ Reeb Orbits}
\label{table: icosahedral homotopy classes of Reeb orbits}
    \end{minipage}%
}}
\end{table}

\begin{remark} \label{remark: identify}
 Suppose $p_1$, $p_2\in \text{Fix}(\P)\subset S^2$ are antipodal, i.e. $p_1=-p_2$. Select $z_i\in \fP^{-1}(p_i)\subset S^3$. The Hopf fibration $\fP$ has the property that $\fP(v_1)=-\fP(v_2)$ in $S^2$ if and only if $v_1$ and $v_2$ are orthogonal vectors in $\C^2$, so $z_1$ and $z_2$ must be orthogonal. Now, $p_i$ is either of vertex, edge, or face type, so let $\gamma_i$ denote the orbit $\mathcal{V}$, $\mathcal{E}$, or $\mathcal{F}$, depending on this type of $p_i$. Because $p_1$ and $p_2$ are antipodal, the order of $\gamma_1$ equals that of $\gamma_2$ in $\pi_1(S^3/\P^*)$, call this order $d$, and let $T:=\frac{2\pi}{d}$. Now, consider that the map \[\Gamma_1:[0,T]\to S^3,\,\,\, t\mapsto e^{it}\cdot z_1\] is a lift of $\gamma_1$ to $S^3$. Thus,  $z_1$ is an eigenvector with eigenvalue $e^{iT}$ for some $g\in\P^*$, and $[\gamma_1]\cong[g]$ because $g\cdot\Gamma_1(0)=\Gamma_1(T)$. Because $g$ is special unitary, we must also have that $z_2$ is an eigenvector of $g$ with eigenvalue $e^{i(2\pi-T)}=e^{i(d-1)T}$. Now the map \[\Gamma_2^{d-1}:[0,(d-1)T]\to S^3,\,\,\,t\mapsto e^{it}\cdot z_2\] is a lift of $\gamma_2^{d-1}$ to $S^3$. This provides $g\cdot\Gamma_2^{d-1}(0)=\Gamma_2^{d-1}((d-1)T)$ which implies $[\gamma_1]=[\gamma_2^{d-1}]$, as both are identified with $[g]$. This means that 
 \begin{enumerate}
 \itemsep-.35em
     \item $[\mathcal{V}]=[\mathcal{F}^5]$, for $\P^*=\T^*$,
     \item $[\mathcal{V}]=[\mathcal{V}^7]$, for $\P^*=\Oc^*$,
     \item $[\mathcal{V}]=[\mathcal{V}^9]$, for $\P^*=\I^*$.
 \end{enumerate}
 \end{remark}

\begin{remark}\label{remark: distinguish 1}
 Suppose, for $i=1,2$, $\gamma_i$ is one of $\mathcal{V}$, $\mathcal{E}$, or $\mathcal{F}$, and suppose $\gamma_1\neq\gamma_2$. Let $d_i$ be the order of $\gamma_i$ in $\pi_1(S^3/\P^*)$, and select $k_i\in\N$ for $i=1,2$. If $\frac{2\pi k_1}{d_1}\equiv \frac{2\pi k_2}{d_2}$ modulo $2\pi\Z$ and if $\pi\nmid\frac{2\pi k_1}{d_1}$, then $[\gamma_1^{k_1}]\neq[\gamma_2^{k_2}]$. To prove this, consider that we have $g_i\in\P^*$ with eigenvector $z_i$ in $S^3$, with eigenvalue $\lambda:=e^{2\pi k_1 /d_1}=e^{2\pi k_2 i/d_2}$ so that $[\gamma_i^{k_i}]\cong[g_i]$. Note that  $\gamma_1\neq\gamma_2$ implies $\fP(z_1)$ is not in the same $\P$-orbit as $\fP(z_2)$ in $S^2$, i.e. $\pi_{\P}(\fP(z_1))\neq\pi_{\P}(\fP(z_2))$. Now, $[\gamma_1^{k_1}]=[\gamma_2^{k_2}]$ would imply $g_1=x^{-1}g_2x$, for some $x\in\P^*$, ensuring that $x\cdot z_1$ is a $\lambda$ eigenvector of $g_2$. Because $\lambda\neq\pm1$, we know that the $\lambda$-eigenspace of $g_2$ is complex 1-dimensional, so we must have that $x\cdot z_1$ and $z_2$ are co-linear. That is, $x\cdot z_1=\alpha z_2$ for some $\alpha\in S^1$, which implies that \[\pi_{\P}(\fP(z_1))=\pi_{\P}(\fP(x\cdot z_1))=\pi_{\P}(\fP(\alpha\cdot z_2))=\pi_{\P}(\fP(z_2)),\] a contradiction. This has the following applications: 
 
 \begin{enumerate}
 \itemsep-.35em
     \item $[\mathcal{V}^2]\neq[\mathcal{F}^2]$, for $\P^*=\T^*$,
     \item $[\mathcal{E}]\neq[\mathcal{V}^2]$, for $\P^*=\Oc^*$
 \end{enumerate}
\end{remark}

\begin{remark}\label{remark: distinguish 2}
 For $i=1,2$, select $k_i\in\N$ and let $\gamma_i$ denote one of $\mathcal{V}$, $\mathcal{E}$, or $\mathcal{F}$. Let $d_i$ denote the order of $\gamma_i$ in $\pi_1(S^3/\P^*)$. Suppose that $\frac{2\pi k_i}{d_i}$ is not a multiple of $2\pi$, and $\frac{2\pi k_1}{d_1}\ncong\frac{2\pi k_2}{d_2}$ mod $2\pi\Z$. If $\frac{2\pi k_1}{d_1}+\frac{2\pi k_2}{d_2}$ is not a multiple of $2\pi$, then $[\gamma_1^{k_1}]\neq[\gamma_2^{k_2}]$. To prove this, consider that we have $g_i\in\P^*$ with $[\gamma_i^{k_i}]\cong[g_i]$. This tells us that $e^{2\pi k_j i/d_j}$ is an eigenvalue of $g_j$. If it were the case that $[\gamma_1^{k_1}]=[\gamma_2^{k_2}]$ in $[S^1,S^3/\P^*]$, then we would have $[g_1]=[g_2]$ in $\text{Conj}(G)$. Because conjugate elements share eigenvalues, we would have \[\text{Spec}(g_1)=\text{Spec}(g_2)=\{e^{2\pi k_1 i/d_1},e^{2\pi k_2 i/d_2}\}.\]  However, the product of these eigenvalues is not 1, contradicting that $g_i\in\SU(2)$. This has the following two applications:
  \begin{enumerate}
 \itemsep-.35em
     \item $[\mathcal{V}^1]\neq[\mathcal{V}^3]$, for $\P^*=\Oc^*$,
     \item $[\mathcal{V}^2]\neq[\mathcal{V}^4]$, for $\P^*=\I^*$.
 \end{enumerate}
\end{remark}

We are ready to prove Lemma \ref{lemma: polyhedral cobordisms}, which is the polyhedral case of Proposition \ref{proposition: CZ and action}. Remark \ref{remark: same underlying embedded orbits} will streamline some casework:

\begin{remark}\label{remark: same underlying embedded orbits}
For $\P^*=\T^*$, $\Oc^*$, or $\I^*$, fix $c\in\text{Conj}(\P^*)$ and let $\gamma$ denote one of $\mathcal{V}$, $\mathcal{E}$, or $\mathcal{F}$. Define $S_{\gamma, c}:=\{\gamma^k\,|\,[\gamma^k]\cong c\}$ (potentially empty). Observe that the map $S_{\gamma, c}\to\Z$, $\gamma^k\mapsto\mu_{\CZ}(\gamma^k)$, is injective. Thus, if $\gamma_+\in\mathcal{P}^{L_N}(\lambda_N)$ and $\gamma_-\in\mathcal{P}^{L_M}(\lambda_M)$ project to the same orbifold critical point, and are in the same free homotopy class, then 
\begin{enumerate}[(a)]
\itemsep-.35em
    \item $\mu_{\CZ}(\gamma_+)=\mu_{\CZ}(\gamma_-)$ implies $m(\gamma_+)=m(\gamma_-)$, i.e., $\gamma_{+}\sim\gamma_-$,
    \item  $\mu_{\CZ}(\gamma_+)<\mu_{\CZ}(\gamma_-)$ implies $m(\gamma_+)<m(\gamma_-)$, and so  $\mathcal{A}(\gamma_+)<\mathcal{A}(\gamma_-)$.
\end{enumerate}
\end{remark}

\begin{lemma}\label{lemma: polyhedral cobordisms}
Suppose $[\gamma_+]=[\gamma_-]\in[S^1,S^3/\P^*]$ for $\gamma_+\in\mathcal{P}^{L_N}(\lambda_N)$ and $\gamma_-\in\mathcal{P}^{L_M}(\lambda_M)$.
\begin{enumerate}[(a)]
\itemsep-.35em
    \item If $\mu_{\CZ}(\gamma_+)=\mu_{\CZ}(\gamma_-)$, then $\gamma_+\sim\gamma_-$.
    \item If $\mu_{\CZ}(\gamma_+)<\mu_{\CZ}(\gamma_-)$, then $\mathcal{A}(\gamma_+)<\mathcal{A}(\gamma_-)$.
\end{enumerate}
\end{lemma}

\begin{proof}

We first prove (a). Recall the Conley-Zehnder index formulas \eqref{equation: CZ indices polyhedral} from Section \ref{subsection: polyhedral}: \begin{equation} \label{equation: CZ indices polyhedral repeat}
    \mu_{\CZ}(\mathcal{V}^k)=2\Bigl\lceil \frac{k}{\mathbf{I}_{\mathscr{V}}} \Bigr\rceil-1,\,\,\,\mu_{\CZ}(\mathcal{E}^k)=k,\,\,\,\mu_{\CZ}(\mathcal{F}^k)=2\Bigl\lfloor \frac{k}{3} \Bigr\rfloor+1,
\end{equation}Consider the following possible values of the homotopy class $[\gamma_{\pm}]$:
        
        \vspace{.25cm}
        
        \emph{Case 1:} $[\gamma_{\pm}]\cong T_{\text{Id}}$, $O_{\text{Id}}$, or $I_{\text{Id}}$, so that $\{\gamma_{\pm}\}\subset\{\mathcal{V}^{2\mathbf{I}_{\mathscr{V}}k_1}, \mathcal{E}^{4k_2}, \mathcal{F}^{6k_3}\,|\,k_i\in\N\}$. By reasoning identical to the analogous case of $\D_{2n}^*$ (Lemma \ref{lemma: dihedral cobordisms} (a), Case 1), $\gamma_+\sim\gamma_-$.         
        \vspace{.25cm}
        
        \emph{Case 2:} $[\gamma_{\pm}]\cong T_{-\text{Id}}$, $O_{-\text{Id}}$, or $I_{-\text{Id}}$, so that $\{\gamma_{\pm}\}\subset\{\mathcal{V}^{\mathbf{I}_{\mathscr{V}}+2\mathbf{I}_{\mathscr{V}}k_1}, \mathcal{E}^{2+4k_2}, \mathcal{F}^{3+6k_3}\,|\,k_i\in\Z_{\geq0}\}$. Again, as in the the analogous case of $\D_{2n}^*$ (Lemma \ref{lemma: dihedral cobordisms} (a), Case 2), $\gamma_+\sim\gamma_-$.        
        \vspace{.25cm}

        \emph{Case 3:} $[\gamma_{\pm}]\cong T_{6,A}, T_{6,B}, T_{3,A}$, or $T_{3,B}$. If both $\gamma_{\pm}$ are iterates of $\mathcal{V}$, then by Remark \ref{remark: same underlying embedded orbits} (a), they must share the same multiplicity, i.e. $\gamma_+\sim\gamma_-$. If both $\gamma_{\pm}$ are iterates of $\mathcal{F}$ then again by Remark \ref{remark: same underlying embedded orbits} (a), they must share the same multiplicity, i.e. $\gamma_+\sim\gamma_-$. So we must argue, using \eqref{equation: CZ indices polyhedral repeat}, that in each of these four free homotopy classes that it is impossible that $\gamma_{\pm}$ project to different orbifold points whenever $\mu_{\CZ}(\gamma_+)=\mu_{\CZ}(\gamma_-)$.
        \begin{itemize}
        \itemsep-.35em
            \item If $[\gamma_{\pm}]\cong T_{6,A}$ and $\gamma_{\pm}$ project to different orbifold points then, up to relabeling,\\ $\gamma_+=\mathcal{V}^{1+6k_+}$ and $\gamma_-=\mathcal{F}^{5+6k_-}$. Now, $\mu_{\CZ}(\gamma_+)=4k_++1\neq 4k_-+3=\mu_{\CZ}(\gamma_-)$.
            \item If $[\gamma_{\pm}]\cong T_{6,B}$ and $\gamma_{\pm}$ project to different orbifold points, write $\gamma_+=\mathcal{V}^{5+6k_+}$\\ and $\gamma_-=\mathcal{F}^{1+6k_-}$. Now, $\mu_{\CZ}(\gamma_+)=4k_++3\neq 4k_-+1=\mu_{\CZ}(\gamma_-)$.
            \item If $[\gamma_{\pm}]\cong T_{3,A}$ and $\gamma_{\pm}$ project to different orbifold points, write $\gamma_+=\mathcal{V}^{2+6k_+}$\\ and $\gamma_-=\mathcal{F}^{4+6k_-}$. Now, $\mu_{\CZ}(\gamma_+)=4k_++1\neq 4k_-+3=\mu_{\CZ}(\gamma_-)$.
            \item If $[\gamma_{\pm}]\cong T_{3,B}$ and $\gamma_{\pm}$ project to different orbifold points, write $\gamma_+=\mathcal{V}^{4+6k_+}$\\ and $\gamma_-=\mathcal{F}^{2+6k_-}$. Now, $\mu_{\CZ}(\gamma_+)=4k_++3\neq 4k_-+1=\mu_{\CZ}(\gamma_-)$.
        \end{itemize}
                
        \vspace{.25cm}
        
        \emph{Case 4:} $[\gamma_{\pm}]$ is a homotopy class not covered in Cases 1 - 3. Because every such homotopy class is represented by Reeb orbits either of type $\mathcal{V}$, of type $\mathcal{E}$, or of type $\mathcal{F}$, then we see that $\gamma_{\pm}$ project to the same orbifold point of $S^2/\P$. By Remark \ref{remark: same underlying embedded orbits} (a), we have that $\gamma_{+}\sim\gamma_-$.
        
        \vspace{.25cm}
        
        We now prove (b). Recall the action formulas \eqref{equation: action polyhedral} from Section \ref{subsection: polyhedral}: \begin{equation}\label{equation: action polyhedral repeat}
        \mathcal{A}(\mathcal{V}^k)=\frac{k\pi(1-\varepsilon)}{\mathbf{I}_{\mathscr{V}}},\,\,\,\mathcal{A}(\mathcal{E}^k)=\frac{k\pi}{2},\,\,\,\mathcal{A}(\mathcal{F}^k)=\frac{k\pi(1+\varepsilon)}{3}.
\end{equation} Consider the following possible values of the homotopy class $[\gamma_{\pm}]$:
        \vspace{.25cm}
        
        \emph{Case 1:} $[\gamma_{\pm}]\cong T_{\text{Id}}$, $O_{\text{Id}}$, or $I_{\text{Id}}$, so that $\{\gamma_{\pm}\}\subset\{\mathcal{V}^{2\mathbf{I}_{\mathscr{V}}k_1}, \mathcal{E}^{4k_2}, \mathcal{F}^{6k_3}\,|\,k_i\in\N\}$. By reasoning identical to the analogous case of $\D_{2n}^*$ (Lemma \ref{lemma: dihedral cobordisms} (b), Case 1), $\mathcal{A}(\gamma_+)<\mathcal{A}(\gamma_-)$.         
        \vspace{.25cm}
        
        \emph{Case 2:} $[\gamma_{\pm}]\cong T_{-\text{Id}}$, $O_{-\text{Id}}$, or $I_{-\text{Id}}$, so that $\{\gamma_{\pm}\}\subset\{\mathcal{V}^{\mathbf{I}_{\mathscr{V}}+2\mathbf{I}_{\mathscr{V}}k_1}, \mathcal{E}^{2+4k_2}, \mathcal{F}^{3+6k_3}\,|\,k_i\in\Z_{\geq0}\}$. Again, as in the the analogous case of $\D_{2n}^*$ (Lemma \ref{lemma: dihedral cobordisms} (b), Case 2), $\mathcal{A}(\gamma_+)<\mathcal{A}(\gamma_-)$.   
        
        \vspace{.25cm}

        \emph{Case 3:} $[\gamma_{\pm}]\cong T_{6,A}, T_{6,B}, T_{3,A}$, or $T_{3,B}$. If both $\gamma_{\pm}$ are of type $\mathcal{V}$, then by Remark \ref{remark: same underlying embedded orbits} (b), $\mathcal{A}(\gamma_+)<\mathcal{A}(\gamma_-)$. If both $\gamma_{\pm}$ are of type $\mathcal{F}$, then again by Remark \ref{remark: same underlying embedded orbits} (b), $\mathcal{A}(\gamma_+)<\mathcal{A}(\gamma_-)$. So we must argue, using \eqref{equation: CZ indices polyhedral repeat} and \eqref{equation: action polyhedral repeat}, that for each of these four free homotopy classes that if $\gamma_+$ and $\gamma_-$ are not of the same type, and if $\mu_{\CZ}(\gamma_+)<\mu_{\CZ}(\gamma_-)$, then $\mathcal{A}(\gamma_+)<\mathcal{A}(\gamma_-)$.

        \emph{3.A} If $[\gamma_{\pm}]\cong T_{6,A}$, then the pair $(\gamma_+,\gamma_-)$ is either $(\mathcal{V}^{1+6k_+}, \mathcal{F}^{5+6k_-})$ or $(\mathcal{F}^{5+6k_+}, \mathcal{V}^{1+6k_-})$. If the former holds, then $\mu_{\CZ}(\gamma_+)<\mu_{\CZ}(\gamma_-)$ implies $k_+\leq k_-$, and so \[\mathcal{A}(\gamma_+)=\frac{(1+6k_+)\pi}{3}(1-\varepsilon_N)<\frac{(5+6k_-)\pi}{3}(1+\varepsilon_M)=\mathcal{A}(\gamma_-).\] If the latter holds, then $\mu_{\CZ}(\gamma_+)<\mu_{\CZ}(\gamma_-)$ implies $k_+<k_-$, and again the action satisfies \[\mathcal{A}(\gamma_+)=\frac{(5+6k_+)\pi}{3}(1+\varepsilon_N)<\frac{(1+6k_-)\pi}{3}(1-\varepsilon_M)=\mathcal{A}(\gamma_-).\] 
        
         \emph{3.B} If $[\gamma_{\pm}]\cong T_{6,B}$, then the pair $(\gamma_+,\gamma_-)$ is either $(\mathcal{V}^{5+6k_+}, \mathcal{F}^{1+6k_-})$ or $(\mathcal{F}^{1+6k_+}, \mathcal{V}^{5+6k_-})$. If the former holds, then  $\mu_{\CZ}(\gamma_+)<\mu_{\CZ}(\gamma_-)$ implies $k_+<k_-$, and so \[\mathcal{A}(\gamma_+)=\frac{(5+6k_+)\pi}{3}(1-\varepsilon_N)<\frac{(1+6k_-)\pi}{3}(1+\varepsilon_M)=\mathcal{A}(\gamma_-).\] If the latter holds, then $\mu_{\CZ}(\gamma_+)<\mu_{\CZ}(\gamma_-)$ implies $k_+\leq k_-$, and so \[\mathcal{A}(\gamma_+)=\frac{(1+6k_+)\pi}{3}(1+\varepsilon_N)<\frac{(5+6k_-)\pi}{3}(1-\varepsilon_M)=\mathcal{A}(\gamma_-).\] 
        
          \emph{3.C}  If $[\gamma_{\pm}]\cong T_{3,A}$, then the pair $(\gamma_+,\gamma_-)$ is either $(\mathcal{V}^{2+6k_+}, \mathcal{F}^{4+6k_-})$ or $(\mathcal{F}^{4+6k_+}, \mathcal{V}^{2+6k_-})$. If the former holds, then $\mu_{\CZ}(\gamma_+)<\mu_{\CZ}(\gamma_-)$ implies $k_+\leq k_-$, and so \[\mathcal{A}(\gamma_+)=\frac{(2+6k_+)\pi}{3}(1-\varepsilon_N)<\frac{(4+6k_-)\pi}{3}(1+\varepsilon_M)=\mathcal{A}(\gamma_-).\] If the latter holds, then $\mu_{\CZ}(\gamma_+)<\mu_{\CZ}(\gamma_-)$ implies $k_+<k_-$, and so \[\mathcal{A}(\gamma_+)=\frac{(4+6k_+)\pi}{3}(1+\varepsilon_N)<\frac{(2+6k_-)\pi}{3}(1-\varepsilon_M)=\mathcal{A}(\gamma_-).\] 
        
        \emph{3.D} If $[\gamma_{\pm}]\cong T_{3,B}$, then the pair $(\gamma_+,\gamma_-)$ is either $(\mathcal{V}^{4+6k_+}, \mathcal{F}^{2+6k_-})$ or $(\mathcal{F}^{2+6k_+}, \mathcal{V}^{4+6k_-})$. If the former holds, then $\mu_{\CZ}(\gamma_+)<\mu_{\CZ}(\gamma_-)$ implies $k_+<k_-$, and so \[\mathcal{A}(\gamma_+)=\frac{(4+6k_+)\pi}{3}(1-\varepsilon_N)<\frac{(2+6k_-)\pi}{3}(1+\varepsilon_M)=\mathcal{A}(\gamma_-).\] If the latter holds, then $\mu_{\CZ}(\gamma_+)<\mu_{\CZ}(\gamma_-)$ implies $k_+\leq k_-$, and so \[\mathcal{A}(\gamma_+)=\frac{(2+6k_+)\pi}{3}(1+\varepsilon_N)<\frac{(4+6k_-)\pi}{3}(1-\varepsilon_M)=\mathcal{A}(\gamma_-).\] 
                
        \vspace{.25cm}
        
        \emph{Case 4:} $[\gamma_{\pm}]$ is a homotopy class not covered in Cases 1 - 3. Because every such homotopy class is represented by Reeb orbits either of type $\mathcal{V}$, of type $\mathcal{E}$, or of type $\mathcal{F}$, then we see that $\gamma_{\pm}$ project to the same orbifold point of $S^2/\P$. By Remark \ref{remark: same underlying embedded orbits} (b),  $\mathcal{A}(\gamma_{+})<\mathcal{A}(\gamma_{-})$.
      
\end{proof}

\appendix
\section{Appendix}

\subsection{Signature of crossing forms} \label{appendix: signature of crossing forms}
Proposition \ref{proposition: equivalence of cz definitions} establishes the equivalence of definitions \eqref{equation: cz reeb orbit definition using spectral flow} and \eqref{equation: def of cz of reeb wrt sp matrices} of the Conley-Zehnder index, and relies on the fact that \[\Gamma(\Psi,s)(F_s(\eta))=-\Gamma(A,s)(\eta),\] which we prove now. As in the proposition (and in Figure \ref{figure: crossings}), 
\begin{itemize}
    \item $A=\{A_s\}_{s\in[-1,1]}$ is a path of asymptotic operators from $A_{+}$ to $A_-$ of the form $A_s=-J_0\partial_t-S_s$, where $S_s:S^1\to\text{Sym}(2n)$ is a smooth path of loops of symmetric matrices; $S:[-1,1]\times S^1\to \text{Sym}(2n)$ is $S(s,t):=S_s(t)$.
    \item For each $s\in[-1,1]$, $\Phi_s:[0,1]\to\Sp(2n)$ is the path of symplectic matrices corresponding to the loop $S_s$. That is, $\Phi_s$ is the solution to the initial value problem $\dot{\Phi}_s=J_0S_s\Phi_s$, $\Phi_s(0)=\text{Id}$; $\Phi:[-1,1]\times[0,1]\to\Sp(2n)$ is $(s,t)\mapsto \Phi_s(t)$.
    \item $\Psi:[-1,1]\to\Sp(2n)$ is the path of symplectic matrices from $\Phi_+(1)$ to $\Phi_-(1)$ given by $\Psi(s)=\Phi(s,1)=\Phi_s(1)$.
\end{itemize}
Recall that for any $s\in[-1,1]$, we have an isomorphism $F_s:\text{ker}(A_s)\to\text{ker}(\Psi(s)-\text{Id})$ given by $\eta\mapsto\eta(0)$. Additionally, recall from Sections \ref{subsection: asymptotic operators and spectral flows} and \ref{subsection: the cz index} the quadratic forms \[\Gamma(A,s):\text{ker}(A_s)\to\R,\,\,\, \eta\mapsto\langle\tfrac{d A_s}{ds}\eta,\eta\rangle_{L^2}\] and \[\Gamma(\Psi,s):\text{ker}(\Psi(s)-\text{Id})\to\R, \,\,\,v\mapsto\omega_0(v, \tfrac{d\Psi}{ds}v).\]

\begin{lemma}\label{lemma: crossing forms are related by a negative sign}
For $s\in[-1,1]$ and $\eta\in\mbox{\em ker}(A_s)$, 
\[\Gamma(A,s)(\eta)=-\Gamma(\Psi,s)(F_s(\eta)).\]
\end{lemma}
\begin{proof}
We define one more family of matrices. Let $\widehat{S}:[-1,1]\times[0,1]\to\text{Sym}(2n)$ be the family of symmetric matrices given by
\[\widehat{S}(s,t)=-J_0\tfrac{\partial \Phi}{\partial s}(s,t)\Phi^{-1}(s,t),\] which satisfies the equation \[\tfrac{\partial \Phi}{\partial s}(s,t)=J_0\widehat{S}(s,t)\Phi(s,t).\] Note  that $\widehat{S}(s,0)=0$, because $\Phi(s,0)=\text{Id}$. We argue that \[\tfrac{\partial}{\partial t}\big(\Phi^T\widehat{S}\Phi\big)=\Phi^T\tfrac{\partial S}{\partial s}\Phi.\] We compute,
\begin{align}
    \tfrac{\partial}{\partial t}\big(\Phi^T\widehat{S}\Phi\big)&=\big(\tfrac{\partial \Phi}{\partial t}\big)^T\widehat{S}\Phi+\Phi^T\tfrac{\partial}{\partial t}\big(\widehat{S} \Phi\big)\nonumber\\
    &=\big(J_0S\Phi)^T\widehat{S}\Phi+\Phi^T\tfrac{\partial}{\partial t}\big(-J_0\tfrac{\partial M}{\partial s}\big)\nonumber\\
    &=-\Phi^TSJ_0\widehat{S}\Phi-\Phi^TJ_0\tfrac{\partial^2 \Phi}{\partial t \partial s}\nonumber\\
    &=-\Phi^T\big(SJ_0\widehat{S}\Phi+J_0\tfrac{\partial^2 \Phi}{\partial t \partial s}\big) \nonumber\\
    &=-\Phi^T\big(S\tfrac{\partial \Phi}{\partial s}+J_0\tfrac{\partial}{\partial s}\big(J_0S\Phi\big)\big)\nonumber\\
    &=-\Phi^T\big(S\tfrac{\partial \Phi}{\partial s}-\tfrac{\partial}{\partial s}\big(S\Phi\big)\big)\nonumber\\
    &=-\Phi^T\big(S\tfrac{\partial \Phi}{\partial s}-\tfrac{\partial S}{\partial s}\Phi-S\tfrac{\partial \Phi}{\partial s}\big)\nonumber\\
    &=\Phi^T\tfrac{\partial S}{\partial s}\Phi. \label{equation: differential matrix equality}
\end{align}
We now demonstrate that $F_s$ negates the quadratic forms:
 \begin{align*}
     \Gamma(A,s)(\eta)&=\langle \eta,\tfrac{d}{ds}(A_s\eta)\rangle_{L^2} \\
     &=\int_0^1\langle \eta(t),\tfrac{d}{ds}(A_s\eta)(t)\rangle\,dt\\
     &=-\int_0^1\langle \eta(t),\tfrac{\partial S}{\partial s}(s,t)\eta(t)\rangle\,dt\\
     &=-\int_0^1\langle \Phi_s(t)\eta_0,\tfrac{\partial S}{\partial s}(s,t)\Phi_s(t)\eta_0\rangle\,dt,\end{align*}
where we have used in the fourth line that $\eta(t)=\Phi_s(t)\eta_0$, because $\eta\in\text{ker}(A_s)$, where $\eta_0:=\eta(0)\in\text{ker}(\Psi(s)-\text{Id})$. Applying $\langle\cdot\,,\cdot\rangle=\omega_0(
 \cdot,J_0\cdot)$, we have  
     \begin{align*}\Gamma(A,s)(\eta)&=-\int_0^1\omega_0(\Phi_s(t)\eta_0,J_0\tfrac{\partial S}{\partial s}(s,t)\Phi_s(t)\eta_0)\,dt\\
          &=-\int_0^1\omega_0(\Phi_s(t)\eta_0,\Phi_s(t)J_0\Phi_s^T(t)\tfrac{\partial S}{\partial s}(s,t)\Phi_s(t)\eta_0)\,dt\\
          &=-\int_0^1\omega_0(\eta_0,J_0\Phi_s^T(t)\tfrac{\partial S}{\partial s}(s,t)\Phi_s(t)\eta_0)\,dt\\
        &=-\int_0^1\langle\eta_0,\Phi_s^T(t)\tfrac{\partial S}{\partial s}(s,t)\Phi_s(t)\eta_0\rangle\,dt,\end{align*}
        where we have rewritten $J_0$ as $\Phi_s^T(t)J_0\Phi_s(t)$ in the second line, and that $\Phi_s(t)$ preserves the values of $\omega_0$ in the third line. Next, we apply the fact that $\Phi_s^T(t)\tfrac{\partial S}{\partial s}(s,t)\Phi_s(t)=\tfrac{\partial}{\partial t}\big(\Phi_s^T(t)\widehat{S}(s,t)\Phi_s(t)\big)$ (equation \eqref{equation: differential matrix equality}) to find that
        \begin{align*}
        \Gamma(A,s)(\eta)&=-\int_0^1\langle\eta_0,\tfrac{\partial}{\partial t}\big(\Phi_s^T(t)\widehat{S}(s,t)\Phi_s(t)\big)\eta_0\rangle\,dt\\
        &=-\int_0^1\tfrac{d}{dt}\langle\eta_0,\Phi_s^T(t)\widehat{S}(s,t)\Phi_s(t)\eta_0\rangle\,dt\\
        &=-\langle \eta_0, \Phi_s^T(1)\widehat{S}(s,1)\Phi_s(1)\eta_0-\widehat{S}(s,0)\eta_0\rangle,
        \end{align*}
        where we have used the fundamental theorem of calculus. Recall that $\widehat{S}(s,0)=0$ and that $\Phi_s(1)\eta_0=\eta_0$, because $\eta_0$ is a 1-eigenvector of $\Psi(s)=\Phi_s(1)$, to write
        \begin{align*}
        \Gamma(A,s)(\eta)&=-\langle \eta_0, \Phi_s^T(1)\widehat{S}(s,1)\eta_0\rangle\\
        &=-\omega_0(\eta_0,J_0\Phi_s^T(1)\widehat{S}(s,1)\eta_0)\\
        &=-\omega_0(\Phi_s(1)\eta_0,\Phi_s(1)J_0\Phi_s^T(1)\widehat{S}(s,1)\eta_0)\\
        &=-\omega_0(\eta_0,\Phi_s(1)J_0\Phi_s^T(1)\widehat{S}(s,1)\eta_0)\\
        &=-\omega_0(\eta_0,J_0\hat{S}(s,1)\eta_0)\\
        &=-\omega_0(\eta_0,\tfrac{d\Psi}{ds}(s)\eta_0)=-\Gamma(\Psi, s)(\eta_0)=-\Gamma(\Psi, s)(F_s(\eta)).
 \end{align*}
 
\end{proof}

\subsection{When contractible Reeb orbits have index 3 and fail to be embedded}\label{appendix: CZindex3}

A main result of \cite{HN} allows us to conclude that, if $\lambda$ is an $L$-dynamically convex contact form on a closed 3-manifold $Y$, satisfying
\begin{center}
    $(*)$ \emph{every contractible Reeb orbit} $\gamma\in\mathcal{P}^L(\lambda)$ \emph{with} $\mu_{CZ}(\gamma)=3$ \emph{must be embedded},
\end{center}
then a generic $\lambda$-compatible $J$ on $\R\times Y$ produces a well defined differential $\partial^L$ on the $L$-filtered complex $CC_*^L(Y,\lambda,J)$ that squares to zero. In our cases, $\gamma_{\mathfrak{s}}^n$, $e_-^4$, and $\mathcal{V}^{2X}$ are the sole exceptions to $(*)$ in the cyclic, dihedral, and polyhedral cases respectively.

The condition $(*)$ prohibits a certain type of building appearing as a boundary component of the compactified moduli spaces $\overline{\mathcal{M}_2^J(\gamma_+, \gamma_-)/\R}$, which one expects to be a smooth, oriented 1-manifold with boundary, generically. Once it is shown that these \emph{bad buildings} do not appear, one concludes that all ends of the compactified space consist of buildings $(u_+,u_-)$ of index 1 cylinders $u_{\pm}$. By appealing to appropriately weighted counts of the boundary components of $\overline{\mathcal{M}_2^J(\gamma_+, \gamma_-)/\R}$, one finds that  $\partial^2=0$.

\begin{definition}\label{definition: building}
Let $(Y,\lambda)$ be a contact 3-manifold and take  $J\in\mathcal{J}(\lambda)$. An index 2, genus 0 building $B=(u_1,u_2)$ in $\R\times Y$ is \emph{bad} if for some embedded $\gamma\in\mathcal{P}(\lambda)$ and $d_1, d_2\in\N$:
\begin{itemize}
    \itemsep-.35em
    \item $u_1$ is a holomorphic index 0 pair of pants, a branched cover of the trivial cylinder $\R\times\gamma^{d_1+d_2}$ with positive end asymptotic to $\gamma^{d_1+d_2}$ and with two negative ends, one asymptotic to $\gamma^{d_1}$, and one to $\gamma^{d_2}$,
    \item  $u_2$ is a union of the trivial cylinder $\R\times\gamma^{d_1}$ and $v$, an index 2 holomorphic plane with one positive puncture asymptotic to $\gamma^{d_2}$.
\end{itemize}
\end{definition}

\begin{lemma} \label{lemma: generic j}
Fix $L>0$ and let $\varepsilon>0$ be sufficiently small so that elements of $\mathcal{P}^L(\lambda_{G,\varepsilon})$ project to critical points of the orbifold Morse function $f_H$ under $\fp$ and are nondegenerate, where $\lambda_{G,\varepsilon}=(1+\varepsilon\fp^*f_H) \lambda_G$, and $f_H$ is the corresponding orbifold Morse function used in Section \ref{section: computation of filtered contact homology}. Then, for $\gamma_{\pm}\in\mathcal{P}^L(\lambda_{G,\varepsilon})$ and  $J\in\mathcal{J}(\lambda_{G,\varepsilon})$, the compactified moduli space $\overline{\mathcal{M}^J_2(\gamma_+,\gamma_-)/\R}$ contains no bad buildings.
\end{lemma}

Note that Lemma \ref{lemma: generic j} is not necessary to show that $(\partial^L)^2=0$ for our chain complexes introduced in Section \ref{section: computation of filtered contact homology}; this follows immediately from $\partial^L=0$. However, slightly modified index computations of our proof may be used to rule out bad buildings if one were to use different orbifold Morse functions that produce non-vanishing filtered differentials.

\begin{proof}

We argue more generally that a bad building cannot exist if its positive Reeb orbit has action less than $L$.  Indeed, suppose we have bad building $B=(u_1,u_2)$, with $\gamma$, $d_1$, $d_2$, and $v$ as in Definition \ref{definition: building}, and suppose $\mathcal{A}(\gamma^{d_1+d_2})<L$. Consider that the existence of the holomorphic plane $v$ implies that $\gamma^{d_2}$ is contractible, and that $\mu_{\CZ}(\gamma^{d_2})=3$. Because $\mathcal{A}(\gamma^{d_2})<\mathcal{A}(\gamma^{d_1+d_2})<L$, we know $\gamma^{d_2}$ projects to an orbifold critical point of $f_H$ under $\fp$. By our characterization of these Reeb orbits in Sections \ref{subsection: cyclic}, \ref{subsection: dihedral}, and \ref{subsection: polyhedral} we must be in one of the following cases, depending on $G$:
\begin{enumerate}
 \itemsep-.35em
    \item $(G,\gamma,\, d_2)=(\langle g\rangle,\, \gamma_{\mathfrak{s}},\, n)$ (where $\langle g\rangle$ is cyclic of order $n$),
        \item $(G,\,\gamma,\, d_2)=(\D_{2n}^*,\, e_-,\, 4)$, 
            \item $(G,\,\gamma,\, d_2)=(\P^*, \,\mathcal{V},\, 2\mathbf{I}_{\mathscr{V}})$.
\end{enumerate}
In each case, our index computations \eqref{equation: CZ indices cyclic}, \eqref{equation: CZ indices dihedral}, and \eqref{equation: CZ indices polyhedral} provide that $\mu_{\CZ}(\gamma^k)=2\lceil \frac{2k}{d_2}\rceil-1$. Let us compute $\text{ind}(u_1)$:
\begin{align*}
    \text{ind}(u_1)&= 1+\bigg(2\biggl\lceil \frac{2(d_1+d_2)}{d_2}\biggr\rceil-1\bigg) -\bigg(2\biggl\lceil \frac{2d_1}{d_2}\biggr\rceil-1\bigg) -3\\
    &=2\bigg(\biggl\lceil \frac{2(d_1+d_2)}{d_2}\biggr\rceil - \biggl\lceil \frac{2d_1}{d_2}\biggr\rceil-1\bigg)\\
    &=2(2-1)=2\neq0.
\end{align*}
This contradiction shows that such a $B$ does not exist in the symplectization of $(S^3/G, \lambda_{G,\varepsilon})$.
\end{proof}


\end{document}